\documentclass[11pt, a4paper]{amsart}
\usepackage[utf8]{inputenc}
\usepackage[margin=0.9in]{geometry}

\usepackage{comment}
\usepackage{lmodern}
\usepackage{upgreek}
\usepackage{color}
\usepackage{amsmath}
\usepackage{amssymb}
\usepackage{amsthm}

\usepackage{pdfpages}

\usepackage{fancyhdr}

\usepackage{setspace}

\usepackage{hyperref}
\hypersetup{
	colorlinks=true,
	linkcolor=blue,
	filecolor=blue,      
	urlcolor=blue,
	citecolor=blue,
	pdftitle={Overleaf Example},
	pdfpagemode=FullScreen,
}

\numberwithin{equation}{section}
\newtheorem{theorem}{Theorem}[section]
\newtheorem{lemma}[theorem]{Lemma}

\newtheorem{cor}[theorem]{Corollary}
\newtheorem{prop}[theorem]{Proposition}

\newtheorem{con}[theorem]{Conjecture}

\newtheorem{rem}{Remark}[section]

\theoremstyle{definition}

\newtheorem*{prob*}{Problem}

\title{Almost Prime Orders of Elliptic Curves Over Prime Power Fields}

\author{Likun Xie}
	\address{Department of Mathematics, University of Illinois, 1409 West Green
	Street, Urbana, IL 61801, USA} 
\email{likunx2@illinois.edu}
\subjclass[2020]{Primary: 11G05; Secondary: 11N36}
\keywords{almost primes, elliptic curves,  vector sieves}

\begin{document}

\begin{abstract}

	In 1988, Koblitz conjectured the infinitude of primes \( p \) for which \( |E(\mathbb{F}_p)| \) is prime  for elliptic curves \( E \) over \( \mathbb{Q} \), drawing an analogy with the twin prime conjecture. He also proposed studying the primality of \( |E(\mathbb{F}_{p^\ell})| / |E(\mathbb{F}_p)| \), in parallel with the primality of \( (p^\ell - 1)/(p - 1) \).
	
Motivated by these problems and earlier work on \( |E(\mathbb{F}_p)| \), we study  the infinitude of primes \( p \) such that \( |E(\mathbb{F}_{p^\ell})| / |E(\mathbb{F}_p)| \) has a bounded number of prime factors for primes \(\ell\geq 2 \), considering both CM and non-CM elliptic curves over \( \mathbb{Q} \). In the CM case, we focus on the curve \( y^2 = x^3 - x \) to address gaps in the literature and present a more concrete argument. The result is unconditional and applies Huxley’s large sieve inequality for the associated CM field. In the non-CM case, analogous results follow under GRH via the effective Chebotarev density theorem.
	
	For the CM curve \( y^2 = x^3 - x \), we further apply a vector sieve to combine the almost prime properties of \( |E(\mathbb{F}_p)| \) and \( |E(\mathbb{F}_{p^2})| / |E(\mathbb{F}_p)| \), establishing a lower bound for the number of primes \( p \leq x \) for which \( |E(\mathbb{F}_{p^2})| / 32 \) is a square-free almost prime. We also study cyclic subgroups of finite index in \( E(\mathbb{F}_p) \) and \( E(\mathbb{F}_{p^2}) \) for CM curves.
\end{abstract}

	\maketitle	
	
	\section{Introduction}
	\subsection{Background and Motivation}

	Let \( E / \mathbb{Q} \) be an elliptic curve of conductor \( N_E \), and let \( p \) be a prime of good reduction (i.e., \( p \nmid N_E \)). The group structure of the group \( E(\mathbb{F}_q) \) of \(\mathbb{F}_q\)-rational points of \(E\), where \( q = p^k \), is of particular interest in cryptographic applications, especially when the order of the group is divisible by a large prime \cite{book_cry}.  
	
	In 1988, motivated by cryptographic applications and the classical twin prime problem \cite{twin_prime}, Koblitz made several conjectures regarding the primality of the order of elliptic curves over finite fields \cite{conjecture}.  
	
	First, we consider 	  \( E / \mathbb{Q} \)   an elliptic curve with complex multiplication (CM) by  the ring of integer  $\mathcal{O}_K $ of an imaginary quadratic field \( K \).
	It is known that \( E(\mathbb{Q})_{\mathrm{tors}} \) injects into \( E(\mathbb{F}_p) \) for all primes \( p\neq 2,3 \) of good reduction.  
Define  
\begin{equation}\label{def_d_E}
	d_E := \gcd\left\{ |E(\mathbb{F}_p)| \;\middle|\; p \text{ splits in }  K,\, p\nmid N_E  \right\}.
\end{equation}
Since \( |E(\mathbb{F}_p)| = N(\pi_p - 1) \) for some $\pi_p\in \mathcal{O}_K$, the value of \( d_E \) is determined by the divisibility properties of \( \pi_p - 1 \) in \( \mathcal{O}_K \).  In particular, if \( E/\mathbb{Q} \) is not \(\mathbb{Q}\)-isogenous to a curve with nontrivial \(\mathbb{Q}\)-rational torsion, then \( d_E = 1 \).

Koblitz proposed the following conjecture, reformulated from \cite[p.~164]{conjecture} to incorporate the common divisors \( d_E \):
	\begin{con}\label{conjecture1}
		Let \( E / \mathbb{Q} \) be an elliptic curve with complex multiplication by the ring of integer $\mathcal{O}_K$  of   an imaginary quadratic field \( K \). There exists a constant \( C > 0 \) such that 
		\[
	\#\left\{ 
		p \leq x\,\middle|\,  p \text{ splits in }  K, \, \frac{\left|E(\mathbb{F}_p)\right|}{d_E} \text{ is prime}
		\right\} 
		\sim
		C \frac{x}{ (\log x)^2}.
		\]
	\end{con}
A similar conjecture was also made for non-CM curves in~\cite{conjecture}, and in this case, almost all such curves satisfy \( d_E = 1 \).
 To study the distribution of almost prime values of \( |E(\mathbb{F}_p)| \), one can treat the condition \( \ell^n \mid |E(\mathbb{F}_p)| \) as a Chebotarev condition and apply the effective Chebotarev density theorem, as in Steuding and Weng \cite{steuding}. This method applies to both CM and non-CM cases, assuming the Generalized Riemann Hypothesis (GRH).
	
In the CM case, Cojocaru \cite{cojocaru} employed a more explicit method, using a Bombieri–Vinogradov type theorem for CM fields to avoid reliance on effective Chebotarev density theorem, yielding an unconditional result.

The best known unconditional result toward Conjecture~\ref{conjecture1}, also due to Cojocaru, is the following when \( d_E = 1 \).	
\begin{theorem}[{\cite[Theorem 4]{cojocaru}}]\label{corojacu}
	Let \( E /\mathbb{Q} \) be a CM elliptic curve such that all curves \( \mathbb{Q} \)-isogenous to \( E \) have trivial \( \mathbb{Q} \)-torsion. Let \( N_E \) be the conductor of \( E \). Then there exists a constant \( C(E) > 0 \) such that as \( x \to \infty \),
	\[
	\#\{ p \leq x \,\mid\, p \nmid N_E,\; |E(\mathbb{F}_p)| \text{ has at most 5 prime factors} \} \geq C(E) \frac{x}{(\log x)^2}.
	\]
\end{theorem}

In \cite{IU_elliptic}, a \( P_2 \) result was claimed for Conjecture~\ref{conjecture1} for the curve \( E: y^2 = x^3 - x \):
\[
\# \left\{ 
p \leq x \,\middle|\, p \equiv 1 \pmod{4}, \; \frac{\left| E(\mathbb{F}_p) \right|}{8} = P_2
\right\}  
\gg \frac{x}{(\log x)^2},
\]
where \( P_2 \) denotes a product of at most two distinct primes. This was later extended in \cite{almost_prime_order} to all CM curves.

However, both proofs in \cite{IU_elliptic} and \cite{almost_prime_order} appears to contain a gap in the level of distribution estimate: the claimed level \( x^{1/2 - \epsilon} \) does not hold, see Remark \ref{mistake}. A careful analysis suggests the correct level is \( x^{1/4 - \epsilon} \), which is insufficient for a \( P_2 \) result using Chen's method~\cite{chen}, but does suffice for a \( P_5 \) result via a weighted sieve, similar to Cojocaru~\cite{cojocaru}.  Both~\cite{cojocaru} and~\cite{IU_elliptic} ultimately rely on the same Bombieri--Vinogradov-type estimate for CM fields (see Section~\ref{section_bombierri_number_fields}), leading to the same level of distribution.

\begin{comment}
		
	Assuming the Generalized Riemann Hypothesis (GRH) concerning the nonvanishing of Dedekind zeta-functions \( \zeta_K(s) \) of number fields \( K \) for \( \mathrm{Re}(s) > 1/2 \) (or a more relaxed version—namely, the \( \theta \)-quasi-GRH on p.~8), a stronger level of distribution   \( x^{1/3 - \epsilon} \) for the CM case, can be attained via the effective Chebotarev density theorem. This leads to a \( P_4 \) result, as established by Steuding and Weng in \cite[Theorem 1(ii)]{steuding}; see also \cite{erratum}.\footnote{In \cite[Theorem 1(ii)]{steuding}, the condition \( \Omega(N_p) \leq 3 \) was corrected to \( \Omega(N_p) \leq 4 \) in the erratum \cite{erratum}.}
\end{comment}

Previous work on almost prime orders of elliptic curves has largely focused on \( |E(\mathbb{F}_p)| \). A natural extension is to consider reductions modulo higher prime powers. In a remark on \cite[p.~164]{conjecture}, Koblitz posed the following question:	\begin{prob*}\label{prob}
	For which primes \( p \) is the quantity
	\[
	\frac{|E(\mathbb{F}_{p^\ell})|}{|E(\mathbb{F}_p)|} = N\left(\frac{\pi^\ell - 1}{\pi - 1}\right)
	\]
	prime for some prime \(\ell\)? Here, \( N(z) = z\cdot \overline{z} \) denotes the complex norm.
\end{prob*}

	Let \( p \) be a prime of good reduction for an elliptic curve \( E/\mathbb{Q} \). Then for \( q = p^r \), the number of \( \mathbb{F}_q \)-rational points is given by \cite[Theorem V.2.3.1]{silverman}:
	\begin{equation}\label{pi}
		|E(\mathbb{F}_q)| = q + 1 - \pi_p^r - \overline{\pi}_p^r = N(\pi_p^r - 1),
	\end{equation}
where \( \pi_p, \overline{\pi}_p \in \mathbb{C} \) are the roots of \( T^2 - a_p T + p \), satisfying \( |\pi_p| = \sqrt{p} \), and \( \pi_p \) lies in the ring of integers of \( \mathbb{Q}(\sqrt{a_p^2 - 4p}) \).

In the CM case, primes of good reductions are classified as follows: a prime \( p \) is called \emph{supersingular} if it is inert or ramified in \( \mathcal{O}_K \), and \emph{ordinary} if it splits in \( \mathcal{O}_K \).
	For supersingular \( p \), \( |E(\mathbb{F}_q)| \) is a polynomial in \( p \), and classical sieve methods apply to study its almost prime values. Thus, we focus on the more subtle case when \( p \) splits.

In particular, for the elliptic curve \( E_n : y^2 = x^3 - n^2x \) and primes \( p \nmid 2n \), it follows from \cite[Theorem II.2, p.~59]{book_elliptic_modular} that
\[
\pi_p = i\sqrt{p}, \quad \text{if } p \equiv 3 \pmod{4};
\]
and if \( p \equiv 1 \pmod{4} \), one may choose \( \pi_p \in \mathbb{Z}[i] \) such that \( N(\pi_p) = p \) and
\[
\pi_p \equiv \left( \frac{n}{p} \right) \pmod{2 + 2i}.
\]
For a broader discussion of the CM case, see \cite{pi_1, pi_2}.

\subsection{Main Results}
We study the almost prime property of \( |E(\mathbb{F}_{p^\ell})| \) for prime powers \( p^\ell \), where \( \ell \geq 2 \) is a fixed prime.
 We apply sieve methods to sequences defined by norms, where divisibility conditions of norms can be related to divisibility by ideals in \( \mathcal{O}_K \) in the CM case, and to Chebotarev  conditions in the non-CM setting.
\begin{comment}
	  Beyond its connection to Koblitz's problems, we hope that the techniques developed for analyzing \( |E(\mathbb{F}_{p^l})| \) will have further applications to other structural properties of elliptic curves.
\end{comment}

\begin{comment}
For example, to study the almost prime behavior of
\(\frac{|E(\mathbb{F}_{p^l})|}{|E(\mathbb{F}_p)|}\),
we consider the sequence
\begin{equation}\label{intro_sequence}
\mathcal{A}(x) := \left\{ N\left( \frac{\pi_p^l - 1}{\pi_p - 1} \right) \,\middle|\, N(\pi_p) = p \leq x,\; \pi_p^2 - a_p \pi_p + p = 0 \right\},
\end{equation}
where \( a_p \) is the trace of the Frobenius at \( p \).
A key feature of the CM setting is that all elements \( \pi_p \), associated with ordinary primes \( p \), lie in the same number ring—namely, the ring of integers \( \mathcal{O}_K \) of the CM field \( K \). This reduces the norm divisibility problem to one for elements in \( \mathcal{O}_K \). Moreover, the error terms can be controlled using a Bombieri--Vinogradov-type theorem for number fields.
For the curve \( E : y^2 = x^3 - x \), the sequence \eqref{intro_sequence} simplifies to
\[
\mathcal{A}(x) := \left\{ N\left( \frac{\pi^l - 1}{\pi - 1} \right) \,\middle|\, N(\pi) = p \leq x,\, \pi \in \mathbb{Z}[i],\, \pi \equiv 1 \bmod 2(1 + i) \right\},
\]
and the problem reduces to studying primes and divisibility conditions in \( \mathbb{Z}[i] \).

\end{comment}

In the CM case, we focus on the specific curve \( E : y^2 = x^3 - x \), chosen both to address issues in the earlier \( P_2 \) result in~\cite{IU_elliptic} and to allow for a more concrete argument within the Gaussian domain \( \mathbb{Z}[i] \). The general CM case can be treated similarly, with adjustments depending on the CM field \( K \) and the associated integer \( d_E \).

For a prime \(\ell\), studying the almost primality of  
\(\frac{|E(\mathbb{F}_{p^\ell})|}{|E(\mathbb{F}_p)|}\)
leads to a linear sieve problem (see Section~\ref{discussion}). Using a weighted linear sieve, we prove the following:

\begin{theorem}\label{prop_intro}
	Let \( E/\mathbb{Q} \) be the elliptic curve \( y^2 = x^3 - x \). For any odd prime \(\ell\geq 3 \), define  
	\[
	n_\ell = \lfloor 4.1(\ell - 1) + 1.2 \rfloor.
	\]
	Then, for sufficiently large \( x \),
	\[
	\# \left\{ p \leq x \,\middle|\, p \equiv 1 \pmod{4},\, \Omega\left( \frac{|E(\mathbb{F}_{p^\ell})|}{|E(\mathbb{F}_p)|} \right) \leq n_\ell \right\} \gg \frac{x}{(\log x)^2}.
	\]	
	When \(\ell= 2 \), for sufficiently large \( x \),
	\[
	\# \left\{ p \leq x \,\middle|\, p \equiv 1 \pmod{4},\, \frac{|E(\mathbb{F}_{p^2})|}{4|E(\mathbb{F}_p)|} = P_5 \right\} \gg \frac{x}{(\log x)^2},
	\]
	where \( P_5 \) denotes a product of  at most 5 distinct  primes.
\end{theorem}

To study the almost prime property of \( |E(\mathbb{F}_{p^2})| \), we use a vector sieve to count splitting primes \( p \leq x \) for which both  
\[
\frac{|E(\mathbb{F}_p)|}{8} \quad \text{and} \quad \frac{|E(\mathbb{F}_{p^2})|}{4|E(\mathbb{F}_p)|}
\]
are almost primes. The contribution from primes where these two quantities share a common factor is negligible, yielding a count of primes \( p \leq x \) for which \( |E(\mathbb{F}_{p^2})|/32 \) is a square-free almost prime.

Although this naturally leads to a two-dimensional sieve problem, we instead apply a vector sieve combining two linear sieves. This simplifies the computation, as linear sieve bounds are more explicit and better understood than those in higher dimensions.

	\begin{comment}
	If we fix \( d_E  \), say $d_E=1$, a general argument can be applied to treat elliptic curves with \( d_E = 1 \), as in the case of Conjecture \ref{conjecture1} and Theorem \ref{corojacu}. This approach can also be adapted to study the number of splitting primes \( p \leq x \) such that both  
	\[
	\frac{|E(\mathbb{F}_p)|}{8} \quad \text{and} \quad \frac{|E(\mathbb{F}_{p^l})|}{|E(\mathbb{F}_p)|}
	\]
	are almost primes, where \(\ell> 2 \) is a prime; see Section \ref{discussion}.
	\end{comment}

	\begin{theorem}\label{theorem1_intro}
			Let \( E/\mathbb{Q} \) be the elliptic curve defined by \( y^2 = x^3 - x \). Then, for sufficiently large \( x \), we have
		\[
		\left\{ 
		p \leq x \;\middle|\; 
		p \equiv 1 \pmod{4},\ 
		\frac{|E(\mathbb{F}_p)|}{8} = P_k(x^{1/30}),\ 
		\frac{|E(\mathbb{F}_{p^2})|}{4|E(\mathbb{F}_p)|} = P_{10-k}(x^{1/31}),\ 
		1 \leq k \leq 9 
		\right\}
		\gg \frac{x}{(\log x)^3},
		\]
		where \( P_m(z) \) denotes a product of at most \( m \) distinct primes, each   \( >z \).
		
		Moreover, one can additionally require that
		\[
	 \left( 
		\frac{|E(\mathbb{F}_p)|}{8},\ 
		\frac{|E(\mathbb{F}_{p^2})|}{4|E(\mathbb{F}_p)|} 
		\right) = 1,
		\]
		in which case we obtain
	\[	\left\{ 
	p \leq x \;\middle|\; 
	p \equiv 1 \pmod{4},\ 
	\frac{|E(\mathbb{F}_{p^2})|}{32} = P_{10} 
	\right\}
	\gg \frac{x}{(\log x)^3},\]
		where \( P_{10} \) denotes a product of at most 10 distinct primes.
	\end{theorem}

It is possible to extend Theorem~\ref{theorem1_intro}  to any prime \(\ell\geq 3 \), yielding a lower bound for primes \( p \leq x \) such that \( |E(\mathbb{F}_p)|/8 \) and \( |E(\mathbb{F}_{p^\ell})| / |E(\mathbb{F}_p)| \) are both almost primes. However, it may not exclude cases where \( |E(\mathbb{F}_{p^\ell})| / |E(\mathbb{F}_p)| \) is not square-free.

Theorem~\ref{theorem1_intro} implies that for each counted prime \( p \), the group \( E(\mathbb{F}_p) \) contains a cyclic subgroup of order \( P_k \) and index 8, while \( E(\mathbb{F}_{p^2}) \) contains a cyclic subgroup of order \( P_{10-k} \) and index 32. To refine the structure, we consider the injections
\[
E(\mathbb{Q}(\zeta_{p^r - 1}))[2^\infty] \hookrightarrow E(\mathbb{F}_{p^r}),
\]
which yield the following result.
\begin{theorem}\label{theorem2_intro}
	Let \( E/\mathbb{Q} \) be the elliptic curve \( y^2 = x^3 - x \). Then, for sufficiently large \( x \),
\[
\# \left\{ 
p \leq x :\,
E(\mathbb{F}_p) \cong \mathbb{Z}/2  \times \mathbb{Z}/4m ,\ 
E(\mathbb{F}_{p^2}) \cong \mathbb{Z}/4  \times \mathbb{Z}/8mn ,\ \mu (mn) \neq 0 ,
\ \omega(m n) \leq 10
\right\}
\gg \frac{x}{(\log x)^3}.
\]
\end{theorem}

	\begin{comment}
	\subsection*{Other Applications}
	\textcolor{red}{may modify this later}
	
	For instance, this approach could be used to study lower bounds for the number of splitting primes \( p \leq x \) such that
	\[
	\frac{|E(\mathbb{F}_p)|}{8} = 8P_{k_1}, \quad  
	\frac{|E(\mathbb{F}_{p^2})|}{|E(\mathbb{F}_p)|} = 4P_{k_2}, \quad  
	\frac{|E(\mathbb{F}_{p^l})|}{|E(\mathbb{F}_p)|} = P_{k_l} \quad \text{for any prime } l \leq N,
	\]
	with the condition that 
	\[
	k_1 + k_2 + \dots + k_{\pi(N)} \leq r,
	\]
	where the constant \( r \) is optimized through numerical computations for specific cases.
	\end{comment}
Theorem~\ref{theorem2_intro} gives a lower bound on the number of primes \( p \leq x \) for which both \( E(\mathbb{F}_p) \) and \( E(\mathbb{F}_{p^2}) \) contain a cyclic subgroup of finite index. More generally, for CM curves, a simplified sieve applied to the vector sequence \( (p - 1, p + 1) \) yields the following result.

\begin{theorem}\label{intro_cyclicity_lower_bound}
	Let \( E/\mathbb{Q} \) be an elliptic curve with complex multiplication and conductor \( N_E \). Then, for sufficiently large \( x \),
	\[
	\#\left\{ p \leq x : p \nmid N_E,\ E(\mathbb{F}_p) \text{ and } E(\mathbb{F}_{p^2}) \text{ contain a cyclic subgroup of index } \leq 12 \right\} \gg \frac{x}{(\log x)^3}.
	\]
\end{theorem}

For non-CM elliptic curves, we apply the effective Chebotarev density theorem under the \( \theta \)-quasi-GRH assumption (see p.~8), with \( \theta = \frac{1}{2} \). Similar results hold for any \( \theta = \frac{1}{2} + \epsilon \), \( 0 < \epsilon < \tfrac{1}{2} \), with adjusted sieve parameters. In Section~\ref{section_non-CM}, we establish the following analogue of Theorem~\ref{prop_intro}.

\begin{theorem}\label{intro_non-CM}
	Let \( E/\mathbb{Q} \) be a non-CM elliptic curve, and suppose \( c_E \neq 0 \) (see \eqref{def_C_E}). For any odd prime \(\ell\geq 3 \), define
	\[
	m_\ell = \left\lfloor 5.1 (\ell - 1) + 1.2 \right\rfloor, \quad 
	n_\ell = \left\lfloor 8.1 (\ell - 1) + 0.5 \right\rfloor.
	\]
	Assume GRH for all division fields \( \mathbb{Q}(E[n]) \). Then, for sufficiently large \( x \),
	\[
	\#\left\{ p \leq x : \omega\left( \frac{|E(\mathbb{F}_{p^\ell})|}{|E(\mathbb{F}_p)|} \right) \leq m_\ell \right\} \gg_E \frac{x}{(\log x)^2},
	\]
	and
	\[
	\#\left\{ p \leq x : \Omega\left( \frac{|E(\mathbb{F}_{p^\ell})|}{|E(\mathbb{F}_p)|} \right) \leq n_\ell \right\} \gg_E \frac{x}{(\log x)^2}.
	\]
	In particular, if \( E \) is a Serre curve, then \( c_E \neq 0 \), so the result holds for almost all non-CM curves.
\end{theorem}

\subsection{Comparisons in the CM and Non-CM Cases}\label{limitation}

In the CM case, where \( E \) has complex multiplication by an order \( \mathcal{O} \subset K \), we have \( \mathbb{Q}(\pi_p) \cong \mathcal{O} \otimes_\mathbb{Z} \mathbb{Q} \) for any ordinary prime \( p \). Since all such \( \pi_p \) lie in \( \mathcal{O}_K \), the ring of integers of \( K \), divisibility conditions on norms can be related  to divisibility by ideals in \( \mathcal{O}_K \), allowing one to apply a Bombieri–Vinogradov-type theorem over number fields to control error terms.

In contrast, in the non-CM case, the \( \pi_p \) vary over an infinite extension~\cite{infinite}, rather than lying in a fixed number field.   Divisibility conditions such as
\(q^n \mid |E(\mathbb{F}_p)| \) or \(q^n \mid \frac{|E(\mathbb{F}_{p^\ell})|}{|E(\mathbb{F}_p)|}\)
(for primes \( q \nmid pN_E \)) are interpreted via Chebotarev conditions. This approach applies for both CM and non-CM cases, but controlling the error requires the effective Chebotarev density theorem and thus   the  GRH.

\subsection{Notation}

We write \( \tau \), \( \mu \), and \( \varphi \) for the divisor function, the Möbius function, and Euler’s totient function, respectively. The function \( \Omega(n) \) counts the total number of prime divisors of \( n \) (with multiplicity), while \( \omega(n) \) counts the number of distinct prime divisors (without multiplicity).
 The constant \( \gamma \) denotes the Euler–Mascheroni constant. Primes are denoted by \( p, q \), and \( \ell\).

We write \( P_r \) for a product of at most \( r \) \textbf{distinct} prime factors, and \( P_r(z) \) for a product of at most \( r \) \textbf{distinct} primes, each at least \( z \).

Let \( B \subset \mathbb{N} \) and \( C \subset \mathbb{N} \times \mathbb{N} \) be multisets (elements may appear with multiplicity). For integers \( d, d_1, d_2 \), define
\[
B_d := \left\{ b \in B : d \mid b \right\}, \quad 
C_{d_1, d_2} := \left\{ (c_1, c_2) \in C : d_1 \mid c_1,\; d_2 \mid c_2 \right\},
\]
where multiplicities are preserved.

For \( z \in \mathbb{C} \), we write \( N(z) = z\overline{z} \) for its norm. If \( z \in K \), an imaginary quadratic field, this agrees with the field norm \( N_{K/\mathbb{Q}}(z) \). For a general extension \( L/K \), we write \( N_{L/K} \) for the field norm.

	\section{Sieving Tools}

	\subsection{The Bombieri--Vinogradov Theorem for Algebraic Number Fields}\label{section_bombierri_number_fields}
We now outline a Bombieri--Vinogradov type theorem that will be used in the CM case.  
Let \( K/\mathbb{Q} \) be an algebraic number field with ring of integers \( \mathcal{O}_K \). For an ideal \( I \subset \mathcal{O}_K \), define the \textit{generalized Euler function} by  
\begin{equation}\label{def_I}
	\Phi(I) = \left| \left(\mathcal{O}_K / I \right)^* \right|.
\end{equation}
Note that \( \Phi \) is multiplicative on ideals and satisfies  
\[
\Phi(\mathfrak{p}^k) = N_{K/\mathbb{Q}}(\mathfrak{p})^k \left( 1 - \frac{1}{N_{K/\mathbb{Q}}(\mathfrak{p})} \right),
\]
for any prime ideal \( \mathfrak{p} \) and integer \( k \geq 1 \); see, for example, \cite[Theorem~1.19]{Euler_formula_book}.

For an ideal \( I \subset \mathcal{O}_K \), define the \textit{generalized von Mangoldt function} \( \Lambda_K \) by  
\begin{equation}\label{def_Lamda}
	\Lambda_K(I) =
	\begin{cases}
		\log N_{K/\mathbb{Q}}(\mathfrak{p}), & \text{if } I = \mathfrak{p}^k \text{ for some prime ideal } \mathfrak{p},\, k \geq 1, \\
		0, & \text{otherwise}.
	\end{cases}
\end{equation}

We now state the following Bombieri--Vinogradov type theorem for number fields, due to Huxley \cite[Theorem~1]{huxley}.

	\begin{theorem}\cite[Theorem 1]{huxley}\label{theorem_Huxley}
	Let \( K/\mathbb{Q} \) be an algebraic number field of degree \( n \). For an ideal \( I \subset \mathcal{O}_K \), let \( h(I) \) denote the number of reduced narrow ideal classes modulo \( I \), and let \( \Phi_K(I) \) denote the generalized Euler function as defined in~\eqref{def_I}. For a positive number \( y \) and a reduced narrow ideal class \( H \) modulo \( I \), define
	\[
	\Psi(y, H) := \sum_{\substack{\mathfrak{p}^n \in H \\ N_{K/\mathbb{Q}}(\mathfrak{p}^n) \leq y}} \Lambda_K(\mathfrak{p}^n),
	\]
	where the sum is over prime ideal powers \( \mathfrak{p}^n \in H \) with \( N_{K/\mathbb{Q}}(\mathfrak{p}^n) \leq y \), and \( \Lambda_K(\mathfrak{p}^n) \) is the generalized von Mangoldt function as defined in~\eqref{def_Lamda}. Then for any \( A > 0 \), there exists \( B = B(A) > 0 \) such that
	\[
	\sum_{N_{K/\mathbb{Q}}(I) \leq \frac{x^{1/2}}{(\log x)^B}} \frac{h(I)}{\Phi_K(I)} 
	\max_H \max_{y \leq x} \left| \Psi(y, H) - \frac{y}{h(I)} \right| 
	\ll_{A,K} \frac{x}{(\log x)^A},
	\]
	where the first maximum is taken over reduced narrow ideal classes modulo \( I \).
	\end{theorem}
We apply Theorem~\ref{theorem_Huxley} to the case where \( K/\mathbb{Q} \) is the CM field of a CM elliptic curve. Since all such fields \( K \) are imaginary quadratic fields of class number one, their rings of integers \( \mathcal{O}_K \) are principal ideal domains (PIDs).

\begin{prop}\label{Bombierri-number field_2}
	Let \( K/\mathbb{Q} \) be an imaginary quadratic field of class number 1, and let \( w_K := |\mathcal{O}_K^\times| \) denote the number of units in the ring of integers \( \mathcal{O}_K \) of \( K \). Then, for any \( A > 0 \), there exists \( B = B(A) > 0 \) such that  
	\begin{equation}\label{rhs}
		\sum_{4 < N_{K/\mathbb{Q}}(I) \leq \frac{x^{1/2}}{(\log x)^B}}  
		\max_{\alpha \in (\mathcal{O}_K/I)^\times} \max_{y \leq x} \left| \Pi(y; \alpha, I) - \frac{w_K \operatorname{li}(y)}{\Phi(I)} \right| 
		\ll_{A,K} \frac{x}{(\log x)^A},
	\end{equation}
	where \( \Pi(y; \alpha, I) \) is defined by
	\[
	\Pi(y; \alpha, I) := \#\left\{ (\pi) \subset \mathcal{O}_K \,\middle|\, \pi \equiv \alpha \pmod{I}, \, (\pi) \text{ is a prime ideal}, \, N_{K/\mathbb{Q}}(\pi) \leq y \right\}.
	\]
\end{prop}
	\begin{proof}
		For any ideal \( I \subset \mathcal{O}_K \), let \( K^{I} \) be the subgroup of elements \( \alpha \in K^\times \) such that \( (\alpha) \) is a fractional ideal coprime to \( I \). Let \( K^{I,1} \subset K^I \) be the subgroup of elements \( \alpha \in K^I \) such that \( I \mid (\alpha - 1) \). Let \( \mathrm{Cl}_K^I \) denote the ray class group modulo \( I \). Since the ideal class group $Cl_K$ of $K$ is trivial, we have the exact sequence
		\[
		1 \to \mathcal{O}_K^\times \cap K^{I,1} \to \mathcal{O}_K^\times \to K^{I}/K^{I,1} \cong (\mathcal{O}_K/I)^{\times} \to \mathrm{Cl}_K^I \to 1.
		\]
	Hence, we may identify a class  \( H_\alpha \in \mathrm{Cl}_K^I \) for some \( \alpha \in (\mathcal{O}_K/I)^\times \) as
	\[
	H_\alpha = \left\{ (\pi) \subset \mathcal{O}_K \,\middle|\, \pi \equiv \alpha \pmod{I} \right\},
	\]
	and we have
	\[
	h(I) = \frac{|(\mathcal{O}_K/I)^\times|}{|\mathcal{O}_K^\times : \mathcal{O}_K^\times \cap K^{I,1}|} = \frac{\Phi(I)}{|\mathcal{O}_K^\times : \mathcal{O}_K^\times \cap K^{I,1}|}.
	\]
	In particular, if \( I \nmid (u - 1) \) for any unit \( u \neq 1 \), then \( \mathcal{O}_K^\times \cap K^{I,1} = \{1\} \), and we have
	\[
	h(I) = \frac{\Phi(I)}{|\mathcal{O}_K^\times|}.
	\]		
For imaginary quadratic fields \( K \), the unit group \( \mathcal{O}_K^\times \) is finite. Therefore, there exists a small integer \( N_K \) such that for any ideal \( I \) with \( N_{\mathbb{Q}/K}(I) > N_K \), we have \( \mathcal{O}_K^\times \cap K^{I,1} = \{1\} \). A direct check shows that \( N_K = 4 \) suffices for all such fields \( K \). Applying Theorem~\ref{theorem_Huxley} then completes the proof.
	\end{proof}

\begin{rem}
	Proposition~\ref{Bombierri-number field_2} can also be deduced from the large sieve inequality for imaginary quadratic fields established by Johnson; see \cite[Corollary]{Johnson}.
\end{rem}

For any algebraic number field, only finitely many primes are ramified. Moreover, in an imaginary quadratic field, inert primes \( q \) have norms \( q^2 \), so their contribution to \eqref{rhs} is expected to be negligible. Thus, in the definition of \( \Pi(y; \alpha, I) \), we may restrict \( \pi \) to be a splitting prime, as we will show in the following proposition.

	\begin{prop}\label{Bombierri-number field_3}
		Let \( K/\mathbb{Q} \) be an imaginary quadratic field of class number 1. Then, for any \( A > 0 \), there exists \( B = B(A) > 0 \) such that  
		\[
		\sum_{4 < N_{K/\mathbb{Q}}(I) \leq \frac{x^{1/2}}{(\log x)^B} }  
		\max_{\alpha \in (\mathcal{O}_K/I)^\times} \max_{y \leq x} \left| \Pi'(y; \alpha, I) - \frac{w_K \operatorname{li} \, y}{\Phi(I)} \right| 
		\ll_{A,K} \frac{x}{(\log x)^A},
		\]
		where \( \Pi'(y; \alpha, I) \) is defined by 
		\[
		\Pi'(y; \alpha, I) := \#\{ (\pi) \subset \mathcal{O}_K  \,\mid\, \pi \equiv \alpha \pmod{I}, \,   N_{K/\mathbb{Q}}(\pi) =p\leq y \}.
		\]
	\end{prop}
	\begin{proof}
Let \( K = \mathbb{Q}(\sqrt{D}) \), and let \( \mathcal{O}_K = \mathbb{Z}[\omega_D] \), where \( \omega_D = \sqrt{D} \) if \( D \equiv 2,3 \pmod{4} \), and \( \omega_D = \frac{1 + \sqrt{D}}{2} \) if \( D \equiv 1 \pmod{4} \).
 
Let \( q \) be a rational prime that is inert in \( \mathcal{O}_K \) and satisfies \( q \equiv \alpha \pmod{I} \), where \( \alpha = a + b\omega_D \) with \( a, b \in \mathbb{Z} \), and suppose that \( N_{K/\mathbb{Q}}(q) = q^2 \leq x \). Then we have the following   integer divisibility conditions: 
 If \( D \equiv 2,3 \pmod{4} \), then
\[
N_{K/\mathbb{Q}}(I) \mid (q - a)^2 + b^2 |D|;
\]
If \( D \equiv 1 \pmod{4} \), then
\[
N_{K/\mathbb{Q}}(I) \mid \left(q - a - \frac{b}{2}\right)^2 + \frac{b^2 |D|}{4},
\]
which implies
\[
N_{K/\mathbb{Q}}(I) \mid \left(2q - 2a - b\right)^2 + b^2 |D|.
\]
For a fixed integer \( c \), the number of solutions to the congruence \( x^2 \equiv c \pmod{N_{K/\mathbb{Q}}(I)} \) is at most \( 2^{\omega(N_{K/\mathbb{Q}}(I)) + 1} \), and each such solution corresponds to at most two congruence classes for \( q \bmod N_{K/\mathbb{Q}}(I) \). Therefore, the number of such primes \( q \in \mathcal{O}_K \) satisfying \( q \equiv \alpha \pmod{I} \) and \( N_{K/\mathbb{Q}}(q) = q^2 \leq x \) is bounded by
\[
\ll \frac{2^{\omega(N_{K/\mathbb{Q}}(I))} \cdot \operatorname{li}(x^{1/2})}{\varphi(N_{K/\mathbb{Q}}(I))}.
\]

	Summing over ideals \( I \) with \( N_{K/\mathbb{Q}}(I) < x^{1/2} \), and noting that the number of such \( I \) with norm \( d \) is \( \leq \tau(d) \), the total contribution to the right-hand side of \eqref{rhs} from inert primes is bounded by 
	\[
	\ll \sum_{d < x^{1/2}} \frac{\tau(d) \cdot 2^{\omega(d)}}{\varphi(d)} \cdot \operatorname{li}(x^{1/2}) 
	\ll \frac{x^{1/2}}{\log x} \sum_{d < x^{1/2}} \frac{\tau^2(d)}{\varphi(d)} 
	\ll x^{1/2} \log^3 x.
	\]
	
	Combining this with Proposition~\ref{Bombierri-number field_2} completes the proof.
	\end{proof}
	\subsection{The Effective Version of the Chebotarev Density Theorem}\label{chebotarev}
	
Here, we outline the version of the Chebotarev density theorem that will be applied in the non-CM case.

In \cite{steuding}, the effective Chebotarev density theorem of Murty, Murty, and Saradha \cite[Proposition~3.12]{chebotarev}, together with a reduction method due to Serre \cite[Propositions~7 and~8]{serre}, was used under the assumption of the Generalized Riemann Hypothesis (GRH) for Artin \( L \)-functions. 
Alternatively, one can obtain a similar estimate under a weaker assumption by starting the derivation from the effective Chebotarev density theorem of Lagarias and Odlyzko \cite{unconditional_chabotarev} under a milder hypothesis.

Let \( K/\mathbb{Q} \) be a number field, and let \( L/K \) be a Galois extension of number fields with Galois group \( G := \operatorname{Gal}(L/K) \). Denote by \( n_K = [K:\mathbb{Q}] \) and \( n_L = [L:\mathbb{Q}] \), and let \( d_K \) and \( d_L \) be the absolute discriminants of \( K \) and \( L \), respectively.

Let \( \mathfrak{p} \) be a prime ideal of \( K \), and let \( \mathfrak{q} \) be a prime ideal of \( L \) lying above \( \mathfrak{p} \). Denote by \( \mathbb{F}_{\mathfrak{p}} \) and \( \mathbb{F}_{\mathfrak{q}} \) the corresponding residue fields. Let \( D_{\mathfrak{q}} \) and \( I_{\mathfrak{q}} \) be the decomposition and inertia subgroups of \( G \) at \( \mathfrak{q} \). Then
\[
D_{\mathfrak{q}} / I_{\mathfrak{q}} \cong \operatorname{Gal}(\mathbb{F}_{\mathfrak{q}} / \mathbb{F}_{\mathfrak{p}}),
\]
and this Galois group is cyclic, generated by the Frobenius automorphism at \( \mathfrak{p} \), denoted \( \sigma_{\mathfrak{p}} \), which acts via \( x \mapsto x^{|\mathbb{F}_{\mathfrak{p}}|} \).
 If \( \mathfrak{q} \) is unramified in \( L \), then \( I_{\mathfrak{q}} \) is trivial, and the Frobenius automorphism \( \sigma_{\mathfrak{p}} \) can be uniquely identified, up to conjugacy in \( G = \operatorname{Gal}(L/K) \), as the Artin symbol \( \left( \dfrac{L/K}{\mathfrak{p}} \right) \). 
 
 Let \( \varphi \) be a \emph{class function} on \( G \), that is, a function invariant under conjugation. For example, \( \varphi \) could be the characteristic function of a union of conjugacy classes.
  Define
\[
\pi_\varphi(x, L/K) := \sum_{\substack{N_{K/\mathbb{Q}}(\mathfrak{p}) \leq x \\ \mathfrak{p} \text{ unramified in } L/K}} \varphi\left( \left( \frac{L/K}{\mathfrak{p}} \right) \right),
\]
where \( \left( \frac{L/K}{\mathfrak{p}} \right) \) denotes the Artin symbol.

We now define a variation \( \widetilde{\pi}_\varphi(x, L/K) \), which requires extending the definition of \( \varphi \) to possibly ramified primes:
\[
\widetilde{\pi}_\varphi(x, L/K) := \sum_{\substack{N_{K/\mathbb{Q}}(\mathfrak{p}) \leq x}} \frac{1}{m} \sum_{\substack{g \in D_{\mathfrak{q}} \\ g I_{\mathfrak{q}} = \sigma_{\mathfrak{p}}^m}} \varphi(g),
\]
where the inner sum is over all elements \( g \in D_{\mathfrak{q}} \) whose image in the quotient \( D_{\mathfrak{q}} / I_{\mathfrak{q}} \cong \operatorname{Gal}(\mathbb{F}_{\mathfrak{q}} / \mathbb{F}_{\mathfrak{p}}) \) equals \( \sigma_{\mathfrak{p}}^m \), the \( m \)-th power of the Frobenius automorphism at \( \mathfrak{p} \).

\begin{prop}[{\cite[Proposition 7]{serre}}]\label{prop_serre_7}
	\[
	\widetilde{\pi}_\varphi(x, L/K) - \pi_\varphi(x, L/K) = 
	O\left( 
	\sup_{g \in G} |\varphi(g)| \left( \frac{1}{|G|} \log d_L + n_K x^{1/2} \right) 
	\right).
	\]
\end{prop}

\begin{prop}[{\cite[Proposition 8]{serre}}]\label{prop_serre_8}
	(a) Let \( H \) be a subgroup of \( G \), \( \varphi_H \) a class function on \( H \), and let \( \varphi_G = \operatorname{Ind}_H^G \varphi_H \) be the class function on \( G \) induced from \( \varphi_H \). Then
		\[
		\widetilde{\pi}_{\varphi_G}(x, L/K) = \widetilde{\pi}_{\varphi_H}(x, L/L^H), \quad \text{for all } x \geq 2.
		\]
		
		(b) Let \( N \) be a normal subgroup of \( G \), \( \varphi_{G/N} \) a class function on \( G/N \), and let \( \varphi_G \) be the class function on \( G \) obtained by composing \( \varphi_{G/N} \) with the natural surjection \( G \to G/N \). Then
		\[
		\widetilde{\pi}_{\varphi_G}(x, L/K) = \widetilde{\pi}_{\varphi_{G/N}}(x, L^N/K), \quad \text{for all } x \geq 2.
		\]

\end{prop}

Let \( C \subset G \) be a union of conjugacy classes, and let \( \varphi_C: G \to \{0,1\} \) be the characteristic function of \( C \). Define
\[
\pi_C(x, L/K) := \pi_{\varphi_C}(x, L/K) = \# \left\{ \mathfrak{p} \subset K \mid \mathfrak{p} \text{ unramified in } L,\ \sigma_{\mathfrak{p}} \in C,\ N_{K/\mathbb{Q}}(\mathfrak{p}) \leq x \right\}
\]
as the number of prime ideals \( \mathfrak{p} \) of \( K \) that are unramified in \( L \), have Frobenius class \( \sigma_{\mathfrak{p}} \in C \), and norm \( N_{K/\mathbb{Q}}(\mathfrak{p}) \leq x \).

	Define
	\[
	M(L/K) = [L:K] d_K^{1/n_K} \prod_{p \in P(L/K)} p,
	\]
where \( P(L/K) \) denotes the set of rational primes lying below the primes in \( K \) that ramify in \( L \).

	The following result is a variation of a theorem from Murty, Murty, and Saradha \cite[Proposition 3.9]{chebotarev}, with the specific formulation used here given in David and Wu \cite[Theorem 3.7]{wu}.
	\begin{theorem}[{\cite[Proposition 3.9]{chebotarev}, \cite[Theorem 3.7]{wu}}]\label{theorem_ch1}
		Suppose that \( L/K \) is a Galois extension of number fields with Galois group \( G \). Let \( H \) be a subgroup of \( G \) and \( C \) be a union of  conjugacy classes in \( G \) such that \( C \cap H \neq \emptyset \).  
		
		Let \( C_H \) be the union of conjugacy classes in \( H \) generated by \( C \cap H \). Then, we have
		\[
		\pi_C(x, L/K) = \frac{|H|}{|G|} \frac{|C|}{|C_H|} \pi_{C_H}(x, L/L^H) 
		+ O \left( \frac{|C|}{|C_H| |G|} \log d_L 
		+ \frac{|H|}{|G|} \frac{|C|}{|C_H|}  [L^H : \mathbb{Q}] x^{1/2} + [K : \mathbb{Q}] x^{1/2}  \right).
		\]
	\end{theorem}
\begin{proof}
Let \( \varphi_G \) be the class function on \( G \) induced from \( C_H \). Then
\[
\varphi_G = \operatorname{Ind}_H^G (\varphi_{C_H}) = \frac{|G|}{|H|} \cdot \frac{|C_H|}{|C|} \cdot \varphi_C.
\]
By Proposition~\ref{prop_serre_8}\,(a),
\begin{equation} \label{eq_1}
	\widetilde{\pi}_{\varphi_G}(x, L/K) = \frac{|G|}{|H|} \cdot \frac{|C_H|}{|C|}\cdot \widetilde{\pi}_C(x, L/K)=	\widetilde{\pi}_{C_H}(x, L/L^H),
\end{equation}
and by Proposition~\ref{prop_serre_7} applied to \( \widetilde{\pi}_C(x, L/K) \) and
\( \widetilde{\pi}_{C_H}(x, L/K) \), we have
\begin{align*}
	\pi_C(x, L/K) 
		&=  \widetilde{\pi}_C(x, L/K)
	+ O\left( [K : \mathbb{Q}] x^{1/2} + \frac{\log d_L}{[L:K]}  \right) \\
	&= \frac{|H|}{|G|} \frac{|C|}{|C_H|} \cdot \widetilde{\pi}_{C_H}(x, L/L^H) 
	+ O\left( [K : \mathbb{Q}] x^{1/2} + \frac{\log d_L}{[L:K]}  \right) \\
	&= \frac{|H|}{|G|} \frac{|C|}{|C_H|} \cdot \pi_{C_H}(x, L/L^H) \\
	&\quad + O\left( 
\frac{|H|}{|G|} \frac{|C|}{|C_H|} \left( \frac{\log d_L}{[L : L^H]} + [L^H : \mathbb{Q}] x^{1/2} \right)
	+ [K : \mathbb{Q}] x^{1/2} + \frac{\log d_L }{[L:K]} 
	\right) \\
	&= \frac{|H|}{|G|} \frac{|C|}{|C_H|} \cdot \pi_{C_H}(x, L/L^H) \\
	&\quad + O\left( 
	\frac{|C|}{|G||C_H|} \log d_L
	+ \frac{|H|}{|G|}  \frac{|C|}{|C_H|} \cdot [L^H : \mathbb{Q}] x^{1/2}
	+ [K : \mathbb{Q}] x^{1/2} 
	\right).
\end{align*}
\end{proof}

	\newtheorem*{conjecture}{Conjecture}
	\newtheorem*{hypothesis}{Hypothesis}
	We now introduce some notation and state Artin's holomorphy conjecture and the \( \theta \)-quasi-GRH.

\begin{conjecture}[Artin’s Holomorphy Conjecture]
	Let \( L/K \) be a Galois extension with Galois group \( G = \operatorname{Gal}(L/K) \). We say that \( L/K \) satisfies the \textbf{Artin Holomorphy Conjecture (AHC)} if, for every irreducible character \( \chi \colon G \to \mathbb{C} \), the associated Artin \( L \)-function \( L(s, \chi) \) admits a meromorphic continuation to the entire complex plane and is analytic on \( \mathbb{C} \) when \( \chi \neq 1 \), and analytic on \( \mathbb{C} \setminus \{1\} \) when \( \chi = 1 \).
\end{conjecture}

It is known that the AHC holds when \( \operatorname{Gal}(L/K) \) is abelian~\cite{Ar37}.

	We present a version of the Chebotarev density theorem based on the following weaker hypothesis.

\begin{hypothesis}[\(\theta\)-quasi-GRH] 
	Let \( L/\mathbb{Q} \) be a number field, and let \( \frac{1}{2} \leq \theta < 1 \). We say that the \(\theta\)-\textbf{quasi-Generalized Riemann Hypothesis ($\theta$-quasi-GRH)} holds for \( L \) if the Dedekind zeta function \( \zeta_L(s) \) has no zeros in the half-plane \( \operatorname{Re}(s) > \theta \).
\end{hypothesis}

The following theorem originates from Murty, Murty, and Saradha~\cite[Proposition~3.12]{chebotarev}, which corresponds to the case \( \theta = \frac{1}{2} \). To relax the hypothesis on \( \theta \), one can follow a similar approach to that of~\cite[Proposition~3.12]{chebotarev}, with appropriate modifications. The version stated here is as given in David and Wu~\cite[Theorem~3.8]{wu}.

	\begin{theorem}[{\cite[Proposition 3.12]{chebotarev}, \cite[Theorem 3.8]{wu}}]\label{theorem_ch2}
	Let \( L/K \) be a Galois extension of number fields with Galois group \( G = \operatorname{Gal}(L/K) \). 	Let \( H \subset G \) be a normal subgroup, and let \( D \subset G \) be a non-empty union of conjugacy classes such that \( H D \subseteq D \).   Suppose that the \( \theta \)-quasi-GRH holds for \( L \), and that the  AHC holds for the extension \( L^H/K \), whose Galois group is \( G/H \).
	
Then,
	\[
	\pi_D(x, L/K) = \frac{|D|}{|G|} \operatorname{li}(x) + O\left( \left( \frac{|D|}{|H|} \right)^{1/2} x^\theta n_K \log(M(L/K)x) \right).
	\]
	\end{theorem}
	
	\subsection{Vector sieve of linear sieves}
We begin with the following setup  as in Heath-Brown and Li \cite{HBL2016}. Let \( \mathcal{W} \subset \mathbb{N}^2 \) be a finite subset. Suppose \( z_1, z_2 \geq 2 \) satisfy
\[
\log z_1 \asymp \log z_2.
\]
For positive integers \( d_1, d_2 \), define
\[
\mathcal{W}_{d_1, d_2} := \left\{ (n_1, n_2) \in \mathcal{W} \,:\, d_1 \mid n_1,\; d_2 \mid n_2 \right\},
\]
and
\[
S(\mathcal{W}; z_1, z_2) := \left\{ (m, n) \in \mathcal{W} \,:\,  (m, P(z_1)) = (n, P(z_2)) = 1 \right\},
\]
where
\[
P(z) := \prod_{p < z} p.
\]

	Suppose that
	\begin{equation}\label{W_d}
		|\mathcal{W}_{d_1,d_2}| = h(d_1,d_2) X + r(d_1,d_2)
	\end{equation}
	for some multiplicative function $h(d_1,d_2) \in (0, 1]$ such that for all prime $p$,
	\begin{equation}\label{W_d_condition1}
		h(p, 1) + h(1, p) - 1 < h(p, p) \leq h(p, 1) + h(1, p),
	\end{equation}  and there exists a constant \( C_1 \geq 2 \) such that
\begin{equation}\label{condition_h}
h(p, 1),\; h(1, p) \leq C_1 p^{-1}, \quad \text{and} \quad h(p, p) \leq C_1 p^{-2}.
\end{equation}

	Moreover, suppose the functions 
	\[
	h_1(d) := h(d, 1) \quad \text{and} \quad h_2(d) := h(1, d)
	\] satisfy the linear sieve assumption 
	\begin{equation}\label{linear_sieve_condition}
		\prod_{w \leq p < z} (1 - h(p))^{-1} \leq \frac{\log z}{\log w} \left( 1 + \frac{L}{\log w} \right).
	\end{equation}
Finally, for real numbers \( z_0 < z_1, z_2 \), define
\begin{align}\label{definition_Vi}
	V_i := \prod_{z_0 \leq p < z_i} \left( 1 - h_i(p) \right) \quad (i = 1, 2), \qquad
	V(z_0, h^*) := \prod_{p < z_0} \left( 1 - h^*(p) \right),
\end{align}
where \( h^* \) is the multiplicative function given by
\begin{equation} \label{Def: hstar}
	h^*(d) = \sum_{e_1 e_2 e_3 = d} h(e_1 e_3, e_2 e_3) \mu(e_3).
\end{equation}
	
	Then we have the following upper and lower bounds for $S(\mathcal{W}; z_1,z_2)$.
	
	\begin{prop}\label{prop_vector sieve}\cite[Proposition 1]{HBL2016}
		Suppose that \( h(d_1,d_2) \) satisfies \eqref{condition_h}, and that
	\[
	h_1(d) := h(d, 1), \qquad h_2(d) := h(1, d)
	\]
	both satisfy the linear sieve assumption \eqref{linear_sieve_condition}. Let \( F(s) \) and \( f(s) \) denote the standard upper and lower bound functions for the linear sieve \cite[Theorem~11.12; see also (12.1) and (12.2)]{FI2010}.
	
	Let \( z_1, z_2 \geq 2 \) with \( \log z_1 \asymp \log z_2 \), and define \( z_0 := e^{(\log z_1 z_2)^{1/3}} \). Let \( D = z_1^{s_1} z_2^{s_2} \), where \( 1 \leq s_i \ll 1 \) for \( i \in \{1, 2\} \), and define
	\[
	\sigma_i := \frac{\log D}{\log z_i}, \qquad i = 1, 2.
	\]
	
	Then we have the upper bound
\[
|S(\mathcal{W}; z_1, z_2)| \leq X V(z_0, h^*) V_1 V_2 F(\sigma_1, \sigma_2) \left\{ 1 + O\left((\log D)^{-1/6} \right) \right\}
+ O_{\epsilon} \left( \sum_{d_1 d_2 < D^{1 + \epsilon}} \tau(d_1 d_2)^4 |r(d_1, d_2)| \right),
\]
for any fixed \( \epsilon > 0 \), where
\[
F(\sigma_1, \sigma_2) := \inf \left\{ F(s_1) F(s_2) \,\middle\vert\, \frac{s_1}{\sigma_1} + \frac{s_2}{\sigma_2} = 1,\ s_i \geq 1 \ (i = 1, 2) \right\}.
\]

Similarly, we have the lower bound
\[
|S(\mathcal{W}; z_1, z_2) |\geq X V(z_0, h^*) V_1 V_2 f(\sigma_1, \sigma_2) \left\{ 1 + O\left((\log D)^{-1/6} \right) \right\}
+ O_{\epsilon} \left( \sum_{d_1 d_2 < D^{1 + \epsilon}} \tau(d_1 d_2)^4 |r(d_1, d_2)| \right),
\]
where
\[
f(\sigma_1, \sigma_2) := \sup \left\{ f(s_1) F(s_2) + f(s_2) F(s_1) - F(s_1) F(s_2) \,\middle\vert\, \frac{s_1}{\sigma_1} + \frac{s_2}{\sigma_2} = 1,\ s_i \geq 2 \ (i = 1, 2) \right\}.
\]
	\end{prop}
 
	\begin{proof}
		See \cite[Section 3.3]{HBL2016} for details.
	\end{proof}

	\begin{rem}
	The vector sieve in Proposition~\ref{prop_vector sieve} extends to two beta-sieves with mixed dimensions \cite[Proposition~3.10]{kunja}. It can also be adapted to sieve problems with additional variables \( z_1, \dots, z_k \) for \( k \geq 2 \), given suitable combinatorial modifications.

		\begin{comment}
		By employing a vector sieve with additional variables, one could potentially study lower bounds on the number of splitting primes \( p \leq x \) for some fixed \( N > 2 \) satisfying  
		\[
		\frac{|E(\mathbb{F}_p)|}{8} = 8P_{k_1}, \quad  
		\frac{|E(\mathbb{F}_{p^2})|}{|E(\mathbb{F}_p)|} = 4 P_{k_2}, \quad  
		\frac{|E(\mathbb{F}_{p^l})|}{|E(\mathbb{F}_p)|} = P_{k_l}, \quad  
		\text{for any prime } 2 <\ell\leq N,
		\]
		with  
		\[
		k_1 + k_2 + \dots + k_{\pi(N)} \leq r,
		\]
		where the optimal choice of the parameter \( r \) depends on numerical computations specific to each case. However, such a generalization to multiple variables can be computationally intensive and may not always be of practical interest.
		\end{comment}
	\end{rem}

	\section{Orders   of \(E(\mathbb{F}_p)\) and \(E(\mathbb{F}_{p^2})\) for CM elliptic curves}\label{W(x)}
	
In this section, we prove Theorem~\ref{theorem1_intro} using a vector sieve. We then establish Theorem~\ref{theorem2_intro} by analyzing the group structures of the elliptic curves counted in Theorem~\ref{theorem1_intro}. Finally, we apply a simplified sieve argument to prove Theorem~\ref{intro_cyclicity_lower_bound}.
	
	\subsection{Sieve Set-up and Weighted Terms}\label{section_setup}
We put
\begin{equation}\label{def_P(z)}
	\mathcal{P}(z) := \{ p \text{ prime} : p < z \}, \quad \text{and} \quad P(z) := \prod_{p \in \mathcal{P}(z)} p = \prod_{p < z} p.
\end{equation}

Let \( 0 < \theta_2 \leq \theta_1 < \frac{1}{2} \) be parameters to be chosen later. Define
\[
z_1 := x^{\theta_1}, \qquad z_2 := x^{\theta_2}.
\]
Then,
\[
0 < z_2 \leq z_1 < x^{1/2}.
\]

Let \( \delta_1 \) and \( \delta_2 \) be parameters satisfying \( \theta_1 < \delta_1 < 1 \) and \( \theta_2 < \delta_2 < 1 \), and define
\[
y_1 := x^{\delta_1}, \qquad y_2 := x^{\delta_2}.
\]

We will apply the vector sieve from Proposition~\ref{prop_vector sieve} to the multiset
\[	\mathcal{W} := \left\{ 
\left(N\left(\frac{\pi - 1}{2(1+i)}\right),\; N\left(\frac{\pi + 1}{2}\right)\right)  
\;\middle|\; N(\pi) = p \leq x,\; \pi \in \mathbb{Z}[i] \text{ primary}
\right\},\]
and incorporate a weighted sieve to reduce the total number of prime factors:
\[
\omega\left(N\left(\frac{\pi - 1}{2(1+i)}\right)\right) + \omega\left(N\left(\frac{\pi + 1}{2}\right)\right).
\]
Here, the condition of being \emph{primary} means \( \pi \equiv 1 \pmod{2(1+i)} \).
Note that the elements of \( \mathcal{W} \) are indexed by \( \pi \) and counted with multiplicity. Sets in which elements may appear with multiplicity are referred to as \emph{multisets}.

For integers \( d_1, d_2 \), define  
\[
\mathcal{W}_{d_1, d_2} := \left\{ (m,n) \in \mathcal{W}  \;:\;  d_1 \mid m,\; d_2 \mid n \right\},
\]
and  
\[	S(\mathcal{W}; z_1, z_2) := \left\{ (m, n) \in \mathcal{W} \;:\; (P(z_1), m) = (P(z_2), n) = 1 \right\}.\]

Consider the weight function
\[
w_p(y) := 1 - \frac{\log p}{\log y}.
\]
Note that for any \( (a, b) \in \mathcal{W} \), we have \( a, b \leq x \). Since \( w_p(y_1) < 0 \) for \( p > y_1 \), it follows that
\begin{align*}
	\sum_{z_1 \leq p \leq y_1} w_p(y_1) \left| S(\mathcal{W}_{p,1}; z_1, z_2) \right|
	&\geq \sum_{z_1 \leq p \leq x} w_p(y_1) \left| S(\mathcal{W}_{p,1}; z_1, z_2) \right| \\
	&= \sum_{(a,b) \in S(\mathcal{W}; z_1, z_2)} \left( \omega(a) - \frac{1}{\log y_1} \sum_{p \mid a} \log p \right) \\
	&\geq \sum_{(a,b) \in S(\mathcal{W}; z_1, z_2)} \left( \omega(a) - \frac{\log x}{\log y_1} \right) \\
	&= \sum_{(a,b) \in S(\mathcal{W}; z_1, z_2)} \left( \omega(a) - \frac{1}{\delta_1} \right),
\end{align*}
where \( \omega(a) \) denotes the number of distinct prime factors of \( a \).
Similarly, we have
\[
\sum_{z_2 \leq p \leq y_2} w_p(y_2) \left| S(\mathcal{W}_{1,p}; z_1, z_2) \right|
\geq \sum_{(a,b) \in S(\mathcal{W}; z_1, z_2)} \left( \omega(b) - \frac{1}{\delta_2} \right).
\]

The number of elements \( (a, b) \in S(\mathcal{W}; z_1, z_2) \) for which \( a \) is not square-free is bounded by
\begin{equation}\label{square_free_3.1}
	\sum_{z_1 < p < x} \sum_{\substack{n \leq x \\ p^2 \mid n}} \tau(n) 
	\ll \sum_{z_1 < p < x} \sum_{\substack{k \leq x / p^2}} \tau(k) 
	\ll \sum_{z_1 < p < x} \frac{x}{p^2} \log \frac{x}{p^2} 
	\ll \frac{x}{z_1} \log x.
\end{equation}
Similarly, the number of \( (a, b) \in \mathcal{W} \) such that \( b \) is not square-free is bounded by \( \frac{x}{z_2} \log x \).

It then follows that for any constant \( \lambda > 0 \),
\begin{align}\label{omega(a)+omega(b)}
	&|S(\mathcal{W}; z_1, z_2)| 
	- \lambda \sum_{z_1 \leq p \leq y_1} w_p(y_1)\, |S(\mathcal{W}_{p,1}; z_1, z_2)| 
	- \lambda \sum_{z_2 \leq p \leq y_2} w_p(y_2)\, |S(\mathcal{W}_{1,p}; z_1, z_2)| \nonumber \\
	&\leq \sum_{(a,b) \in S(\mathcal{W}; z_1, z_2)}  
	\left(1 + \frac{\lambda}{\delta_1} + \frac{\lambda}{\delta_2} - \lambda \omega(a) - \lambda \omega(b) \right) \nonumber \\
	&\leq O\left( \frac{x}{z_2} \right) 
	+ \left( 1 + \frac{\lambda}{\delta_1} + \frac{\lambda}{\delta_2} \right)
	\#\left\{ (a, b) \in S(\mathcal{W}; z_1, z_2) \,\middle|\,
	\omega(a) + \omega(b) \leq \frac{1}{\lambda} + \frac{1}{\delta_1} + \frac{1}{\delta_2},\ a, b \text{ square-free} \right\}.
\end{align}

Therefore, it suffices to show that for sufficiently large \( x \),
\begin{align*}
	W(x) :=\ & |S(\mathcal{W}; z_1, z_2)| 
	- \lambda \sum_{z_1 \leq p \leq y_1} w_p(y_1)\, |S(\mathcal{W}_{p,1}; z_1, z_2)| 
	- \lambda \sum_{z_2 \leq p \leq y_2} w_p(y_2)\, |S(\mathcal{W}_{1,p}; z_1, z_2)| \\
	\geq\ & \left( c + o(1) \right) \frac{x}{(\log x)^3},
\end{align*}
for some constant \( c > 0 \).

	\subsection{Estimation of $W(x)$}
To apply the vector sieve to estimate \( W(x) \), we evaluate
\[
\mathcal{W}_{d_1,d_2} := \left\{ (a, b) \in \mathcal{W} \;:\; d_1 \mid a,\; d_2 \mid b \right\}, \quad \text{for } d_1 \mid P(z_1),\; d_2 \mid P(z_2).
\]

Since divisibility of norms does not correspond directly to divisibility  in the Gaussian integers, we convert the condition into a coprimality condition using the following lemma, similar to that in \cite{IU_elliptic}. This is a standard application of the inclusion-exclusion principle and generalizes naturally to any multiset \( \mathcal{C} \subset \mathbb{N}^k \) of \( k \)-tuples.

	\begin{lemma}\label{lemma_Mobius}
		Let \( \mathcal{C} \subset \mathbb{N}\times \mathbb{N}\) be a finite multiset of pairs of non-negative integers. For integers \( k_1, k_2 > 0 \), define
		\[
		S(\mathcal{C}, k_1, k_2) := \left\{ (c_1, c_2) \in \mathcal{C} \;\middle\vert\; (c_1, k_1) = 1, \ (c_2, k_2) = 1 \right\}.
		\]
	Then, for any square-free integers \( d_1, d_2 > 0 \), we have
		\[
		|\mathcal{C}_{d_1,d_2}| = \sum_{k_1 \mid d_1} \mu(k_1) \sum_{k_2 \mid d_2} \mu(k_2)|S(\mathcal{C}, k_1, k_2)|.
		\]
	\end{lemma}
	
	\begin{proof}
		Let \( C = \{c_{m,n}\}_{(m,n) \in \mathbb{N} \times \mathbb{N}} \) be the characteristic function of \( \mathcal{C} \), that is, \( c_{m,n} = k_{m,n} \) if \( (m,n) \in \mathcal{C} \) with multiplicity \( k_{m,n} \), and \( c_{m,n} = 0 \) otherwise. Then, 
		\[
		|S(\mathcal{C}, k_1, k_2)| = \sum_{\substack{t_1 \mid m \\ t_1 \mid k_1}} \mu(t_1) \sum_{\substack{t_2 \mid n \\ t_2 \mid k_2}} \mu(t_2) c_{m,n} = \sum_{t_1 \mid k_1} \mu(t_1) \sum_{t_2 \mid k_2} \mu(t_2) |\mathcal{C}_{t_1,t_2}|.
		\]
	Changing the order of summation, we obtain
		\begin{align} 
			\sum_{k_1 \mid d_1} \mu(k_1) \sum_{k_2 \mid d_2} \mu(k_2)|S(\mathcal{C}, k_1, k_2)|
			&= \sum_{k_1 \mid d_1} \mu(k_1) \sum_{k_2 \mid d_2} \mu(k_2) \sum_{t_1 \mid k_1} \mu(t_1) \sum_{t_2 \mid k_2} \mu(t_2) |\mathcal{C}_{t_1,t_2}|\nonumber \\
			&= \sum_{t_1 \mid d_1} \sum_{n_1 \mid \frac{d_1}{t_1}} \mu(t_1) \mu(n_1 t_1) \sum_{t_2 \mid d_2} \sum_{n_2 \mid \frac{d_2}{t_2}} \mu(t_2) \mu(n_2 t_2) |\mathcal{C}_{t_1,t_2}|\nonumber \\
			&= \sum_{t_1 \mid d_1} \mu^2(t_1) \sum_{t_2 \mid d_2} \mu^2(t_2) |\mathcal{C}_{t_1,t_2}| \sum_{n_1 \mid \frac{d_1}{t_1}} \mu(n_1) \sum_{n_2 \mid \frac{d_2}{t_2}} \mu(n_2)\nonumber \\
			&= |\mathcal{C}_{d_1,d_2}|,
		\end{align}
	where we used \( d_1 = n_1 t_1 \) and  \( d_2 = n_2 t_2 \) are square-free.
	\end{proof}

Applying Lemma~\ref{lemma_Mobius} to \( \mathcal{W} \), we obtain that for any square-free \( d_1, d_2 > 0 \),
	\begin{equation}\label{W_d_1_d_2}
		|\mathcal{W}_{d_1, d_2}| = \sum_{k_1 \mid d_1} \mu(k_1) \sum_{k_2 \mid d_2} \mu(k_2) |S(\mathcal{W}, k_1, k_2)|.
	\end{equation}

Since \( \mathbb{Z}[i] \) (or any imaginary quadratic number ring of class number 1) is a PID, we may identify \( \alpha \in \mathbb{Z}[i] \) with the principal ideal \( (\alpha) \) when the context is clear.

For any square-free integer \( k \) and \( \alpha \in \mathbb{Z}[i] \), we have \(  (N(\alpha), k) = 1 \) if and only if \( (\alpha, k) = (1) \) in \( \mathbb{Z}[i] \). Hence,
\begin{equation}\label{eq_S(W,k_1,k_2)}
	|S(\mathcal{W}, k_1, k_2)| = |S(\widetilde{\mathcal{W}}, k_1, k_2)|,
\end{equation}
where
\[
\widetilde{\mathcal{W}} := \left\{ \left( \frac{\pi - 1}{2(1+i)},\, \frac{\pi + 1}{2} \right) 
\;\middle|\; N(\pi) = p \leq x,\; \pi \in \mathbb{Z}[i] \text{ primary} \right\},
\]
and
\[
S(\widetilde{\mathcal{W}}, k_1, k_2) := 
\left\{ (\alpha, \beta) \in \widetilde{\mathcal{W}} \;\middle|\; (\alpha, k_1) = (1),\; (\beta, k_2) = (1) \right\}.
\]

To estimate \( |S(\mathcal{W}, k_1, k_2)| \), we apply the inclusion-exclusion principle for ideals in \( \mathbb{Z}[i] \). Define the Möbius function on ideals by
\[
\widehat{\mu}(\mathfrak{a}) = 
\begin{cases}
	1, & \text{if } \mathfrak{a} = (1), \\
	(-1)^r, & \text{if } \mathfrak{a} = \lambda_1 \cdots \lambda_r \text{ with distinct prime ideals } \lambda_i, \\
	0, & \text{if } \lambda^2 \mid \mathfrak{a} \text{ for some prime ideal } \lambda.
\end{cases}
\]
This function is multiplicative and satisfies
\[
\sum_{\mathfrak{a} \mid \mathfrak{b}} \widehat{\mu}(\mathfrak{a}) =
\begin{cases}
	1, & \text{if } \mathfrak{b} = (1), \\
	0, & \text{otherwise.}
\end{cases}
\]
The function \( \widehat{\mu} \) is well-defined in any number ring by unique factorization of ideals.

For any integers \( k_1, k_2 \), we apply inclusion-exclusion to obtain
\begin{align}\label{S(W,kappa1,kappa2)}
	|S(\widetilde{\mathcal{W}}, k_1, k_2)| 
	&= \sum_{ (\alpha_1, \alpha_2) \in \widetilde{\mathcal{W}} }  
	\sum_{\mathfrak{a}_1 \mid (\alpha_1, k_1)} \widehat{\mu}(\mathfrak{a}_1) 
	\sum_{\mathfrak{a}_2 \mid (\alpha_2, k_2)} \widehat{\mu}(\mathfrak{a}_2) \nonumber \\
	&= \sum_{\mathfrak{a}_1 \mid k_1} \sum_{\mathfrak{a}_2 \mid k_2} 
	\widehat{\mu}(\mathfrak{a}_1)\, \widehat{\mu}(\mathfrak{a}_2) 
	\left| \widetilde{\mathcal{W}}_{\mathfrak{a}_1, \mathfrak{a}_2} \right|,
\end{align}
where
\[
\widetilde{\mathcal{W}}_{\mathfrak{a}_1, \mathfrak{a}_2} := \left\{ (\alpha, \beta) \in \widetilde{\mathcal{W}} \;:\; \mathfrak{a}_1 \mid \alpha,\; \mathfrak{a}_2 \mid \beta \right\}.
\]
Thus, the problem reduces to estimating \( |\widetilde{\mathcal{W}}_{\mathfrak{a}_1, \mathfrak{a}_2}| \) for square-free ideals \( \mathfrak{a}_1 \mid P(z_1) , \mathfrak{a}_2 \mid P(z_2) \).

Suppose \( \mathfrak{a}_1 \mid \pi - 1 \) and \( \mathfrak{a}_2 \mid \pi + 1 \). If \( (\mathfrak{a}_1, \mathfrak{a}_2) \neq (1) \), then the only possible common prime divisor is \( (1 + i) \). However, if \( (1 + i) \mid \frac{\pi + 1}{2} \), then
\[
\pi \equiv -1 \pmod{2(1 + i)},
\]
contradicting the assumption that \( \pi \) is primary (i.e., \( \pi \equiv 1 \pmod{2(1 + i)} \)). Therefore, if \( \widetilde{\mathcal{W}}_{\mathfrak{a}_1, \mathfrak{a}_2} \neq \emptyset \), we must have \( \left( (1 + i)\mathfrak{a}_1,\, \mathfrak{a}_2 \right) = (1) \). \label{coprime_condition}

Next, observe that the condition
\[
\pi \equiv 1 \pmod{2(1 + i)\mathfrak{a}_1}, \quad \pi \equiv -1 \pmod{2\mathfrak{a}_2}
\]
is equivalent to
\[
\pi \equiv 1 \pmod{2(1 + i)\mathfrak{a}_1}, \quad \pi \equiv -1 \pmod{\mathfrak{a}_2}.
\]
Since \( (2(1 + i)\mathfrak{a}_1, \mathfrak{a}_2) = 1 \), by the Chinese Remainder Theorem, this is equivalent to a single congruence
\[
\pi \equiv \alpha \pmod{2(1 + i)\mathfrak{a}_1 \mathfrak{a}_2}
\]
for some \( \alpha \in \left( \mathbb{Z}[i] / 2(1 + i)\mathfrak{a}_1 \mathfrak{a}_2 \right)^\times \).

The condition that \( \pi \) is primary (i.e., \( \pi \equiv 1 \pmod{2(1 + i)} \)) ensures a unique representative among its associates. Thus, counting primary   prime elements \( \pi \) with $N(\pi)=p$ and  \( \pi \equiv \alpha \pmod{2(1 + i)\mathfrak{a}_1 \mathfrak{a}_2} \) is equivalent to counting prime ideals \( (\pi) \subset \mathbb{Z}[i] \) with $N(\pi)=p$ and  \( \pi \equiv \alpha \pmod{2(1 + i)\mathfrak{a}_1 \mathfrak{a}_2} \). This count is given by
\[
\Pi'(x; 2(1 + i) \mathfrak{a}_1 \mathfrak{a}_2, \alpha) := \#\left\{ (\pi) \subset \mathbb{Z}[i] \;\middle|\; \pi \equiv \alpha \!\!\!\!\pmod{2(1 + i)\mathfrak{a}_1 \mathfrak{a}_2},\; N(\pi) = p \leq x \right\}.
\]
We apply Proposition~\ref{Bombierri-number field_3} to ideals of the form
\[
I = 2(1 + i)\mathfrak{a}, \quad \text{with } \mathfrak{a} \text{ square-free},
\]
and observe that
\[
\frac{|\mathbb{Z}[i]^\times|}{\Phi(2(1 + i)\mathfrak{a})} = 
\begin{cases}
	\frac{1}{\Phi(\mathfrak{a})}, & \text{if } (1 + i) \nmid \mathfrak{a}, \\[5pt]
	\frac{1}{2\Phi(\mathfrak{a})}, & \text{if } (1 + i) \mid \mathfrak{a}.
\end{cases}
\]

We define \( \widetilde{g} \) to be the multiplicative function on ideals in \( \mathbb{Z}[i] \), supported on square-free ideals, given by
\[
\widetilde{g}(1 + i) = \frac{1}{2}, \qquad 
\widetilde{g}(\mathfrak{p}) = \frac{1}{\Phi(\mathfrak{p})} \quad \text{for prime ideals } \mathfrak{p} \neq (1 + i).
\]
Specifically,
\[
\widetilde{g}(\mathfrak{p}) = 
\begin{cases}
	\frac{1}{2}, & \text{if } \mathfrak{p} = (1 + i), \\[5pt]
	\frac{1}{p - 1}, & \text{if } \mathfrak{p} = (\pi) \text{ with } N(\pi) = p \text{ a splitting prime}, \\[5pt]
	\frac{1}{q^2 - 1}, & \text{if } \mathfrak{p} = (q) \text{ with } q \text{ an inert prime}.
\end{cases}
\]
Let 
\begin{equation}\label{def_R}
	R(\mathfrak{a}) := \max_{\alpha \in \left( \mathbb{Z}[i] / 2(1+i)\mathfrak{a} \right)^\times} 
	\left| \Pi'(x; 2(1 + i)\mathfrak{a}, \alpha) - \widetilde{g}(\mathfrak{a}) \operatorname{li}(x) \right|.
\end{equation}
Applying Proposition~\ref{Bombierri-number field_3}, for any \( A > 0 \), there exists \( B = B(A) > 0 \) such that
\begin{equation}\label{error_term_original}
	\sum_{ N(\mathfrak{a}) \leq \frac{x^{1/2}}{(\log x)^B} } \widehat{\mu}^2(\mathfrak{a})
	R(\mathfrak{a}) \ll \frac{x}{(\log x)^A},
\end{equation}
where the sum is over square-free ideals \( \mathfrak{a} \subset \mathbb{Z}[i] \) with \(N(\mathfrak{a}) \leq \frac{x^{1/2}}{(\log x)^B}\).

Define the multiplicative function \( \widetilde{h}(\mathfrak{a}_1, \mathfrak{a}_2) \) by
\[
\widetilde{h}(\mathfrak{a}_1, \mathfrak{a}_2) := 
\begin{cases}
	\widetilde{g}(\mathfrak{a}_1)\, \widetilde{g}(\mathfrak{a}_2), & \text{if } ((1 + i)\mathfrak{a}_1, \mathfrak{a}_2) = (1), \\[5pt]
	0, & \text{otherwise},
\end{cases}
\]
and we can write
\begin{equation}\label{eq_tilde_W}
	\left| \widetilde{\mathcal{W}}_{\mathfrak{a}_1, \mathfrak{a}_2} \right| = \widetilde{h}(\mathfrak{a}_1, \mathfrak{a}_2)\, \operatorname{li}(x) + \widetilde{r}(\mathfrak{a}_1, \mathfrak{a}_2),
\end{equation}
with the error term bounded by
\[
|\widetilde{r}(\mathfrak{a}_1, \mathfrak{a}_2)| \leq R(\mathfrak{a}_1 \mathfrak{a}_2).
\]

For square-free integers \( d_1, d_2 \), by \eqref{W_d_1_d_2}, \eqref{eq_S(W,k_1,k_2)}, and \eqref{S(W,kappa1,kappa2)}, we have
\begin{equation}\label{W_interchange}
	| \mathcal{W}_{d_1, d_2} | = \sum_{k_1 \mid d_1} \mu(k_1) \sum_{k_2 \mid d_2} \mu(k_2) \sum_{\mathfrak{a}_1 \mid k_1} \widehat{\mu}(\mathfrak{a}_1) \sum_{\mathfrak{a}_2 \mid k_2} \widehat{\mu}(\mathfrak{a}_2) |\widetilde{\mathcal{W}}_{\mathfrak{a}_1, \mathfrak{a}_2}|.
\end{equation}
	
	For an ideal \( \mathfrak{a} \subset \mathbb{Z}[i] \), define \[
	d(\mathfrak{a}) := \text{the smallest positive integer divisible by } \mathfrak{a}.
	\]
	In other words, $d(\mathfrak{a})$ \label{def_d_a}   is the unique positive integer such that \( (d(\mathfrak{a})) = \mathfrak{a} \cap \mathbb{Z} \).  \begin{comment}
		Specifically, suppose we have the factorization of the ideal \(\mathfrak{a}\) as 
		\[
		\mathfrak{a} = (1 + i)^k \prod_{i=1}^r \left((\pi_i)^{a_{i,1}} (\overline{\pi}_i)^{a_{i,2}}\right) \prod_{j=1}^s (q_j)^{b_j},
		\]
		where \( p_i = \pi_i \overline{\pi}_i \) are distinct primes \(\equiv 1 \pmod{4}\), and \( q_i \) are distinct primes \(\equiv 3 \pmod{4}\). Then
		\[
		d(\mathfrak{a}) = 2^k \prod_{i=1}^r p_i^{\max(a_{i,1}, a_{i,2})} \prod_{j=1}^s q_j^{b_j}.
		\]
	\end{comment}
	Clearly, $\mathfrak{a} \mid k$  if and only if \( d(\mathfrak{a}) \mid k \) for any integer \( k \).
	
Rewriting \eqref{W_interchange} by interchanging the order of summation (as in Lemma~\ref{lemma_Mobius}), we obtain for square-free integers \( d_1, d_2 >0\), 
	\begin{align}\label{order_summation}
		| \mathcal{W}_{d_1, d_2} | 
		= & \sum_{k_1 \mid d_1} \mu(k_1) \sum_{k_2 \mid d_2} \mu(k_2) 
		\sum_{d(\mathfrak{a}_1) \mid k_1} \widehat{\mu}(\mathfrak{a}_1) 
		\sum_{d(\mathfrak{a}_2) \mid k_2} \widehat{\mu}(\mathfrak{a}_2) 
		|\widetilde{\mathcal{W}}_{\mathfrak{a}_1, \mathfrak{a}_2}| \nonumber \\ 
		= & \sum_{d(\mathfrak{a}_1) \mid d_1} \widehat{\mu}(\mathfrak{a}_1) 
		\sum_{n \mid \frac{d_1}{d(\mathfrak{a}_1)}} \mu(n d(\mathfrak{a}_1)) 
		\sum_{d(\mathfrak{a}_2) \mid d_2} \widehat{\mu}(\mathfrak{a}_2) 
		\sum_{n \mid \frac{d_2}{d(\mathfrak{a}_2)}} \mu(n d(\mathfrak{a}_2)) 
		|\widetilde{\mathcal{W}}_{\mathfrak{a}_1, \mathfrak{a}_2}| \nonumber \\ 
		= & \sum_{d(\mathfrak{a}_1) = d_1} \sum_{d(\mathfrak{a}_2) = d_2} 
		\widehat{\mu}(\mathfrak{a}_1\mathfrak{a}_2) \mu(d(\mathfrak{a}_1)) \mu(d(\mathfrak{a}_2)) 
		|\widetilde{\mathcal{W}}_{\mathfrak{a}_1, \mathfrak{a}_2}|.
	\end{align}
In the last step, we used the fact that \( \widehat{\mu}(\mathfrak{a}_1 \mathfrak{a}_2) = \widehat{\mu}(\mathfrak{a}_1)\widehat{\mu}(\mathfrak{a}_2) = 1 \) whenever \( (\mathfrak{a}_1, \mathfrak{a}_2) = (1) \), and that \( |\widetilde{\mathcal{W}}_{\mathfrak{a}_1, \mathfrak{a}_2}| = 0 \) when \( (\mathfrak{a}_1, \mathfrak{a}_2) \neq (1) \).

By \eqref{eq_tilde_W} and \eqref{order_summation}, we have  
\begin{equation}\label{equation_W_{d1,d2}}
	|\mathcal{W}_{d_1, d_2}| = \operatorname{li}(x) \, h(d_1, d_2) + r(d_1, d_2),
\end{equation}
where  
\begin{align}
	h(d_1, d_2) &= \sum_{\substack{d(\mathfrak{a}_1) = d_1 \\ d(\mathfrak{a}_2) = d_2}} \widehat{\mu}(\mathfrak{a}_1 \mathfrak{a}_2)\, \mu(d(\mathfrak{a}_1)) \mu(d(\mathfrak{a}_2))\, \widetilde{h}(\mathfrak{a}_1, \mathfrak{a}_2), \nonumber \\
	r(d_1, d_2) &= \sum_{\substack{d(\mathfrak{a}_1) = d_1 \\ d(\mathfrak{a}_2) = d_2}} \widehat{\mu}(\mathfrak{a}_1 \mathfrak{a}_2)\, \mu(d(\mathfrak{a}_1)) \mu(d(\mathfrak{a}_2))\, \widetilde{r}(\mathfrak{a}_1, \mathfrak{a}_2). \label{eq_R}
\end{align}

The function \( h(d_1, d_2) \) is multiplicative and supported on pairs of square-free integers. Its values at primes are given by:
\begin{equation}\label{def:h_values}
	\begin{aligned}
		h(p, 1) &= 
		\begin{cases}
			\frac{1}{2}, & \text{if } p = 2, \\
			\frac{2}{p - 1} - \frac{1}{(p - 1)^2}, & \text{if } p \equiv 1 \pmod{4}, \\
			\frac{1}{p^2 - 1}, & \text{if } p \equiv 3 \pmod{4},
		\end{cases} \\
		h(1, p) &= 
		\begin{cases}
			0, & \text{if } p = 2, \\
			\frac{2}{p - 1} - \frac{1}{(p - 1)^2}, & \text{if } p \equiv 1 \pmod{4}, \\
			\frac{1}{p^2 - 1}, & \text{if } p \equiv 3 \pmod{4},
		\end{cases} \\
		h(p, p) &=
		\begin{cases}
			0, & \text{if } p = 2 \text{ or } p \equiv 3 \pmod{4}, \\
			\frac{2}{(p - 1)^2}, & \text{if } p \equiv 1 \pmod{4}.
		\end{cases}
	\end{aligned}
\end{equation}

\begin{comment}
	One can also verify the above expressions before performing the change of summation in \eqref{order_summation}. This change greatly simplifies the expression.
\end{comment}

 \subsection{Estimation of   $r(d_1,d_2)$}

To apply the vector sieve in Proposition~\ref{prop_vector sieve}, we need to bound the error term
\[
\sum_{\substack{d_1 d_2 < D^{1 + \epsilon} \\ d_1 \mid P(z_1),\, d_2 \mid P(z_2)}} \tau(d_1 d_2)^4\, |r(d_1, d_2)|.
\]

We begin by deriving a twisted version of \eqref{error_term_original}, using a standard application of the Cauchy--Schwarz inequality.

\begin{lemma}\label{lemma_error_term_1}
	Let \( R(\mathfrak{a}) \) be as defined in \eqref{def_R}.  
	For any \( A > 0 \), and integers \( k, r \geq 0 \), there exists a constant \( B = B(A, k, r) > 0 \) such that  
	\begin{equation}\label{lemma_1_eq}
		\sum_{N(\mathfrak{a}) \leq \frac{x^{1/2}}{(\log x)^B}} 
		\widehat{\mu}^2(\mathfrak{a})\,
		\tau(N(\mathfrak{a}))^k\, 2^{r \omega(N(\mathfrak{a}))}\,
		R(\mathfrak{a}) 
		\ll_{k,r } \frac{x}{(\log x)^A}.
	\end{equation}
\end{lemma}

	\begin{proof}
First, observe that for any square-free ideal \( \mathfrak{a} \) with \( N(\mathfrak{a}) \leq x^{1/2} \), we have
\[
R(\mathfrak{a}) \ll \frac{x}{\Phi(\mathfrak{a})} \leq \frac{2^{2\omega(N(\mathfrak{a}))}}{N(\mathfrak{a})} x,
\]
where the second inequality uses the bound \( \Phi(\mathfrak{p}) \geq \frac{1}{2} N(\mathfrak{p}) \) for any prime ideal \( \mathfrak{p} \) in \( \mathbb{Z}[i] \).

Since the norm \( N(\mathfrak{a}) \) of a square-free ideal \( \mathfrak{a} \) is cube-free, we obtain
\[	\sum_{N(\mathfrak{a}) \leq \frac{x^{1/2}}{(\log x)^B}} 
\widehat{\mu}^2(\mathfrak{a})\, 
\tau(N(\mathfrak{a}))^{2k}\, 2^{2r \omega(N(\mathfrak{a}))}\, 
R(\mathfrak{a})  
\ll x \sum_{\substack{d < x \\ d \text{ cube-free}}}
\frac{\tau(d)^{2k+1} 2^{(2r+2)\omega(d)}}{d},\]
where the extra factor \( \tau(d) \) accounts for the number of ideals \( \mathfrak{a} \subset \mathbb{Z}[i] \) with \( N(\mathfrak{a}) = d \).

We estimate the sum on the right via an Euler product:
\begin{align*} 
x\sum_{\substack{d < x \\ d \text{ cube-free}}}
	\frac{\tau(d)^{2k+1} 2^{(2r+2)\omega(d)}}{d}
	&\leq x\prod_{p < x} \left( 1 + \frac{2^{2(k+r)+3}}{p} + \frac{3^{2k+1} 2^{2r+2}}{p^2} \right) \nonumber \\
	&\leq x \prod_{p < x} \left( 1 + \frac{3^{2k+2} 2^{2r+1}}{p} \right)
	\ll x(\log x)^{3^{2k+2} 2^{2r+1}}.
\end{align*}

Combining this with \eqref{error_term_original}, and applying the Cauchy–Schwarz inequality to \eqref{lemma_1_eq}, completes the proof.
\end{proof}
	
	\begin{comment}
	\begin{rem}
	Note that for \( N(\mathfrak{a}) = 2^e m \) with \( (2, m) = 1 \), we must have \( m \equiv 1 \pmod{4} \). Indeed, in later applications, we only need to sum over \( N(\mathfrak{a}) \) of the form \( N(\mathfrak{a}) = 2^e h l^2 \), where \( e = 0, 1 \), \( h \mid Q_1(x) \), and \(\ell\mid Q_2(x) \). Here, we prove a slightly stronger version for the sum of the error terms, where we sum over \( N(\mathfrak{a}) = 2^e m \) with \( m \equiv 1 \pmod{4} \).
	\end{rem}
	\end{comment}

Before proceeding with the estimate of the level of distribution, we include the following observation to clarify the relationship between the norm \( N(\mathfrak{a}) \) of a square-free ideal \( \mathfrak{a} \subset \mathbb{Z}[i] \) and its associated integer \( d(\mathfrak{a}) \).

  Let \( d > 0 \) be square-free with  
\[
d = 2^e \prod_{k=1}^{m} p_{k} \prod_{l=1}^{n} q_{l}, \quad e \in \{0,1\},
\]
where the \( p_k \) are distinct primes with \( p_k \equiv 1 \pmod{4} \), and the \( q_l \) are distinct primes with \( q_l \equiv 3 \pmod{4} \).

For any square-free ideal \( \mathfrak{a} \subset \mathbb{Z}[i] \) with \( d(\mathfrak{a}) = d \), a generator of \( \mathfrak{a} \) (up to units) must be of the form
\[
(1+i)^e \prod_{k=1}^{m} \pi_k^{e_k} \overline{\pi}_k^{f_k} \prod_{l=1}^{n} q_l,
\]
where \( 0 \leq e_k, f_k \leq 1 \) and \( 1 \leq e_k + f_k \leq 2 \).

Thus, the norm of such an ideal satisfies
\begin{equation}\label{bound_N(a)}
	d \leq N(\mathfrak{a}) \leq d^2,
\end{equation}
with equality on the left when all \( e_k + f_k = 1 \) and \( d \) has no prime factor  \( \equiv 3 \pmod{4} \), and equality on the right when all \( e_k + f_k = 2 \) and \( e = 0 \).

	\begin{lemma}\label{lemma2}
		Let \( r(d_1, d_2) \) be the error term given by \eqref{eq_R}. For any \( A > 0 \), there exists \( B = B(A) \) such that 
		\begin{equation}\label{lemma_error_2}
			\sum_{  (d_1 d_2)^2 \leq \frac{x^{1/2}}{(\log x)^B}} \mu^2(d_1) \mu^2(d_2)\tau(d_1 d_2)^4 |r(d_1, d_2)| 
			\ll \frac{x}{(\log x)^A}.
		\end{equation}
	\end{lemma}
	
\begin{proof}
	By \eqref{eq_R}, for \( d_1 \mid P(z_1) \) and \( d_2 \mid P(z_2) \), we have
\[	r(d_1, d_2) \ll 
\sum_{\substack{d(\mathfrak{a}_1) = d_1 \\ d(\mathfrak{a}_2) = d_2}} 
\widehat{\mu}^2(\mathfrak{a}_1 \mathfrak{a}_2)\, R(\mathfrak{a}_1 \mathfrak{a}_2).\]
	
	Therefore,
	\begin{align*}
		\sum_{\substack{d_1, d_2 \text{ square-free} \\ (d_1 d_2)^2 \leq \frac{x^{1/2}}{(\log x)^B}}}    
		\tau(d_1 d_2)^4\, |r(d_1, d_2)| 
		&\ll \sum_{\substack{d_1, d_2 \text{ square-free} \\ (d_1 d_2)^2 \leq \frac{x^{1/2}}{(\log x)^B}}} 
		\tau(d_1 d_2)^4  
		\sum_{\substack{d(\mathfrak{a}_1) = d_1\\ d(\mathfrak{a}_2) = d_2}}
		\widehat{\mu}^2(\mathfrak{a}_1 \mathfrak{a}_2)\, R(\mathfrak{a}_1 \mathfrak{a}_2)  \\[5pt]
		&\leq \sum_{\substack{d_1, d_2 \text{ square-free} \\ (d_1 d_2)^2 \leq \frac{x^{1/2}}{(\log x)^B}}} 
		\sum_{\substack{d(\mathfrak{a}_1) = d_1\\ d(\mathfrak{a}_2) = d_2}}
		\widehat{\mu}^2(\mathfrak{a}_1 \mathfrak{a}_2)\, 
		\tau(N(\mathfrak{a}_1 \mathfrak{a}_2))^4\, R(\mathfrak{a}_1 \mathfrak{a}_2) \\[5pt]
		&\leq \sum_{d(\mathfrak{a})^2 \leq \frac{x^{1/2}}{(\log x)^B}} 
		\widehat{\mu}^2(\mathfrak{a})\, 
		\tau(N(\mathfrak{a}))^4\, 2^{2\omega(N(\mathfrak{a}))} R(\mathfrak{a}) \\[5pt]
		&\leq \sum_{N(\mathfrak{a}) \leq \frac{x^{1/2}}{(\log x)^B}} 
		\widehat{\mu}^2(\mathfrak{a})\, 
		\tau(N(\mathfrak{a}))^4\, 2^{2\omega(N(\mathfrak{a}))} R(\mathfrak{a}),
	\end{align*}
where in the third inequality, we use that the number of factorizations \( \mathfrak{a} = \mathfrak{a}_1 \mathfrak{a}_2 \) is at most \( 2^{2\omega(N(\mathfrak{a}))} \).
	
	The result now follows from Lemma~\ref{lemma_error_term_1}.
\end{proof}
\subsection{Estimation of the sieving functions}
Hence, in Proposition~\ref{prop_vector sieve}, the remainder term is sufficiently bounded by taking \( D = x^{1/4 - \epsilon} \) for some small \( \epsilon > 0 \). We are now in a position to apply the vector sieve to the sieving functions.

\begin{prop}\label{prop_W_1}
	Let \( \epsilon > 0 \) be arbitrarily small, and set \( \alpha = \frac{1}{4} - \epsilon \). Suppose \( 0 < \theta_2 \leq \theta_1 < \frac{1}{2} \), with \( \theta_1 < \delta_1 < \alpha \) and \( \theta_2 < \delta_2 < \alpha \). Define \[
	z_1 = x^{\theta_1}, \quad z_2 = x^{\theta_2}, \quad y_1 = x^{\delta_1}, \quad y_2 = x^{\delta_2}.
	\]
	Then, for sufficiently large \( x \), we have:
	\begin{align*}
		|S(\mathcal{W}; z_1, z_2)| 
		&\geq \left(C + o(1)\right) \operatorname{li}(x) V(z_1) V(z_2) f\left( \frac{\alpha}{\theta_1}, \frac{\alpha}{\theta_2} \right), \\
		\sum_{z_1 \leq p \leq y_1} w_p(y_1) \left| S(\mathcal{W}_{p,1}; z_1, z_2) \right| 
		&\leq \left(C + o(1)\right) \operatorname{li}(x) V(z_1) V(z_2) I_1(\delta_1, \theta_1, \theta_2), \\
		\sum_{z_2 \leq p \leq y_2} w_p(y_2) \left| S(\mathcal{W}_{1,p}; z_1, z_2) \right| 
		&\leq \left(C + o(1)\right) \operatorname{li}(x) V(z_1) V(z_2) I_2(\delta_2, \theta_1, \theta_2),
	\end{align*}
	where
	\[
	C = \frac{1}{2} \prod_{p \equiv 3 \pmod{4}} \left(1 - \frac{2}{p^2 - 1}\right) 
	\prod_{p \equiv 1 \pmod{4}} \left(1 - \frac{3p - 1}{(p - 1)^3} \right),
	\]
	\[
	V(z) = \prod_{\substack{p < z \\ p \equiv 1 \pmod{4}}} \left(1 - \frac{1}{p} \right)^2 \sim \frac{c^2}{\log z}, \quad \text{for some constant } c > 0,
	\]
	\[
	I_i(\delta_i, \theta_1, \theta_2) 
	= \int_{\theta_i}^{\delta_i} \left(1 - \frac{t}{\delta_i} \right) 
	F\left( \frac{\alpha - t}{\theta_1}, \frac{\alpha - t}{\theta_2} \right) \frac{dt}{t}, 
	\quad \text{for } i = 1, 2,
	\]
	and the functions \( F \) and \( f \) are the upper and lower bound functions for the vector sieve, given in Proposition~\ref{prop_vector sieve}.
\end{prop}

	\begin{proof}
	By Lemma~\ref{lemma2}, the error term in \eqref{equation_W_{d1,d2}} is sufficiently bounded if we take \( D = x^{1/4 - \epsilon} \). The   function \( h(d_1,d_2) \), given in \eqref{def:h_values}, satisfies the condition~\eqref{condition_h}, and both \( h_1(d) := h(d,1) \) and \( h_2(d) := h(1,d) \) satisfy the linear sieve condition~\eqref{linear_sieve_condition}, as can be verified via Mertens' estimates.
	
	From \eqref{Def: hstar}, we have
	\[
	h^*(p) = h(1, p) + h(p, 1) - h(p, p).
	\]
	Using the values of \( h \) from \eqref{def:h_values} and the definitions in \eqref{definition_Vi}, the Euler factors in \( V(z_0, h^*) V_1 V_2 \) are computed as follows:
	\[
	\begin{cases}
		\frac{1}{2}, & \text{if } p = 2, \\[6pt]
		1 - \frac{2}{p^2 - 1}, & \text{if } 2 < p < z_0,\; p \equiv 3 \pmod{4}, \\[6pt]
		\left(1 - \frac{2}{p - 1}\right)^2 = \left(1 - \frac{3p - 1}{(p - 1)^3} \right) \left(1 - \frac{1}{p}\right)^4, & \text{if } 2 < p < z_0,\; p \equiv 1 \pmod{4}, \\[6pt]
		\left(1 - \frac{1}{p^2 - 1}\right)^2, & \text{if } z_0 \leq p < z_2,\; p \equiv 3 \pmod{4}, \\[6pt]
		\left(1 - \frac{1}{p - 1}\right)^4 = \left(1 - \frac{1}{(p - 1)^2}\right)^4 \left(1 - \frac{1}{p}\right)^4, & \text{if } z_0 \leq p < z_2,\; p \equiv 1 \pmod{4}, \\[6pt]
		1 - \frac{1}{p^2 - 1}, & \text{if } z_2 \leq p < z_1,\; p \equiv 3 \pmod{4}, \\[6pt]
		\left(1 - \frac{1}{p - 1}\right)^2 = \left(1 - \frac{1}{(p - 1)^2}\right)^2 \left(1 - \frac{1}{p}\right)^2, & \text{if } z_2 \leq p < z_1,\; p \equiv 1 \pmod{4}.
	\end{cases}
	\]
	
Thus, \( V(z_0, h^*) V_1 V_2 \sim C V(z_1) V(z_2) \), where
\[
C := \frac{1}{2} \prod_{p \equiv 3 \pmod{4}} \left(1 - \frac{2}{p^2 - 1}\right) 
\prod_{p \equiv 1 \pmod{4}} \left(1 - \frac{3p - 1}{(p - 1)^3} \right),
\]
and
\[
V(z) := \prod_{\substack{p < z \\ p \equiv 1 \pmod{4}}} \left(1 - \frac{1}{p} \right)^2 
\sim \frac{c^2}{\log z},
\]
for some constant \( c > 0 \), by Mertens' theorem in progression. 

   Applying Proposition~\ref{prop_vector sieve} to \( \mathcal{W} \) with $D=x^\alpha$, we obtain 
\[
|S(\mathcal{W}, z_1, z_2) |
\geq \left(C + o(1)\right) \operatorname{li}(x)\, V(z_1)\, V(z_2)\, 
f\left(\frac{\alpha}{\theta_1}, \frac{\alpha}{\theta_2}\right).
\]

	Next, we apply the upper bound vector sieve to the set \( \mathcal{W}_{p,1} \) for \( z_1 \leq p \leq y_1 = x^{\delta_1} \).
	
	Since \( d_1 \mid P(z_1) \), \( d_2 \mid P(z_2) \), and \( p \geq z_1 \), we have \( (p, d_1) = 1 \), so
	\[
	h(pd_1, d_2) = h(p,1)\, h(d_1, d_2).
	\]
Applying Proposition~\ref{prop_vector sieve} to \( \mathcal{W}_{p,1} \) with \( D = x^{1/4 - \epsilon} / p \), we obtain
\[
|S(\mathcal{W}_{p,1}, z_1, z_2) |
\leq \left(C + o(1)\right) \operatorname{li}(x)\, V(z_1)\, V(z_2)\, h(p, 1)\, F\left(s_1(p), s_2(p)\right),
\]
where
\[
s_i(p) := \frac{\log \left(x^{1/4 - \epsilon} / p \right)}{\log z_i}, \quad i = 1, 2.
\]
Summing over \( z_1 \leq p \leq y_1 \), we obtain
\[
\sum_{z_1 \leq p \leq y_1} w_p(y_1)\, \left| S(\mathcal{W}_{p,1}, z_1, z_2) \right|
\leq (C+o(1)) \operatorname{li}(x) V(z_1) V(z_2) \left(\sum_{z_1 \leq p \leq y_1} \left(1 - \frac{\log p}{\log y_1}\right)h(p, 1) F\left(s_1(p), s_2(p)\right)\right)  .
\]
Define
\[
H(w) := \sum_{p \leq w} h(p,1) = \log \log w + c_1 + O\left( \frac{1}{\log w} \right) \quad \text{for some constant } c_1,
\]
and set
\[
\sigma_i(w) := \frac{\log (x^{1/4 - \epsilon} / w)}{\log z_i}, \quad i = 1, 2.
\]
Then we have
\begin{align*}
	&\sum_{z_1 \leq p \leq y_1} w_p(y_1) \left| S(\mathcal{W}_{p,1}, z_1, z_2) \right|\\
	\leq &\left(C + o(1)\right) \operatorname{li}(x) V(z_1) V(z_2)
	\int_{z_1}^{y_1} \left(1 - \frac{\log w}{\log y_1} \right)
	F(\sigma_1(w), \sigma_2(w)) \, d(H(w)) \\
	= &\left(C + o(1)\right) \operatorname{li}(x) V(z_1) V(z_2)
	\int_{z_1}^{y_1} \left(1 - \frac{\log w}{\log y_1} \right)
	F(\sigma_1(w), \sigma_2(w)) \, d(\log \log w) \\
	&\quad + O\left(\frac{\max\left(F(\sigma_1(y_1), \sigma_2(y_1)), F(\sigma_1(y_2), \sigma_2(y_2))\right)}{\log z_1}\right),
\end{align*}
where the \( O \)-term is of size \( O\left(\frac{1  }{\log z_1}\right) \), since the function \( F  \) is bounded. 
Let \( \alpha = \frac{1}{4} - \epsilon \). We change variables in the integral to obtain
\[
\int_{z_1}^{y_1} \left(1 - \frac{\log w}{\log y_1}\right) 
F(\sigma_1(w), \sigma_2(w)) \frac{dw}{w \log w}
= \int_{\theta_1}^{\delta_1} \left(1 - \frac{t}{\delta_1} \right)
F\left( \frac{\alpha - t}{\theta_1}, \frac{\alpha - t}{\theta_2} \right)
\frac{dt}{t}.
\]

The estimate for 
\[
\sum_{z_2 \leq p \leq y_2} w_p(y_2) \left| S(\mathcal{W}_{1,p}, z_1, z_2) \right|
\]
follows similarly.
	\end{proof}

	Recall that 
	\[
	W(x ) := |S(\mathcal{W}; z_1, z_2)| - \lambda \sum_{z_1 \leq p \leq y_1} w_p(y_1) |S(\mathcal{W}_{p,1}, z_1, z_2)| 
	- \lambda \sum_{z_2 \leq p \leq y_2} w_p(y_2) |S(\mathcal{W}_{1,p}, z_1, z_2)|.
	\]
By Propositions~\ref{prop_W_1},   for sufficiently large \( x \),
\[
W(x) \geq H(\lambda, \delta_1, \delta_2, \theta_1, \theta_2) \, C \, V(x^{\theta_1}) V(x^{\theta_2}) \operatorname{li}(x),
\]
where
\[
H(\lambda, \delta_1, \delta_2, \theta_1, \theta_2) = f\left(\frac{\alpha}{\theta_1}, \frac{\alpha}{\theta_2}\right)
- \lambda \left( I_1(\delta_1, \theta_1, \theta_2) + I_2(\delta_2, \theta_1, \theta_2) \right),
\quad \text{with } \alpha = \frac{1}{4} - \epsilon.
\]
We choose parameters such that
\[
H(\lambda, \delta_1, \delta_2, \theta_1, \theta_2) > 0
\quad \text{and} \quad
\frac{1}{\lambda} + \frac{1}{\delta_1} + \frac{1}{\delta_2} \text{ is as small as possible.}
\]	
Taking \( \alpha = \frac{1}{4.01} \), we find
	\[
	H\left(\frac{1}{2.4}, \frac{1}{4.2}, \frac{1}{4.3}, \frac{1}{30}, \frac{1}{31}\right) = 0.1274,
	\]
	which implies that in ~\eqref{omega(a)+omega(b)},
\[
\omega(a) + \omega(b) \leq \left\lfloor
\frac{1}{\lambda} + \frac{1}{\delta_1} + \frac{1}{\delta_2}
\right\rfloor = 10, \quad z_1 = x^{1/30}, \quad z_2 = x^{1/31}.
\]
	This completes the proof of the first part of Theorem~\ref{theorem2_intro}.

To prove the second part of the theorem, it suffices to show that the contribution from those primes \( p \leq x \) counted in \( W(x) \) for which
\(\left( \frac{|E(\mathbb{F}_p)|}{8}, \frac{|E(\mathbb{F}_{p^2})|}{4|E(\mathbb{F}_p)|} \right) \neq 1\)
is negligible.

Suppose that
\[
q \biggm| \left( \frac{|E(\mathbb{F}_p)|}{8}, \frac{|E(\mathbb{F}_{p^2})|}{4|E(\mathbb{F}_p)|} \right)
\]
for some prime \( q > z_2 = x^{\theta_2} \). Such a prime must satisfy \( q \equiv 1 \pmod{4} \), since
\[
\left( \frac{\pi - 1}{2(1+i)}, \frac{\pi + 1}{2} \right) = (1),
\]
as noted on p.~\pageref{coprime_condition}.

Write \( q = \pi_q \overline{\pi_q} \) in \( \mathbb{Z}[i] \), and without loss of generality, assume
\begin{equation}\label{divide_condition}
	\pi_q \mid \pi - 1 \quad \text{and} \quad \overline{\pi_q} \mid \pi + 1, \quad \text{where } N(\pi) = p,
\end{equation}
which   implies
\[
\overline{\pi_q} \mid \overline{\pi} - 1 \quad \text{and} \quad \pi_q \mid \overline{\pi} + 1.
\]
Note that \( \pi \) is a root of the equation \( x^2 - a_p x + p = 0 \), so we have
\[
\pi_q\mid (\pi-1)+(\pi+1) =a_p .
\]
Since \( \operatorname{Re}(\pi) = \frac{a_p}{2} \in \mathbb{Z} \), and by Hasse’s bound \( |a_p| \leq 2\sqrt{p} \), we conclude that \( q \leq \sqrt{p} \).

By the Chinese Remainder Theorem, there exists a unique \( \alpha \in (\mathbb{Z}[i]/\pi_q \overline{\pi_q})^\times \) such that~\eqref{divide_condition} is equivalent to
\[
\pi \equiv \alpha \pmod{\pi_q \overline{\pi_q}}.
\]
We may choose a representative \( \alpha \in \mathbb{Z}[i] \) with \( |\alpha| < |\pi_q \overline{\pi_q}| = q \). Therefore,
\[
q^2 \mid N(\pi - \alpha) \leq (|\pi| + |\alpha|)^2 \leq (2\sqrt{x})^2 = 4x.
\]
Hence, for each such \( q \), the number of primes \( p \) satisfying
\[
q \mid \frac{|E(\mathbb{F}_p)|}{8} \quad \text{and} \quad q \mid \frac{|E(\mathbb{F}_{p^2})|}{4|E(\mathbb{F}_p)|}
\]
is bounded by the number of prime elements \( \pi \in \mathbb{Z}[i] \) such that \( N(\pi - \alpha) \leq 4x \) and \( q^2 \mid N(\pi - \alpha) \). 

Since \( \alpha \) is fixed for each \( q \), this count is bounded by the number of elements \( \beta \in \mathbb{Z}[i] \) with \( N(\beta) \leq 4x \) and \( q^2 \mid N(\beta) \), which is
\[
\ll \sum_{\substack{n \leq 4x \\ q^2 \mid n}} \tau(n).
\]

Therefore, summing over \( x^{\theta_2} < q \leq x^{1/2} \), the total number of such primes \( p \) is bounded by
\[
\sum_{x^{\theta_2} < q  \leq  x^{1/2}} \sum_{\substack{n \leq 4x \\ q^2 \mid n}} \tau(n)
\ll \sum_{x^{\theta_2}   \leq  x^{1/2}} \frac{x}{q^2} \log \frac{x}{q^2}
\ll x^{1 - \theta_2} \log x.
\]
This shows that the contribution from such primes is negligible, thereby completing the proof of Theorem~\ref{theorem1_intro}.

We now restate Theorem~\ref{theorem1_intro} below for convenience.

\begin{theorem}[Theorem~\ref{theorem1_intro}]\label{almost_triple}
	Let \( E/\mathbb{Q} \) be the elliptic curve defined by \( y^2 = x^3 - x \). Then, for sufficiently large \( x \), we have
	\[
	\left\{ 
	p \leq x \;\middle|\; 
	p \equiv 1 \pmod{4},\ 
	\frac{|E(\mathbb{F}_p)|}{8} = P_k(x^{1/30}),\ 
	\frac{|E(\mathbb{F}_{p^2})|}{4|E(\mathbb{F}_p)|} = P_{10-k}(x^{1/31}),\ 
	1 \leq k \leq 9 
	\right\}
	\gg \frac{x}{(\log x)^3},
	\]
	where \( P_m(z) \) denotes a product of at most \( m \) distinct primes, each   \( >z \).
	
	Moreover, one can additionally require that
	\[
	\left( 
	\frac{|E(\mathbb{F}_p)|}{8},\ 
	\frac{|E(\mathbb{F}_{p^2})|}{4|E(\mathbb{F}_p)|} 
	\right) = 1,
	\]
	in which case we obtain
\[		\left\{ 
p \leq x \;\middle|\; 
p \equiv 1 \pmod{4},\ 
\frac{|E(\mathbb{F}_{p^2})|}{32} = P_{10} 
\right\}
\gg \frac{x}{(\log x)^3},\]
	where \( P_{10} \) denotes a product of at most 10 distinct primes.
\end{theorem}

 \subsection{Group Structures of \( E(\mathbb{F}_p) \) and \( E(\mathbb{F}_{p^2}) \)}\label{group_structure}
 
 Theorem~\ref{almost_triple} implies a lower bound for the number of primes \( p \leq x \) for which both \( E(\mathbb{F}_p) \) and \( E(\mathbb{F}_{p^2}) \) contain a cyclic subgroup of finite index, where the index is determined by the injection of torsion points.
 
 We now analyze the \( 2 \)-primary subgroups of \( E(\mathbb{F}_p) \) and \( E(\mathbb{F}_{p^2}) \), and determine the full group structures of \( E(\mathbb{F}_p) \) and \( E(\mathbb{F}_{p^2}) \) for the primes counted in Theorem~\ref{almost_triple}.

\begin{theorem}[Theorem~\ref{theorem2_intro}]\label{theorem_group_structure}
	Let \( E/\mathbb{Q} \) be the elliptic curve given by \( y^2 = x^3 - x \). Then, for sufficiently large \( x \),
\[
\# \left\{ 
p \leq x :\,
E(\mathbb{F}_p) \cong \mathbb{Z}/2  \times \mathbb{Z}/4m ,\ 
E(\mathbb{F}_{p^2}) \cong \mathbb{Z}/4  \times \mathbb{Z}/8mn ,\ \mu (mn) \neq 0 
,\ \omega(m n) \leq 10
\right\}
\gg \frac{x}{(\log x)^3}.
\]

\end{theorem}
\begin{proof}
Suppose \( |E(\mathbb{F}_p)| = 8 P_k \) and \( \frac{|E(\mathbb{F}_{p^2})|}{|E(\mathbb{F}_p)|} = 4 P_{10-k} \), with \( (P_k, P_{10-k}) = 1 \), as in the second part of Theorem~\ref{almost_triple}. 

We determine the group structures of \( E(\mathbb{F}_p) \) and \( E(\mathbb{F}_{p^2}) \) by analyzing their \( 2 \)-primary components.

For any prime \( p \neq 2 \), it follows from \cite[Proposition~VII.3.1]{silverman} that there is an injection
\[
E(\mathbb{Q})_{\mathrm{tors}} \cong \mathbb{Z}/2\mathbb{Z} \times \mathbb{Z}/2\mathbb{Z} \hookrightarrow E(\mathbb{F}_p),
\]
where \( E(\mathbb{Q})_{\mathrm{tors}} \) consists of the points \( \infty \), \( (0, 0) \), and \( (\pm 1, 0) \).
Moreover, we have
\[
E(\mathbb{Q}(i))_{\mathrm{tors}} \cong \mathbb{Z}/2\mathbb{Z} \times \mathbb{Z}/4\mathbb{Z},
\]
consisting of the points \( (i, 1-i) \), \( (i, -1+i) \), \( (-i, 1+i) \), \( (-i, -1-i) \), together with the four   points in \( E(\mathbb{Q})_{\mathrm{tors}} \).
For any prime \( p \equiv 1 \pmod{4} \), again by \cite[Proposition~VII.3.1]{silverman}, we have for \( p \neq 2 \) and \( (n, p) = 1 \) an injection
\[
E(\mathbb{Q}(i))_{\mathrm{tors}}[n] \hookrightarrow E(\mathbb{Q}_p)[n] \hookrightarrow E(\mathbb{F}_p),
\]
thus,
\[
E(\mathbb{Q}(i))_{\mathrm{tors}} \cong \mathbb{Z}/2\mathbb{Z} \times \mathbb{Z}/4\mathbb{Z} \hookrightarrow E(\mathbb{F}_p).
\]

Since \( |E(\mathbb{F}_p)| = 8m \), where \( m \) is an odd square-free integer, the full group structure must be
\[
E(\mathbb{F}_p) \cong \mathbb{Z}/2\mathbb{Z} \times \mathbb{Z}/4m\mathbb{Z}.
\]

We now consider the structure of the subgroup of order \( 32 \) in \( E(\mathbb{F}_{p^2}) \).

Consider the unramified extension \( \mathbb{Q}_p(\zeta_{p^2 - 1}) / \mathbb{Q}_p \), whose residue field is \( \mathbb{F}_{p^2}/\mathbb{F}_p \). By \cite[Prop.~VII.3.1]{silverman}, for \( p \neq 2 \) and \( (n, p) = 1 \), we have
\[
E(\mathbb{Q}(\zeta_{p^2 - 1}))_{\mathrm{tors}}[n] \hookrightarrow E(\mathbb{Q}_p(\zeta_{p^2 - 1}))_{\mathrm{tors}}[n] \hookrightarrow E(\mathbb{F}_{p^2}).
\]
For \( p \equiv 1 \pmod{4} \), we have \( 8 \mid (p^2 - 1) \), so
\[
E(\mathbb{Q}(\zeta_8))_{\mathrm{tors}} \cong \mathbb{Z}/4\mathbb{Z} \times \mathbb{Z}/4\mathbb{Z} \hookrightarrow E(\mathbb{F}_{p^2}),
\]
where the structure of \( E(\mathbb{Q}(\zeta_8))_{\mathrm{tors}} \) can be verified using Magma.
Since the group of points on a reduced elliptic curve over a finite field is always of the form \( \mathbb{Z}/d \mathbb{Z}\times \mathbb{Z}/e\mathbb{Z} \) with \( d \mid e \), the \( 2 \)-primary subgroup of \( E(\mathbb{F}_{p^2}) \) must be
\[
\mathbb{Z}/4 \mathbb{Z}\times \mathbb{Z}/8\mathbb{Z}.
\]
Thus, if \( |E(\mathbb{F}_{p^2})| = 32mn \) with \( mn \) an odd square-free integer, we conclude
\[
E(\mathbb{F}_{p^2}) \cong \mathbb{Z}/4 \mathbb{Z}\times \mathbb{Z}/8mn\mathbb{Z}.
\]
		\end{proof}

In the proof of Theorem~\ref{theorem_group_structure}, we used the injection  
\[
E(\mathbb{Q}(\zeta_{8}))_{\mathrm{tors}} \cong \mathbb{Z}/4\mathbb{Z} \times \mathbb{Z}/4\mathbb{Z} \hookrightarrow E(\mathbb{Q}(\zeta_{p^2 - 1}))_{\mathrm{tors}}.
\]
In fact, this map is an isomorphism.

\begin{prop}\label{prop_abelian_extension}
	Let \( E : y^2 = x^3 - x \) be the elliptic curve over \( \mathbb{Q} \), and let \( \mathbb{Q}^{\mathrm{ab}} \) denote the maximal abelian extension of \( \mathbb{Q} \). Then
	\[
	E(\mathbb{Q}^{\mathrm{ab}})_{\mathrm{tors}} \cong \mathbb{Z}/4\mathbb{Z} \times \mathbb{Z}/4\mathbb{Z}.
	\]
	In particular, for any prime \( p \equiv 1 \pmod{4} \) and any even integer \( r \geq 2 \), we have
	\[
	E(\mathbb{Q}(\zeta_{p^r - 1}))_{\mathrm{tors}} \cong \mathbb{Z}/4\mathbb{Z} \times \mathbb{Z}/4\mathbb{Z}.
	\]
\end{prop}

	\begin{proof}
	If \( r \) is even and \( p \equiv 1 \pmod{4} \), then \( 8 \mid p^r - 1 \). Thus,
	\[
	E(\mathbb{Q}(\zeta_8))_{\mathrm{tors}} \cong \mathbb{Z}/4 \mathbb{Z}\times \mathbb{Z}/4\mathbb{Z}
	\hookrightarrow E(\mathbb{Q}(\zeta_{p^r - 1}))_{\mathrm{tors}}.
	\]
	Since \( \mathbb{Q}(\zeta_{p^r - 1}) \subset \mathbb{Q}^{\mathrm{ab}} \), it suffices to show that
	\[
	E(\mathbb{Q}^{\mathrm{ab}})_{\mathrm{tors}} \cong \mathbb{Z}/4 \times \mathbb{Z}/4.
	\]
	
	By \cite[Theorem~1.2]{chou}, the possible torsion structures over \( \mathbb{Q}^{\mathrm{ab}} \) are:
	\[
	\mathbb{Z}/4\mathbb{Z} \times \mathbb{Z}/4n\mathbb{Z}, \quad n = 1, 2, 3, 4, \quad \text{and} \quad \mathbb{Z}/8\mathbb{Z} \times \mathbb{Z}/8\mathbb{Z}.
	\]
Let \( C(E) \) denote the number of \( \mathbb{Q} \)-isogeny classes of \( E \), and let \( C_p(E) \) denote its \( p \)-primary component, as defined in \cite{kenku}. For \( E : y^2 = x^3 - x \), we have \( C(E) = C_2(E) = 4 \), with isogeny degrees \( 1, 2, 2, 2 \), as verified, for example, in \cite[\href{https://www.lmfdb.org/EllipticCurve/Q/32/a/3}{Elliptic Curve 32.a3}]{lmfdb}.
	It then follows from \cite[Table~1]{chou} that
	\[
	E(\mathbb{Q}^{\mathrm{ab}})[2^\infty] \cong \mathbb{Z}/4 \mathbb{Z}\times \mathbb{Z}/4\mathbb{Z}.
	\]
	Moreover, since \( E \) admits no 3-isogeny, we can exclude the structure \( \mathbb{Z}/4\mathbb{Z} \times \mathbb{Z}/12\mathbb{Z} \) by \cite[Lemma~2.7]{chou}. Hence,
	\[
	E(\mathbb{Q}^{\mathrm{ab}})_{\mathrm{tors}} \cong \mathbb{Z}/4 \mathbb{Z}\times \mathbb{Z}/4\mathbb{Z}.
	\]

	\end{proof}

The proof of Theorem \ref{almost_triple} can be adapted to any elliptic curve with complex multiplication   to establish a lower bound of order \( \frac{x}{(\log x)^3} \) for the number of primes \( p \leq x \) such that both \( E(\mathbb{F}_p) \) and \( E(\mathbb{F}_{p^2}) \) contain a large cyclic subgroup of finite index. This index is related to the injection of torsion points and is of the form \( 2^a \cdot 3^b \) for some non-negative integers \( a \) and \( b \).

Part of the technical complexity in Theorem \ref{almost_triple} arises from the sieve method used to control the number of prime factors. However, when focusing only the existence of large cyclic subgroups, the sieve can be considerably simplified.  We now present a streamlined proof of a lower bound for the number of primes \( p \leq x \) such that both \( E(\mathbb{F}_p) \) and \( E(\mathbb{F}_{p^2}) \) contain a cyclic subgroup of  finite index. This approach is inspired by the work of Gupta and Murty, who proved an unconditional lower bound for the number of primes \( p \leq x \) such that \( E(\mathbb{F}_p) \) is cyclic \cite[Theorem 1]{gupta}.

\begin{theorem}[Theorem \ref{intro_cyclicity_lower_bound}]\label{cyclicity_lower_bound}
	Let \( E/\mathbb{Q} \) be an elliptic curve with complex multiplication and conductor \( N_E \). Then, for sufficiently large \( x \),
\[
\#\left\{ p \leq x : p \nmid N_E,\ E(\mathbb{F}_p) \text{ and } E(\mathbb{F}_{p^2}) \text{ contain a cyclic subgroup of index } \leq 12 \right\} \gg \frac{x}{(\log x)^3}.
\]
\end{theorem}

\begin{proof}
Let \( \pi_p \) denote a root of the characteristic polynomial \( x^2 + a_p x + p = 0 \), where \( a_p \) is the Frobenius trace at \( p \).
Then we have
\[
\mathbb{Z}/m\mathbb{Z} \times \mathbb{Z}/m\mathbb{Z} \subset E(\mathbb{F}_{p^2})
\quad \text{if and only if} \quad
\pi_p^2 - 1 \text{ is divisible by } m \text{ in } \operatorname{End}(E(\mathbb{F}_{p^2})).
\]
This condition in particular implies that $m\mid p^2-1 $ and $m^2\mid |E(\mathbb{F}_{p^2})|$, see \cite[Lemma 1 and 2]{elliptic_notes}.

Consider the set
\[
S_\epsilon(x) := \left\{
\begin{aligned}
	p \leq x :\ & p \equiv 5 \pmod{8},\ p \equiv 5 \pmod{9}, \\
	& \text{all odd prime factors } \ell > 3 \text{ of } p - 1 \text{ are distinct and } \ell > z_1, \\
	& \text{all odd prime factors } \ell > 3 \text{ of } p + 1 \text{ are distinct and } \ell > z_2
\end{aligned}
\right\},
\]
where \( z_1 = x^{\theta_1} \) and \( z_2 = x^{\theta_2} \) for some \( 0 < \theta_1, \theta_2 < 1 \). We take \( \theta_1 = \theta_2 = \frac{1}{8} - \epsilon \) for some \( \epsilon > 0 \), and an application of the vector sieve (Proposition~\ref{prop_vector sieve}) yields the lower bound
\[
\#S_\epsilon(x) \gg \frac{x}{(\log x)^3}.
\]
Note that the congruence conditions on \( p \) ensure that \( 8 \nmid p^2 - 1 \) and \( 9 \nmid p^2 - 1 \). Consequently, 
\[
\mathbb{Z}/2^k3^l\mathbb{Z} \times \mathbb{Z}/2^k3^l\mathbb{Z} \not\subset E(\mathbb{F}_p) \quad \text{and } \quad \mathbb{Z}/2^k3^l\mathbb{Z} \times \mathbb{Z}/2^k3^l\mathbb{Z} \not\subset E(\mathbb{F}_{p^2})
\]
for any \( k \geq 3 \) or \( l\geq 2 \).

Then we have
\begin{align*}
	&\#\left\{ p \leq x : E(\mathbb{F}_p) \text{ and } E(\mathbb{F}_{p^2}) \text{ contain a cyclic subgroup of index } \leq  12 \right\} \\
	\geq\ &\#\left\{ p \in S_\epsilon(x) : \mathbb{Z}/\ell\mathbb{Z} \times \mathbb{Z}/\ell\mathbb{Z} \not\subset E(\mathbb{F}_{p^2}) \text{ for any prime } \ell > 3 \right\} \\
	=\ &|S_\epsilon(x)| - \#\left\{ p \in S_\epsilon(x) : \mathbb{Z}/\ell\mathbb{Z} \times \mathbb{Z}/\ell\mathbb{Z} \subset  E(\mathbb{F}_{p^2}) \text{ for some prime } \ell > 3 \right\}.
\end{align*}
Any such prime $\ell$ in 
\[
\#\left\{ p \in S_\epsilon(x) : \mathbb{Z}/\ell\mathbb{Z} \times \mathbb{Z}/\ell\mathbb{Z} \subset E(\mathbb{F}_{p^2}) \text{ for some prime } \ell > 3 \right\}
\]
must satisfy 
\begin{equation}\label{divisibility_condtition_1}
\ell^2 \mid N(\pi_p^2 - 1) = (p + 1 - a_p)(p + 1 + a_p),
\end{equation}
and
 \begin{equation}\label{divisibility_condition_2}
 	\ell \mid p^2 - 1.
 \end{equation}
Since $p\in S_\epsilon(x)$, from \eqref{divisibility_condition_2}, we have 
\[
\ell > x^{1/8 - \epsilon}.
\]
On the other hand, from both \eqref{divisibility_condtition_1} and \eqref{divisibility_condition_2}, we have  that $\ell$ divides at least one of $a_p$, $a_p + 2$, or $a_p - 2$. By Hasse's bound, $|a_p| \leq 2\sqrt{p} \leq 2x^{1/2}$, so
\begin{equation}\label{bound_l}
	\ell \leq |a_p| + 2 \leq 2x^{1/2} + 2 .
\end{equation}

In the CM case, $\pi_p$ lies in the ring of integers $\mathcal{O}_K$ of the CM field $K$. Therefore, the number of elements $\alpha \in \mathcal{O}_K$ with norm $N(\alpha) = t$ is bounded by $\tau(t)$. This follows, for example, from the unique factorization of elements in $\mathcal{O}_K$.

Suppose $\ell^2 \mid N(\pi_p - 1)$ or $\ell^2 \mid N(\pi_p + 1)$. Then the number of such primes $p$ is  
\[
\ll \sum_{x^{1/8 - \epsilon} < \ell \leq 4x^{1/2}} \sum_{\substack{n \leq 4x \\ \ell^2 \mid n}} \tau(n) \ll x^{7/8 + \epsilon}\log x .
\]

Now suppose both $\ell \mid N(\pi_p - 1)$ and $\ell \mid N(\pi_p + 1)$. Since \( \ell > 3 \), the prime \( \ell \) must split in \( \mathcal{O}_K \), say \( \ell = \pi_\ell \overline{\pi_\ell} \), and we may assume
\[
\pi_\ell \mid \pi_p - 1 \quad \text{and} \quad \overline{\pi_\ell} \mid \pi_p + 1.
\]
By the Chinese Remainder Theorem, there exists a unique $\alpha \in (\mathcal{O}_K / \ell)^\times$ such that
\[
\pi_p \equiv \alpha \pmod{\ell}
\]
is equivalent to the conditions $\pi_\ell \mid \pi_p - 1$ and $\overline{\pi_\ell} \mid \pi_p + 1$.

Assume \( K = \mathbb{Q}(\sqrt{D}) \) and \( \mathcal{O}_K = \mathbb{Z}[\omega_K] \), where \( \omega_K = \sqrt{D} \) if \( D \equiv 2,3 \pmod{4} \), and \( \omega_K = \frac{1 + \sqrt{D}}{2} \) if \( D \equiv 1 \pmod{4} \).
Then we may choose representatives of $\mathcal{O}_K / \ell$ as
\[
\left\{ m + n \omega_K : -\frac{\ell + 1}{2} \leq m, n \leq \frac{\ell - 1}{2} \right\}.
\]
Using the bound \( \ell \leq 2x^{1/2} + 2 \) from~\eqref{bound_l}, it follows that for sufficiently large \( x \),
\[
N(\alpha) \leq \frac{9(\ell + 1)^2}{16} + \frac{|D|(\ell + 1)^2}{4} \leq \left(\frac{|D|}{4} + \frac{9}{16}\right)(\ell + 1)^2 \leq ( |D|+4) x .
\]
Therefore, the number of such primes $p \in S_\epsilon(x)$ is bounded by
\[
\sum_{x^{1/8 - \epsilon} < \ell \leq 4x^{1/2}} \sum_{\substack{n \leq ( |D|+4) x \\ \ell^2 \mid n}} \tau(n) \ll x^{7/8 + \epsilon} \log x.
\]
Since $\epsilon > 0$ can be taken arbitrarily small, we conclude that
\begin{align*}
&\#\left\{ p \leq x : E(\mathbb{F}_p) \text{ and } E(\mathbb{F}_{p^2}) \text{ contain a cyclic subgroup of index } \leq 12 \right\}\\  \gg& \frac{x}{(\log x)^3} + O(x^{7/8 + \epsilon} \log x)\\ \gg& \frac{x}{(\log x)^3}.
\end{align*}
\end{proof}
\begin{comment}
\begin{rem}
We note that the method used here does not directly extend to the non-CM case to yield the same unconditional lower bound.
\end{rem}
\end{comment}

	\section{Almost prime order of $\frac{|E(\mathbb{F}_{p^\ell})|}{ |E(\mathbb{F}_p)|}$ for CM elliptic curves }\label{discussion}
In this section, we consider the almost prime property of \(\frac{|E(\mathbb{F}_{p^\ell})|}{|E(\mathbb{F}_p)|},\) for a fixed prime \( \ell \geq 2 \), as motivated by Koblitz’s question (see p.~2). By combining this with an almost prime condition on \( |E(\mathbb{F}_p)| \), one can analyze the almost primality of \( |E(\mathbb{F}_{p^\ell})| \) itself.

Note that
\begin{equation}\label{order}
	|E(\mathbb{F}_{p^n})| = N(\pi_p^n - 1) = \prod_{d \mid n} N(\Phi_d(\pi_p)),
\end{equation}
where \( \Phi_d \) is the \( d \)-th cyclotomic polynomial and \( N \) denotes the complex norm, given by \( N(z) = z\overline{z} \). In the case \( n = \ell \) prime, this simplifies to
\[
|E(\mathbb{F}_{p^\ell})| = N(\pi_p^\ell - 1).
\]

	\begin{theorem}[Theorem \ref{prop_intro}]\label{higher_order}
		Let \( E/\mathbb{Q} \) be the elliptic curve \( y^2 = x^3 - x \). For any odd prime \(\ell\geq 3 \), define  
	\[
	n_\ell = \lfloor 4.1(\ell - 1) + 1.2 \rfloor.
	\]
	Then, for sufficiently large \( x \),
	\[
	\# \left\{ p \leq x \,\middle|\, p \equiv 1 \pmod{4},\, \Omega\left( \frac{|E(\mathbb{F}_{p^\ell})|}{|E(\mathbb{F}_p)|} \right) \leq n_\ell \right\} \gg \frac{x}{(\log x)^2}.
	\]	
	When \(\ell= 2 \), for sufficiently large \( x \),
	\[
	\# \left\{ p \leq x \,\middle|\, p \equiv 1 \pmod{4},\, \frac{|E(\mathbb{F}_{p^2})|}{4|E(\mathbb{F}_p)|} = P_5 \right\} \gg \frac{x}{(\log x)^2},
	\]
	where \( P_5 \) denotes a product of  at most 5 distinct  primes.
	\end{theorem}

	\begin{proof}
Let \( \ell \geq 3 \) be an odd prime. We have 
\[
\frac{|E(\mathbb{F}_{p^\ell})|}{|E(\mathbb{F}_p)|} = N\left( \Phi_\ell(\pi) \right),
\]
where \( \pi \in \mathbb{Z}[i] \) with \( N(\pi) = p \), and \( \Phi_\ell \) denotes the \( \ell \)-th cyclotomic polynomial.

We will see that this problem reduces to a linear sieve problem. This parallels the classical setting, where the sifting set arises from the values of an irreducible polynomial evaluated at rational primes; see, for example, \cite[Example~1.3]{sieve}.

Consider the sets
\[
\mathcal{A}^\ell := \left\{ N\left( \Phi_\ell(\pi) \right) \;\middle|\; N(\pi) = p \leq x,\ 
\pi \equiv 1 \pmod{2(1+i)} \right\},
\]
\[
\widetilde{\mathcal{A}^\ell} := \left\{ \Phi_\ell(\pi) \;\middle|\; N(\pi) = p \leq x,\ 
\pi \equiv 1 \pmod{2(1+i)} \right\}.
\]
We first bound the elements of \( \mathcal{A}^\ell \). Observe that
\begin{equation}
	N\left( \Phi_\ell(\pi) \right) 
= \left| \Phi_\ell(\pi) \right|^2 
= \left| \prod_{k=1}^{\ell - 1} \left( \pi - e^{2\pi i k / \ell} \right) \right|^2  
\leq \left( \sqrt{p} + 1 \right)^{2(\ell - 1)} 
\leq c_\ell x^{\ell - 1}, \label{bound_elements}
\end{equation}
for some constant \( c_\ell > 0 \).

As in the proof of Theorem~\ref{almost_triple}, the key step is to estimate \( \left| \widetilde{\mathcal{A}^\ell}_{\mathfrak{a}} \right| \), where
\[
\widetilde{\mathcal{A}^\ell}_{\mathfrak{a}} := \left\{ \alpha \in \widetilde{\mathcal{A}^\ell} \;:\; \mathfrak{a} \mid \alpha \right\},
\]
and \( \mathfrak{a} \subset \mathbb{Z}[i] \) is a square-free ideal such that \( d(\mathfrak{a}) \mid P(x) \). Here, \( d(\mathfrak{a}) \) denotes the smallest positive integer divisible by \( \mathfrak{a} \), as defined on p.~\pageref{def_d_a}.

For a square-free ideal \( \mathfrak{a} \subset \mathbb{Z}[i] \), the following conditions are equivalent by the Chinese Remainder Theorem and the properties of cyclotomic polynomials:
\begin{enumerate}
	\item[(1)] \( \mathfrak{a} \mid \Phi_\ell(\pi) \) in \( \mathbb{Z}[i] \).
	\smallskip
	\item[(2)] For each prime ideal \( \mathfrak{p}_i \mid \mathfrak{a} \), let \( k_i = \mathbb{Z}[i]/\mathfrak{p}_i \) be the corresponding residue field:
	\begin{itemize}
		\item[\textup{(a)}] If \( \operatorname{char}(k_i) \neq \ell \), then \( \Phi_\ell(X) \) splits completely over \( k_i \); that is,
		\[
		\Phi_\ell(X) = \prod_{j=1}^{\ell - 1} (X - c_{i,j}),
		\]
		where \( \{c_{i,1}, \dots, c_{i,\ell - 1}\} \subset k_i \) are the primitive \( \ell \)th roots of unity, and
		\[
		\pi \equiv c_{i,j}' \pmod{\mathfrak{p}_i} \quad \text{for some } 1 \leq j \leq \ell - 1,
		\]
		with \( c_{i,j}' \in \mathbb{Z}[i] \) a lift of \( c_{i,j} \).
		
		\item[\textup{(b)}] If \( \operatorname{char}(k_i) = \ell \), then in \( k_i \),
		\[
		\Phi_\ell(X) = (X - 1)^{\ell - 1}, \quad \text{and} \quad \pi \equiv 1 \pmod{\mathfrak{p}_i}.
		\]
	\end{itemize}
\end{enumerate}
In case~\textup{(2)(a)}, we must have \( N(\mathfrak{p}_i) \equiv 1 \pmod{\ell} \), and in case~\textup{(2)(b)}, \( \ell \mid N(\mathfrak{p}_i) \). In particular, \( (1+i) \nmid \mathfrak{a} \), since this would imply \( 2 \equiv 1 \pmod{\ell} \) or \( \ell \mid 2 \), both excluded by the assumption \( \ell \geq 3 \). Therefore, \( \widetilde{\mathcal{A}^\ell}_{\mathfrak{a}} = \emptyset \) if \( (\mathfrak{a}, (1+i)) \neq (1) \).

Moreover, in case~\textup{(2)(a)}, \( \pi \) lies in one of \( \ell - 1 \) residue classes modulo \( \mathfrak{p}_i \), whereas in case~\textup{(2)(b)}, it lies in a single class modulo \( \mathfrak{p}_i \).

We write
\[
|\widetilde{\mathcal{A}^\ell}_{\mathfrak{a}}| = \widetilde{g}(\mathfrak{a})\, \operatorname{li}(x) + \widetilde{r}(\mathfrak{a}),
\]
where \( \widetilde{g}(\mathfrak{a}) \) is a multiplicative function supported on square-free ideals of \( \mathbb{Z}[i] \), with values on prime ideals given as follows:
\begin{align*}
	& \widetilde{g}(1 + i) = 0, \\
	& \widetilde{g}(\pi) = \frac{\ell - 1}{p - 1},    \quad\text{if } N(\pi) = p \text{ splits and } p \equiv 1 \pmod{\ell}, \\
	& \widetilde{g}(q) = \frac{\ell - 1}{q^2 - 1},   \quad\text{if } q \text{ is inert and } q^2 \equiv 1 \pmod{\ell}, \\
	& \widetilde{g}(\pi_0) = \frac{1}{\ell - 1},   \quad\text{if } \ell \equiv 1 \pmod{4} \text{ and } \pi_0 \mid \ell, \\
	& \widetilde{g}(\ell) = \frac{1}{\ell^2 - 1},    \quad\text{if } \ell \equiv 3 \pmod{4}.
\end{align*}
By Proposition~\ref{Bombierri-number field_3}, for any \( A > 0 \), there exists   \( B = B(A) > 0 \) such that  
\begin{equation}\label{error_tilde}
	\sum_{ N(\mathfrak{a}) \leq \frac{x^{1/2}}{(\log x)^B} } \widehat{\mu}^2(\mathfrak{a})|\widetilde{r}(\mathfrak{a})| \ll _\ell \frac{x}{(\log x)^A},
\end{equation}
analogous to~\eqref{error_term_original}.

	We now compute the density function and error term using the same approach as in \eqref{W_interchange} and \eqref{order_summation}. For \( d \mid P(x) \), we have  
	\begin{equation}\label{A^l_d}
		|\mathcal{A}^\ell_d| = \sum_{k \mid d} \mu(k) \sum_{\mathfrak{a} \mid k} \widehat{\mu}(\mathfrak{a}) |\widetilde{\mathcal{A}^\ell}_{\mathfrak{a}}| = \sum_{d(\mathfrak{a}) = d} \widehat{\mu}(\mathfrak{a}) \mu(d)\, |\widetilde{\mathcal{A}^\ell}_{\mathfrak{a}}| = g(d) \operatorname{li}(x) + r(d),
	\end{equation}
	where
	\[
	g(d) = \sum_{d(\mathfrak{a}) = d} \widehat{\mu}(\mathfrak{a}) \mu(d) \widetilde{g}(\mathfrak{a}), \qquad
	r(d) = \sum_{d(\mathfrak{a}) = d} \widehat{\mu}(\mathfrak{a}) \mu(d) \widetilde{r}(\mathfrak{a}).
	\]
The function \( g(d) \) is multiplicative and supported on square-free integers, with values at primes  given by 
	\begin{equation} \label{eq:g_function}
		g(p) =
		\begin{cases}
			\displaystyle \frac{2\varphi(\ell)}{p - 1} - \frac{\varphi(\ell)^2}{(p - 1)^2}, & \text{if } p \equiv 1 \pmod{4} \text{ and } p \equiv 1 \pmod{\ell}, \\[8pt]
			\displaystyle \frac{\varphi(\ell)}{p^2 - 1}, & \text{if } p \equiv 3 \pmod{4} \text{ and } p^2 \equiv 1 \pmod{\ell}, \\[8pt]
			\displaystyle \frac{2}{\ell - 1} - \frac{1}{(\ell - 1)^2}, & \text{if } p = \ell \text{ and } \ell \equiv 1 \pmod{4}, \\[8pt]
			\displaystyle \frac{1}{\ell^2 - 1}, & \text{if } p = \ell \text{ and } \ell \equiv 3 \pmod{4}, \\[8pt]
			0, & \text{if } p = 2 \text{ or otherwise}.
		\end{cases}
	\end{equation}
	
For square-free integers \( d \), the square-free ideals \( \mathfrak{a} \) with \( d(\mathfrak{a}) = d \) are distinct as \( d \) varies, and satisfy
\(d \leq N(\mathfrak{a}) \leq d^2,\)
with both bounds attained; see~\eqref{bound_N(a)}. It then follows from~\eqref{error_tilde} that
\begin{equation}\label{error_d}
	\sum_{\substack{d < \frac{x^{1/4}}{(\log x)^{B/2}}}} \mu^2(d)\, |r(d)| \ll \frac{x}{(\log x)^A}.
\end{equation}

By Mertens' estimate, we have
\[
\sum_{\substack{p \leq x \\ p \equiv 1 \pmod{4} \\ p \equiv 1 \pmod{\ell}}} \frac{1}{p-1}
= \frac{1}{2(\ell - 1)} \log \log x + M(\ell) + O\left( \frac{1}{\log x} \right),
\]
for some constant \( M(\ell) \). 
Then,
\[	V(x) := \prod_{p < x} \left(1 - g(p)\right)
= \left(1 - g(\ell)\right)
\prod_{\substack{p < x \\ p \equiv 1 \pmod{4} \\ p \equiv 1 \pmod{\ell}}} \left(1 - \frac{\varphi(\ell)}{p - 1} \right)^2
\prod_{\substack{q < x \\ q \equiv 3 \pmod{4} \\ q^2 \equiv 1 \pmod{\ell}}} \left(1 - \frac{\varphi(\ell)}{q^2 - 1} \right)
\sim \frac{C_\ell}{\log x},\]
where
\[
C_\ell = (1 - g(\ell)) \cdot e^{-2(\ell-1)M(\ell)}
\prod_{\substack{q \equiv 3 \pmod{4} \\ q^2 \equiv 1 \pmod{\ell}}} \left(1 - \frac{\varphi(\ell)}{q^2 - 1} \right),
\]
and
\[
1 - g(\ell) = \frac{1 + \chi_4(\ell)}{2} \left(1 - \frac{1}{\ell - 1} \right)^2
+ \frac{1 - \chi_4(\ell)}{2} \left(1 - \frac{1}{\ell^2 - 1} \right),
\]
with \( \chi_4 \) the nontrivial Dirichlet character modulo 4.
		
Thus, the density function \( g \) satisfies the assumptions of the linear sieve.  

Let \( 0 < \theta < \delta < \tfrac{1}{2} \), and let \( \lambda > 0 \) be a parameter to be determined later. Define \( z = x^\theta \) and \( y = x^\delta \). We consider the following weighted sieve \label{similar_S(x)} expression
\[
S(x) := \left| S(\mathcal{A}^\ell, z) \right| 
- \lambda \sum_{z \leq p < y} w_p \left| S\left((\mathcal{A}^\ell)_p, z\right) \right|,
\qquad w_p := 1 - \frac{\log p}{\log y}.
\]
For \( p \geq y \), we have \( w_p \leq 0 \), so
\[
\sum_{z \leq p < y} w_p \left| S\left((\mathcal{A}^\ell)_p, z\right) \right| 
\geq \sum_{p \geq z} w_p \left| S\left((\mathcal{A}^\ell)_p, z\right) \right|,
\]
and the tail sum \( \sum_{p \geq y} w_p \left| S\left((\mathcal{A}^\ell)_p, z\right) \right| \) is always negative, regardless of whether multiplicities are considered.
Therefore,
\begin{align*}
	S(x) 
	&\leq \left| S(\mathcal{A}^\ell, z) \right| 
	- \lambda \sum_{p \geq z} w_p \left| S\left((\mathcal{A}^\ell)_p, z\right) \right| \\
	&\leq \sum_{a \in S(\mathcal{A}^\ell, z)} 
	\left( 1 - \lambda \left( \omega(a) - \frac{\sum_{p \mid a} \log p}{\log y} \right) \right) \\
	&\leq \sum_{a \in S(\mathcal{A}^\ell, z)} 
	\left( 1 - \lambda \left( \omega(a) - \frac{\log a}{\log y} \right) \right),
\end{align*}
and to replace \( \omega(a) \) with \( \Omega(a) \), we must account for the contribution from primes \( z \leq p < y \) with \( p^2 \mid a \).
Now, using the bound on elements in \( \mathcal{A}^\ell \) from~\eqref{bound_elements}, we obtain
\begin{align}
	S(x) 
	&\leq \sum_{a \in S(\mathcal{A}^\ell, z)} 
	\left( 1 - \lambda \left( \omega(a) - \frac{\log(c_\ell x^{\ell - 1})}{\log y} \right) \right) \nonumber \\
	&\leq \left( 1 + \frac{\log c_\ell}{\delta \log x} + \frac{\ell - 1}{\delta} \right) 
	\#\left\{ a \in S(\mathcal{A}^\ell, z) : \omega(a) < \frac{1}{\lambda} + \frac{\ell - 1}{\delta} \right\}. \label{omega}
\end{align}

We now show that the contribution from elements \( a \in S(\mathcal{A}^\ell, z) \) divisible by \( p^2 \) for some \( z \leq p \leq x^{1/2} \) is negligible. Moreover, terms with \( p \geq y \) contribute negatively due to the weight \( w_p \leq 0 \), regardless of whether multiplicities are counted. Hence, we may   replace \( \omega(a) \) with \( \Omega(a) \) in~\eqref{omega}.

Suppose \( z < q \leq   x^{1/2}  \) is a prime such that \( q^2 \mid N(\Phi_\ell(\alpha)) = \Phi_\ell(\alpha)\Phi_\ell(\overline{\alpha}) \) for some \( \alpha \in \mathbb{Z}[i] \) with \( N(\alpha) = p \) a splitting prime. We may assume \( q \neq \ell \).

If \( q \) splits, then \( \pi_1 \pi_2 \mid \Phi_\ell(\alpha) \) with prime ideals \( \pi_1, \pi_2 \in \mathbb{Z}[i] \) satisfying \( N(\pi_1) = N(\pi_2) = q \). So there exists \( \beta \in \mathbb{Z}[i] \) with \( N(\beta) < N(\pi_1\pi_2) = q^2 \) such that
\[
\Phi_\ell(\beta) \equiv 0 \pmod{\pi_1 \pi_2}, \quad \text{and} \quad \alpha \equiv \beta \pmod{\pi_1 \pi_2}.
\]
Hence,
\[
q^2 = N(\pi_1 \pi_2) \mid N(\alpha - \beta) \leq (|\alpha| + |\beta|)^2 \leq (\sqrt{x} + q)^2 \leq 4x.
\]

If \( q \) is inert, then \( q \mid \Phi_\ell(\alpha) \), so there exists \( \gamma \in \mathbb{Z}[i] \) with \( N(\gamma) < q^2 \) such that
\[
\Phi_\ell(\gamma) \equiv 0 \pmod{q}, \quad \text{and} \quad \alpha \equiv \gamma \pmod{q},
\]
and again,
\[
q^2 = N(q) \mid N(\alpha - \gamma) \leq (|\alpha| + |\gamma|)^2 \leq 4x.
\]

In both cases, for fixed \( q \), the number of such \( \alpha \in \mathbb{Z}[i] \) is bounded by the number of \( \theta \in \mathbb{Z}[i] \) with \( N(\theta) \leq 4x \) and \( q^2 \mid N(\theta) \), which satisfies
\[
\sum_{\substack{n \leq 4x \\ q^2 \mid n}} \tau(n) \ll \frac{x}{q^2} \log x.
\]
Summing over \( z < q \leq   x^{1/2} \), the total number of elements in \( \mathcal{A}^\ell \) divisible by \( q^2 \) is bounded by
\[
\sum_{ z < q \leq   x^{1/2} } \sum_{\substack{n \leq 4x \\ q^2 \mid n}} \tau(n) 
\ll x \log x \sum_{q > z} \frac{1}{q^2} 
\ll \frac{x \log x}{z}.
\]		
Thus, we may replace \( \omega(a) \) with \( \Omega(a) \) in~\eqref{omega}, and interpret \( S(x) \) as
\begin{equation}\label{Omega}
	S(x) \leq \left( 1 + \frac{\log c_\ell}{\delta \log x} + \frac{\ell - 1}{\delta} \right) 
	\#\left\{ a \in S(\mathcal{A}^\ell, z) \;\middle|\; \Omega(a) < \frac{1}{\lambda} + \frac{\ell - 1}{\delta} \right\}.
\end{equation}

Taking the level of distribution \( D = x^{1/4 - \epsilon} \) for small \( \epsilon > 0 \), as justified by~\eqref{error_d}, and applying the linear sieve~\cite[Theorem~11.12 with \( \kappa = 1 \)]{FI2010}, we obtain for sufficiently large \( x \),
\[
S(x) \geq V(z)\, \operatorname{li}(x)  \, G(\theta, \delta, \lambda),
\]
where
\begin{equation}\label{G}
	G(\theta, \delta, \lambda) = f\left( \frac{\alpha}{\theta} \right) 
	- \lambda \int_{\theta}^{\delta} F\left( \frac{\alpha - t}{\theta} \right) 
	\left( 1 - \frac{t}{\delta} \right) \frac{dt}{t}, \qquad \alpha = \tfrac{1}{4} - \epsilon,
\end{equation}
and \( f \), \( F \) are the standard lower and upper bound functions of the linear sieve, given by
\begin{equation}\label{f}
	f(s) = \frac{2 e^\gamma \log(s - 1)}{s}, \qquad \text{for } 2 \leq s \leq 4,
\end{equation}
\begin{equation}\label{F}
	F(s) = 
	\begin{cases}
		\displaystyle \frac{2 e^\gamma}{s}, & 1 \leq s \leq 3, \\[8pt]
		2 e^\gamma, & 0 < s < 1.
	\end{cases}
\end{equation}

Note that for \( F(s) \) with \( 0 < s < 1 \), we use the trivial upper bound \( F(1) \). For larger values of \( s \), the functions are extended using the differential-difference equations from~\cite[p.~235]{FI2010}. In our case, it suffices to use \( F \) and \( f \) in the small ranges specified in \eqref{f} and \eqref{F}, provided the parameters are chosen appropriately.

Taking \( \alpha = \frac{1}{4.01} \), we find that  
\[
G\left( \frac{1}{15}, \frac{1}{4.1}, \frac{1}{1.3} \right) = 0.1341 > 0,
\]
which corresponds to the numerical bound in ~\eqref{Omega}  
\[
\Omega(a) \leq  \left\lfloor \frac{1}{\lambda} +\frac{\ell-1}{\delta}\right\rfloor
 = \left\lfloor 4.1(\ell - 1) + 1.2\right\rfloor.
\]
Note that the choice of \( \lambda \), \( \theta \), and \( \delta \) is not unique. However, it is important to ensure \( \delta < \tfrac{1}{4} \), so that the argument \( (\alpha - t)/\theta \) in the definition of \( G \) remains positive throughout the integral. Moreover, the functions \( f \) and \( F \) should be evaluated within the valid ranges specified in~\eqref{f} and~\eqref{F}.

The case \( \ell = 2 \) is similar, but slightly simpler. We define the sifting sets
\[
\mathcal{A}^2 := \left\{ \frac{N(\pi + 1)}{4} \;\middle|\; N(\pi) = p \leq x,\ \pi \equiv 1 \pmod{2(1 + i)} \right\},\]
\[
\widetilde{\mathcal{A}^2} := \left\{ \frac{\pi + 1}{2} \;\middle|\; N(\pi) = p \leq x,\ \pi \equiv 1 \pmod{2(1 + i)} \right\}.
\]
The expression in~\eqref{A^l_d} remains valid for \( |\mathcal{A}^2_d| \), and the density function \( g \) defined in~\eqref{eq:g_function} also applies with \( \ell = 2 \).

As before, the number of elements \( a \in S(\mathcal{A}^2, z) \) divisible by \( p^2 \) for some \(  z < q \leq   x^{1/2}  \) is bounded by
\[
\ll \sum_{ z < q \leq   x^{1/2} } \sum_{\substack{n \leq x \\ p^2 \mid n}} \tau(n) \ll \frac{x \log x}{z},
\]
which is negligible. Moreover, since \( N(\alpha + 1)/4 < x \) for \( N(\alpha) \leq x \), no \( a \in S(\mathcal{A}^2, z) \) can be divisible by \( p^2 \) with \( p \geq x^{1/2} \). Thus, the contribution from non-square-free elements is negligible.
We may therefore interpret \( S(x) \) as
\begin{equation}\label{omega_2}
	S(x) \leq \left( 1 + \frac{1}{\delta} \right) 
	\#\left\{ a \in S(\mathcal{A}^2, z) \;\middle|\; \omega(a) < \frac{1}{\lambda} + \frac{1}{\delta},\; a \text{ square-free} \right\}.
\end{equation}

	Similarly, applying the linear sieve with level of distribution \( D = x^{1/4 - \epsilon} \) for some small \( \epsilon > 0 \), we obtain for sufficiently large \( x \),
	\[
	S(x) \geq V(z)\, \operatorname{li}(x)\, G(\theta, \delta, \lambda),
	\]
	where the function \( G \) is as defined in~\eqref{G}. Taking the same parameters,
	\[
	(\alpha, \theta, \delta, \lambda) = \left( \frac{1}{4.01}, \frac{1}{15}, \frac{1}{4.1}, \frac{1}{1.3} \right),
	\]
the inequality in~\eqref{omega_2} becomes
	\[
	\omega(a) \leq \left\lfloor \frac{1}{\lambda} + \frac{1}{\delta} \right\rfloor = 5.
	\]
\end{proof}
	
\begin{cor}\label{correct}
	Let \( E/\mathbb{Q} \) be the elliptic curve defined by  
\(	y^2 = x^3 - x.\)
	Then for sufficiently large \( x \), we have
	\[
	\# \left\{ p \leq x \;\middle|\; p \equiv 1 \pmod{4},\ \frac{|E(\mathbb{F}_p)|}{8}= P_5 \right\} \gg \frac{x}{(\log x)^2},
	\]
	where \( P_5 \) denotes a product of at most \( 5 \) distinct prime factors.
\end{cor}

	\begin{proof}
	The proof proceeds similarly to the case \( \ell = 2 \) in Theorem \ref{higher_order}. Define
	\[
	\mathcal{A}^1 := \left\{ N\left(\frac{\pi - 1}{2(1 + i)}\right) \;\middle|\; N(\pi) = p \leq x,\; \pi \equiv 1 \pmod{2(1 + i)} \right\}
	\]
	\[
	\widetilde{\mathcal{A}^1} := \left\{ \frac{\pi - 1}{2(1 + i)} \;\middle|\; N(\pi) = p \leq x,\; \pi \equiv 1 \pmod{2(1 + i)} \right\}.
	\]
	Then,
	\[
	|\widetilde{\mathcal{A}}^1_\mathfrak{a}| = \widetilde{g}(\mathfrak{a}) \operatorname{li}(x) + \widetilde{r}(\mathfrak{a}),
	\]
where \( \widetilde{g} \) is a multiplicative function supported on square-free ideals, with values at prime ideals given by
\begin{align*}
	\widetilde{g}(1 + i) &= \frac{1}{2}, \\
	\widetilde{g}(\pi) &= \frac{1}{p - 1}, \quad \text{if } N(\pi) = p \text{ and } p \text{ splits}, \\
	\widetilde{g}(q) &= \frac{1}{q^2 - 1}, \quad \text{if } q \text{ is inert}.
\end{align*}
Moreover, for any \( A > 0 \), there exists \( B = B(A) > 0 \) such that the error term satisfies
\begin{equation*} 
	\sum_{ N(\mathfrak{a}) \leq \frac{x^{1/2}}{(\log x)^B} } \widehat{\mu}^2(\mathfrak{a})|\widetilde{r}(\mathfrak{a})| \ll _\ell \frac{x}{(\log x)^A}.
\end{equation*}
As in \eqref{A^l_d}, for \( d \mid P(z) \), we have  
\[
|\mathcal{A}^1_d| = \sum_{k \mid d} \mu(k) \sum_{\mathfrak{a} \mid k} \widehat{\mu}(\mathfrak{a}) |\widetilde{\mathcal{A}}^1_{\mathfrak{a}}| = \sum_{d(\mathfrak{a}) = d} \widehat{\mu}(\mathfrak{a}) \mu(d)\, |\widetilde{\mathcal{A}}^1_{\mathfrak{a}}| = g(d) \operatorname{li}(x) + r(d),
\]
where
\[
g(d) = \sum_{d(\mathfrak{a}) = d} \widehat{\mu}(\mathfrak{a}) \mu(d) \widetilde{g}(\mathfrak{a}),  
\] 
\begin{equation}\label{r(d)_error}
	r(d) = \sum_{d(\mathfrak{a}) = d} \widehat{\mu}(\mathfrak{a}) \mu(d) \widetilde{r}(\mathfrak{a}).
\end{equation}
	The multiplicative function \( g \) is the same as in the case \( \ell = 2 \) from \eqref{eq:g_function}, except at \( p = 2 \), here we have 
	\[
	g(p) =
	\begin{cases}
		\frac{1}{2}, & p = 2, \\
		\frac{2}{p - 1} - \frac{1}{(p - 1)^2}, & p \equiv 1 \pmod{4}, \\
		\frac{1}{p^2 - 1}, & p \equiv 3 \pmod{4}.
	\end{cases}
	\]
As before, the error term \( r(d) \) satisfies the bound in~\eqref{error_d}: for any \( A > 0 \), there exists \( B = B(A) \) such that
\[
\sum_{\substack{d < x^{1/4} / (\log x)^B}} \mu^2(d)\, |r(d)| \ll \frac{x}{(\log x)^A}.
\]

The remainder of the proof then follows exactly as in Proposition~\ref{higher_order}, using the same level of distribution \( D = x^{1/4 - \epsilon} \) and the same parameters for \( G \) in~\eqref{G}.
	\end{proof}

\begin{rem}\label{mistake}
The setup in Corollary~\ref{correct} is identical to that in~\cite{IU_elliptic}. However, the first equality on~\cite[p.~823]{IU_elliptic} is incorrect, leading to an overestimate of the error term in~\cite[(3.8)]{IU_elliptic} and subsequently in~\cite[(3.9)]{IU_elliptic}.

Let \( d \mid \prod_{p < z} p \), and write \( d = 2^e d_1 d_2 \), where all prime divisors of \( d_1 \) are congruent to \( 1 \pmod{4} \), and all prime divisors of \( d_2 \) are congruent to \( 3 \pmod{4} \).

For an ideal in \( \mathbb{Z}[i] \) of the form \( \mathfrak{d} = (1 + i)^e \alpha_1 b_2 \), with \( N(\alpha_1)   = a_1 \mid d_1 \) and \(b_2\mid d_2\),~\cite{IU_elliptic} defines \( a_{\mathfrak{d}} := 2^e a_1 b_2 \).  
For an integer \( k = 2^e k' \) with \( 2 \nmid k' \),~\cite{IU_elliptic} defines \( \kappa := (1 + i)^e k' \).
Then~\cite{IU_elliptic} claims:
\begin{equation}\label{wrong}
	\sum_{k \mid d} \mu(k) \sum_{\mathfrak{d} \mid \kappa} \widehat{\mu}(\mathfrak{d}) \widehat{r}_{\mathfrak{d}}(x) 
	= \sum_{a_{\mathfrak{d}} = d} \widehat{\mu}(\mathfrak{d}) \widehat{r}_{\mathfrak{d}}(x) \mu(a_{\mathfrak{d}}).
\end{equation}
(Note: \( \kappa \) may be replaced by \( k \), since \( \mathfrak{d} \mid \kappa \iff \mathfrak{d} \mid k \), and likewise \( (\alpha, \kappa) = (1) \iff (\alpha, k) = (1) \); see~\eqref{order_summation}.)

	The issue lies in using \( a_{\mathfrak{d}} \) instead of the proper quantity \( d(\mathfrak{d}) \), the smallest integer divisible by \( \mathfrak{d} \). For example, if \( d = p = \pi \overline{\pi} \), the left-hand side of \eqref{wrong} equals
	\[
	\widehat{r}_{\pi}(x) + \widehat{r}_{\overline{\pi}}(x) - \widehat{r}_{\pi\overline{\pi}}(x),
	\]
	whereas the right-hand side misses the last term:
	\[
	\widehat{r}_{\pi}(x) + \widehat{r}_{\overline{\pi}}(x).
	\]
	
	Correcting this requires replacing \( a_{\mathfrak{d}} \) with \( d(\mathfrak{d}) \), the smallest integer divisible by \( \mathfrak{d} \), as defined on p.~\pageref{def_d_a}. With this fix, the right-hand side of \eqref{wrong} should be replaced by \eqref{r(d)_error},  and Corollary~\ref{correct} matches the result obtained using this framework together with a weighted sieve.
	
	This mistake is critical: it incorrectly suggests a level of distribution \( x^{1/2-\epsilon} \), while the actual level is \( x^{1/4-\epsilon} \). Since Chen’s method depends on this, the argument in \cite{IU_elliptic} does not establish a \( P_2 \) result.
	
\begin{comment}
		Indeed, for \( d(\mathfrak{d}) = d \), we have
	\[
	d \leq N(\mathfrak{d}) \leq d^2,
	\]
	so the condition \( N(\mathfrak{d}) < x^{1/2-\epsilon} \) requires
	\[
	d < x^{1/4 - \epsilon/2}.
	\]
	In contrast, if \( a_{\mathfrak{d}} = d = 2^e d_1 d_2 \), then
	\[
	N(\mathfrak{d}) = 2^e d_1 d_2^2,
	\]
	which leads to the incorrect bound
	\[
	2^e d_1 d_2^2 < x^{1/2 - \epsilon}.
	\]
	Neglecting the small contribution from large \( d_2 \), this yields an erroneous level of distribution \( d < x^{1/2 - \epsilon} \) as in \cite{IU_elliptic}.
\end{comment}
\end{rem}

	\section{Almost prime order of $\frac{|E(\mathbb{F}_{p^\ell})|}{ |E(\mathbb{F}_p)|}$ for non-CM elliptic curves}\label{section_non-CM}
In this section, we study the almost prime property of
\(\frac{|E(\mathbb{F}_{p^\ell})|}{|E(\mathbb{F}_p)|}\)
for a non-CM elliptic curve \( E/\mathbb{Q} \) and prime \( \ell \geq 3 \), paralleling the results in Section~\ref{discussion} for CM curves.  
The case \( \ell = 2 \) is simpler and resembles the setting of \( |E(\mathbb{F}_p)| \) in~\cite{steuding}.  
Our approach is based on the Chebotarev density theorem, as outlined in Section~\ref{chebotarev}.

Let \( E  \) be an elliptic curve defined over \(\mathbb{Q}\) with conductor \( N_E \), and let \( \mathbb{Q}(E[n]) \) denote the \( n \)-division field, obtained by adjoining the coordinates of all \( n \)-torsion points of \( E \). Define \( G_n := \operatorname{Gal}(\mathbb{Q}(E[n])/\mathbb{Q}) \) as its Galois group.
Let \( \rho_{E, n} : G_n \hookrightarrow \operatorname{Aut}(E[n]) \) be the mod-\( n \) Galois representation.  
By choosing a basis for the \( n \)-torsion, this induces a linear representation
\[
\rho_{E, n} : G_n \hookrightarrow \operatorname{GL}_2(\mathbb{Z}/n\mathbb{Z}).
\]

	By Serre \cite{galois_serre}, if \( E \) does not have CM, then there exists a positive integer \( M_E \) such that
	\begin{align}\label{group_G_n}
		&G_n \cong \operatorname{GL}_2(\mathbb{Z}/n\mathbb{Z}), \quad \text{for } (n, M_E) = 1,\nonumber \\ 
		&G_{mn} \cong G_m \times \operatorname{GL}_2(\mathbb{Z}/n\mathbb{Z}), \quad \text{for } (n, M_E) = (m, n) = 1.
	\end{align}

	Let \( p, q \) be primes such that \( p \nmid q N_E \). Denote by \( \sigma_p \) the Frobenius at \( p \) in \( G_{q^n} \), which satisfies  
	\[
	\operatorname{tr} \rho_{E,q^n} (\sigma_p) \equiv a_p \pmod{q^n}, \quad \text{and} \quad \det \rho_{E,q^n} (\sigma_p) \equiv p \pmod{q^n}.
	\]
Hence, the characteristic polynomial of \( \rho_{E,q^n}(\sigma_p) \) is  
\[
T^2 - a_p T + p \pmod{q^n}.
\]
When the context is clear that we are working modulo \( q^n \), we write \( \sigma_p \) for \( \rho_{E,q^n}(\sigma_p) \) to simplify notation.

	Let \( \pi_p \) and \( \overline{\pi}_p \) be the complex conjugate roots of the equation \[
	T^2 - a_p T + p  =0 .
	\]
	Consider the imaginary quadratic field \( \mathbb{Q}(\sqrt{\Delta_p}) \), where \( \Delta_p = a_p^2 - 4p\neq 0 \), then $\pi_p,\overline{\pi}_p$ lie in the ring of integers $\mathcal{O}_{\mathbb{Q}(\sqrt{\Delta_p})}$ of $\mathbb{Q}(\sqrt{\Delta_p})$.   Hence, a prime \( q \in \mathbb{Z} \) has one of the following factorizations:  
	\[
	q\mathcal{O}_{\mathbb{Q}(\sqrt{\Delta_p})}= \mathfrak{q} \quad \text{or} \quad q\mathcal{O}_{\mathbb{Q}(\sqrt{\Delta_p})}= \mathfrak{q}_1\mathfrak{q}_2.
	\]

Suppose \( \ell\mid q^2 - 1 \) or \( \ell= q \), so that the \( \ell\)-th cyclotomic polynomial \( \Phi_\ell(x) \) has roots in \( \mathbb{F}_{q^2} \). Let \( e_{q,1}, e_{q,2}, \dots, e_{q,\ell-1} \) denote the roots of \( \Phi_\ell \) in \( \mathbb{F}_{q^2} \). Then:
\begin{itemize}
	\item If \(\ell\mid q - 1 \), the \( e_{q,i} \) lie in \( \mathbb{F}_q \) and are the primitive \(\ell\)-th roots of unity in \( \mathbb{F}_q \).
	\item If \(\ell\mid q + 1 \), the \( e_{q,i} \) are the primitive \(\ell\)-th roots of unity in \( \mathbb{F}_{q^2} \setminus \mathbb{F}_q \).
	\item If \( q =\ell\), then \( e_{q,i} = 1 \).
\end{itemize}

	Define \( C_q \) to be the union of conjugacy classes in \( G_q \) as follows:  
	\begin{equation}\label{def_conjugacy}
		C_q := \{ \sigma \in G_q \mid e_{q,j}^2 - e_{q,j} \operatorname{tr} \sigma + \det \sigma = 0 \text{ in } \mathbb{F}_{q^2}, \quad \text{for some } 1 \leq j \leq \ell-1 \}.
	\end{equation}
In other words, \( C_q \) consists of those matrices \( \sigma \in G_q \) for which \( e_{q,i} \in \mathbb{F}_{q^2}\) is an eigenvalue of \( \sigma \) for some \( 1 \leq i \leq\ell- 1 \).

	In particular, if \( q =\ell\), then  
	\begin{equation} \label{def_conjugacy_l}
		C_\ell := \{ \sigma \in G_\ell \mid 1 - \operatorname{tr} \sigma + \det \sigma = 0 \text{ in } \mathbb{F}_{\ell } \}.
	\end{equation}

	We now characterize the divisibility condition as follows. 
	\begin{lemma}\label{lemma_condition_1}
		Let \( p, q \) be primes such that \( p \nmid qN_E \), and let \(\ell\geq 3 \) be an odd prime. If \( q \mid \frac{|E(\mathbb{F}_{p^\ell})|}{|E(\mathbb{F}_{p})|} \), then we must have \(\ell\mid q^2 - 1 \) or \(\ell= q \).  
		\begin{enumerate}
			\item Suppose \(\ell\mid q - 1 \) or \(\ell= q \). Then \( q \mid \frac{|E(\mathbb{F}_{p^\ell})|}{|E(\mathbb{F}_{p})|} \) if and only if \( \sigma_p \in C_q \).   
			\item Suppose \(\ell\mid q + 1 \). Then \( q \mid \frac{|E(\mathbb{F}_{p^\ell})|}{|E(\mathbb{F}_{p})|} \) if and only
			\[
			\sigma_p \in C_q  \quad \text{and} \quad q \text{ is inert in } \mathcal{O}_{\mathbb{Q}(\sqrt{\Delta_p})}.
			\]  
		\end{enumerate}
	\end{lemma}
	
	\begin{proof}
	Suppose
	\[
	q \Big| \frac{|E(\mathbb{F}_{p^\ell})|}{|E(\mathbb{F}_p)|} = \Phi_\ell(\pi_p) \Phi_\ell(\overline{\pi_p}).
	\]
	We will show during the proof   that this implies either \(\ell\mid q^2 - 1 \) or \(\ell= q \), based on the splitting of \( \Phi_\ell \) in the corresponding residue fields. In either case, let \( e_{q,1}, \dots, e_{q,\ell-1} \) denote the roots of \( \Phi_\ell \) in \( \mathbb{F}_{q^2} \).
	
Suppose \( q \) either splits or ramifies in \( \mathcal{O}_{\mathbb{Q}(\sqrt{\Delta_p})} \). Then \( \mathfrak{q}_1 \mid \Phi_\ell(\pi_p) \) for some prime ideal \( \mathfrak{q}_1 \) with  \( N(\mathfrak{q}_1) = q \). In this case, there exists \( \alpha \in \mathcal{O}_{\mathbb{Q}(\sqrt{\Delta_p})} \) such that
\[
\Phi_\ell(\alpha) \equiv 0 \pmod{\mathfrak{q}_1} \quad \text{and} \quad \pi_p \equiv \alpha \pmod{\mathfrak{q}_1}.
\]
It follows that \( \Phi_\ell(X) \) splits in \( \mathcal{O}_{\mathbb{Q}(\sqrt{\Delta_p})}/\mathfrak{q}_1 \cong \mathbb{F}_q \), so we must have either \(\ell\mid q - 1 \) or \( q =\ell\).
Therefore,
\[
(\pi_p - \alpha)(\overline{\pi_p} - \alpha) = \alpha^2 - a_p \alpha + p \equiv 0 \pmod{\mathfrak{q}_1},
\]
whose image in \( \mathcal{O}_{\mathbb{Q}(\sqrt{\Delta_p})}/\mathfrak{q}_1 \cong \mathbb{F}_q \) becomes
\[
e_{q,i}^2 - a_p e_{q,i} + p \equiv  e_{q,i}^2 - \operatorname{tr}(\sigma_p)\, e_{q,i} + \det(\sigma_p) \equiv 0\pmod q \quad \text{for some } 1 \leq i \leq\ell- 1,
\]
and hence \( \sigma_p \in C_q \).

	Suppose \( q \) is inert in \( \mathcal{O}_{\mathbb{Q}(\sqrt{\Delta_p})} \), and let \( \mathfrak{q} = q\mathcal{O}_{\mathbb{Q}(\sqrt{\Delta_p})} \). Then \( \mathfrak{q} \mid \Phi_\ell(\pi_p) \), so there exists \( \beta \in \mathcal{O}_{\mathbb{Q}(\sqrt{\Delta_p})} \) such that
	\[
	\Phi_\ell(\beta) \equiv 0 \pmod{\mathfrak{q}} \quad \text{and} \quad \pi_p \equiv \beta \pmod{\mathfrak{q}}.
	\]
	Thus \( \Phi_\ell(X) \) splits in \( \mathbb{F}_{q^2} \), so \(\ell\mid q^2 - 1 \) or \( q =\ell\).
Therefore,
\[
(\pi_p - \beta)(\overline{\pi_p} - \beta) = \beta^2 - a_p \beta + p \equiv 0 \pmod{\mathfrak{q}},
\]
whose image in \( \mathcal{O}_{\mathbb{Q}(\sqrt{\Delta_p})}/\mathfrak{q} \cong \mathbb{F}_{q^2} \) becomes
\[
e_{q,i}^2 - a_p e_{q,i} + p \equiv  e_{q,i}^2 - \operatorname{tr}(\sigma_p)\, e_{q,i} + \det(\sigma_p)\equiv  0 \pmod q \quad \text{for some } 1 \leq i \leq\ell- 1,
\]
and hence \( \sigma_p \in C_q \).

	Note that the case \(\ell\mid p + 1 \) and \( q \mid \frac{|E(\mathbb{F}_{p^\ell})|}{|E(\mathbb{F}_p)|} \) arises only when \( q \) is inert, as shown above. This completes the forward direction of both (1) and (2).
	We now prove the reverse direction.
	
Suppose \(\ell\mid q - 1 \) or \(\ell= q \), and assume
\[
e_{q,i}^2 - e_{q,i} \operatorname{tr} \sigma_p + \det \sigma_p = 0 \quad \text{in } \mathbb{F}_{q^2}, \quad \text{for some } 1 \leq i \leq\ell- 1.
\]
Since \(\ell\mid q - 1 \) or \(\ell= q \), the roots \( e_{q,i} \) lie in \( \mathbb{F}_q \).
Let \( \mathfrak{q} \) be a prime ideal in \( \mathcal{O}_{\mathbb{Q}(\sqrt{\Delta_p})} \) lying above \( q \). By the natural injection
\[
\mathbb{Z}/q\mathbb{Z} \hookrightarrow \mathcal{O}_{\mathbb{Q}(\sqrt{\Delta_p})}/\mathfrak{q},
\]
we may view \( e_{q,i} \) as an element of \( \mathcal{O}_{\mathbb{Q}(\sqrt{\Delta_p})}/\mathfrak{q} \), and lift it to some \( e_{q,i}' \in \mathcal{O}_{\mathbb{Q}(\sqrt{\Delta_p})} \). Then
\[
\Phi_\ell(e_{q,i}') \equiv 0 \pmod{\mathfrak{q}}, \quad \text{and} \quad  e_{q,i}'^2 - e_{q,i}' a_p + p=(\pi_p - e_{q,i}')(\overline{\pi_p} - e_{q,i}') \equiv 0 \pmod{\mathfrak{q}}.
\]
Hence, either \( \mathfrak{q} \mid \Phi_\ell(\pi_p) \) or \( \mathfrak{q} \mid \Phi_\ell(\overline{\pi_p}) \), so \( q \mid N(\Phi_\ell(\pi_p)) \).

Now suppose \(\ell\mid p + 1 \), \( q \) is inert in \( \mathcal{O}_{\mathbb{Q}(\sqrt{\Delta_p})} \), and
\[
e_{q,i}^2 - e_{q,i} \operatorname{tr} \sigma_p + \det \sigma_p = 0 \quad \text{in } \mathbb{F}_{q^2}, \quad \text{for some } 1 \leq i \leq\ell- 1.
\]
Let \( \mathfrak{q} = q\mathcal{O}_{\mathbb{Q}(\sqrt{\Delta_p})} \), and let \( e_{q,i}' \in \mathcal{O}_{\mathbb{Q}(\sqrt{\Delta_p})} \) be a lift of \( e_{q,i} \in \mathbb{F}_{q^2} \cong \mathcal{O}_{\mathbb{Q}(\sqrt{\Delta_p})}/\mathfrak{q} \). Then,
\[
\Phi_\ell(e_{q,i}') \equiv 0 \pmod{\mathfrak{q}} \quad \text{and} \quad e_{q,i}'^2 - e_{q,i}' a_p + p  = (\pi_p - e_{q,i}')(\overline{\pi_p} - e_{q,i}') \equiv 0 \pmod{\mathfrak{q}}.
\]
Hence, either \( \mathfrak{q} \mid \Phi_\ell(\pi_p) \) or \( \mathfrak{q} \mid \Phi_\ell(\overline{\pi_p}) \), so \( q \mid N(\Phi_\ell(\pi_p)) \).

	\end{proof}
	
Lemma~\ref{lemma_condition_1}   determines the density function and thus the order of the main term in the sieve estimate. While condition (2) depends on whether \( q \) is inert in \( \mathcal{O}_{\mathbb{Q}(\sqrt{\Delta_p})} \)—a less accessible property compared to the the CM case—the main term is governed by the density functions  at primes from condition (1). Primes satisfying condition (2) affect only the constant in the lower bound.

We now establish  divisibility conditions by prime squares, which will be used to bound the contribution from elements in the sieving set divisible by a prime square.

Assume \(\ell\mid q^2 - 1 \) or \(\ell= q \). Recall that \( e_{q,1}, \dots, e_{q,\ell-1} \) are the roots of the \(\ell\)-th cyclotomic polynomial \( \Phi_\ell \) in \( \mathbb{F}_{q^2} \), which are primitive when \(\ell\mid q^2 - 1 \).

Let $q\neq \ell$, define \( C_{q^2} \) to be the union of conjugacy classes in \( G_{q^2} \) consisting of elements \( \sigma \in G_{q^2} \) such that either:
\begin{itemize}
	\item the characteristic polynomial \( X^2 - \operatorname{tr}(\sigma) X + \det(\sigma) \pmod q\) splits as 
	\( (X - e_{q,i})(X - e_{q,j}) \) in \( \mathbb{F}_{q^2} \) for some \( 1 \leq i, j \leq\ell- 1 \), or
	\item \(\ell\mid q-1\) and  \( e_{q,k}'^2 - \operatorname{tr}(\sigma)\, e_{q,k}' + \det(\sigma) = 0 \pmod {q^2}\)  for some \( 1 \leq k \leq\ell- 1 \), where 
	\( e_{q,k}' \in \mathbb{Z}/q^2 \) is the unique lift of some root \( e_{q,k} \in \mathbb{F}_{q } \).
\end{itemize}
That is, \( C_{q^2} \) consists of matrices \( \sigma \in G_{q^2} \) whose characteristic polynomial modulo \( q \) has roots \( e_{q,i}, e_{q,j} \in \mathbb{F}_{q^2} \), or  has a root  \( e_{q,k}' \in \mathbb{Z}/q^2\mathbb{Z} \), for some \( 1 \leq i, j, k \leq\ell- 1 \).

\begin{lemma}\label{lemma_condition_2}
Let \( p, q \) be primes with \( p \nmid qN_E \), and let \(\ell\geq 3 \) be an odd prime. Suppose \( q \neq\ell\). If 
	\(	q^2 \mid \frac{|E(\mathbb{F}_{p^\ell})|}{|E(\mathbb{F}_{p})|},\)
	then \(\ell\mid q^2 - 1 \), and \( \sigma_p \in C_{q^2} \).
\end{lemma}

	\begin{proof}
	  
Suppose \( q^2 \mid \frac{|E(\mathbb{F}_{p^\ell})|}{|E(\mathbb{F}_{p})|} = \Phi_\ell(\pi_p) \Phi_\ell(\overline{\pi_p}) \). By Lemma~\ref{lemma_condition_1}, this implies \(\ell\mid q^2 - 1 \).

Assume \( q \) splits in \( \mathcal{O}_{\mathbb{Q}(\sqrt{\Delta_p})} \), so that
\(q \mathcal{O}_{\mathbb{Q}(\sqrt{\Delta_p})} = \mathfrak{q}_1 \mathfrak{q}_2,\)
with distinct primes \( \mathfrak{q}_1 \) and \( \mathfrak{q}_2 \).   If \( \mathfrak{q}_1 \mathfrak{q}_2 \mid \Phi_\ell(\pi_p) \), then \( q \mid \Phi_\ell(\pi_p) \), which implies \( q \mid \Phi_\ell(\overline{\pi_p}) \) as well, so \( \mathfrak{q}_1 \mathfrak{q}_2 \mid \Phi_\ell(\overline{\pi_p}) \).

Thus, we may assume one of the following two cases:
\begin{enumerate}
	\item \(\mathfrak{q}_1 \mid \Phi_\ell(\pi_p)\) and \(\mathfrak{q}_1 \mid \Phi_\ell(\overline{\pi_p})\),
	\item \(\mathfrak{q}_1^2 \mid \Phi_\ell(\pi_p)\).
\end{enumerate}
In either cases,   \( \Phi_\ell \) splits in \( \mathbb{F}_q \),   so \(\ell\mid q - 1 \).

In the first case, there exist \( \alpha, \beta \in \mathcal{O}_{\mathbb{Q}(\sqrt{\Delta_p})} \) such that
\[
\Phi_\ell(\alpha) \equiv 0 \pmod{\mathfrak{q}_1}, \quad \Phi_\ell(\beta) \equiv 0 \pmod{\mathfrak{q}_1}, \quad 
\pi_p \equiv \alpha \pmod{\mathfrak{q}_1}, \quad \overline{\pi_p} \equiv \beta \pmod{\mathfrak{q}_1}.
\]
Then, for some \( 1 \leq i, j \leq\ell- 1 \), 
\[
\alpha \equiv e_{q,i} \pmod{\mathfrak{q}_1}, \quad \beta \equiv e_{q,j} \pmod{\mathfrak{q}_1}.
\]
Hence, 
\[
X^2 -  a_p  X + p   = (X - \pi_p)(X - \overline{\pi_p}) \equiv (X - e_{q,i})(X - e_{q,j}) \pmod{\mathfrak{q}_1},,
\]
and in \( \mathbb{F}_q \),
\[
X^2 - \operatorname{tr}(\sigma_p) X + \det(\sigma_p) \equiv X^2 - a_p  X + p\equiv  (X - e_{q,i})(X - e_{q,j}) \pmod{q}.
\]

In the second case, there exists \( \gamma \in \mathcal{O}_{\mathbb{Q}(\sqrt{\Delta_p})} \) such that  
\[
\Phi_\ell(\gamma) \equiv 0 \pmod{\mathfrak{q}_1^2} \quad \text{and} \quad \pi_p \equiv \gamma \pmod{\mathfrak{q}_1^2}.
\]
Since \( \Phi_\ell(\gamma) \equiv 0\pmod{\mathfrak{q}_1} \), it follows that  
  \(\gamma \) corresponds to the unique lift \( e_{q,i}'\in \mathbb{Z}/q^2\mathbb{Z}\cong \mathcal{O}_{\mathbb{Q}(\sqrt{\Delta_p})}/\mathfrak{q}_1^2 \) of some root \( e_{q,i}\in \mathbb{F}_q \).  
Therefore,  
\[
\pi_p^2 - a_p \pi_p + p \equiv e_{q,i}'^2 - a_p e_{q,i}' + p \pmod{\mathfrak{q}_1^2},
\]
and hence
\[
e_{q,i}'^2 - \operatorname{tr}(\sigma_p)\, e_{q,i}' + \det(\sigma_p) \equiv e_{q,i}'^2 - a_p e_{q,i}' + p \pmod{q^2}.
\]

Now suppose \( q \) ramifies in \( \mathcal{O}_{\mathbb{Q}(\sqrt{\Delta_p})} \), so that  
\(q \mathcal{O}_{\mathbb{Q}(\sqrt{\Delta_p})} = \mathfrak{q}^2.\)
In this case, \(\ell\mid q - 1 \), and both \( \mathfrak{q} \mid \Phi_\ell(\pi_p) \) and \( \mathfrak{q} \mid \Phi_\ell(\overline{\pi_p}) \). Hence,
\[
\pi_p \equiv e_{q,i} \pmod{\mathfrak{q}}, \quad \overline{\pi_p} \equiv e_{q,j} \pmod{\mathfrak{q}} \quad \text{for some } 1 \leq i, j \leq\ell- 1.
\]
Since \( q \) ramifies, we   have \( \pi_p \equiv \overline{\pi_p} \pmod{\mathfrak{q}} \), it must be that \( i = j \). Therefore, the characteristic polynomial of \( \sigma_p \) satisfies
\[
X^2 - \operatorname{tr}(\sigma_p) X + \det(\sigma_p) \equiv  (X - e_{q,i})^2 \pmod q
\]

Finally, suppose 
$q$ is inert,   so that  
\(q \mathcal{O}_{\mathbb{Q}(\sqrt{\Delta_p})} = \mathfrak{q}.\)
Then \( \mathfrak{q} \mid \Phi_\ell(\pi_p) \) and \( \mathfrak{q} \mid \Phi_\ell(\overline{\pi_p}) \), which implies \(\ell\mid q^2 - 1 \), and  for some \(1 \leq i, j \leq\ell- 1\), 
\[
\pi_p \equiv e_{q,i}' \pmod{\mathfrak{q}}, \quad \overline{\pi_p} \equiv e_{q,j}' \pmod{\mathfrak{q}}     ,
\]
where $e_{q,i}'$ and  $e_{q,j}'$  are lifts of \(e_{q,i}\) and \(e_{q,j}\).
Thus, in \( \mathbb{F}_{q^2} \cong \mathcal{O}_{\mathbb{Q}(\sqrt{\Delta_p})}/\mathfrak{q} \), the characteristic polynomial of \( \sigma_p \) satisfies
\[
X^2 - \operatorname{tr}(\sigma_p) X + \det(\sigma_p) \equiv (X - e_{q,i})(X - e_{q,j}) \pmod{q}.
\]
	\end{proof}

	Recall that \( \pi_C(x, L/K) \) is defined in Section \ref{chebotarev} as  
	\[
	\pi_C(x, L/K) = \# \{ \mathfrak{p} \subset K \mid \mathfrak{p} \text{ is unramified in } L, \, \sigma_{\mathfrak{p}} \in C, \, N_{K/\mathbb{Q}}(\mathfrak{p}) \leq x \}.
	\]
	By the definition of \( \pi_C(x, L/K) \) and the criterion of Néron–Ogg–Shafarevich \cite[Theorem 7.1]{silverman}, we have
	\begin{equation}\label{pi_C_q}
	\pi_{C_q}(x, L_q/\mathbb{Q}) 
= \# \{ p \leq x \mid p \nmid q N_E, \, \sigma_p \in C_q  \} .
	\end{equation}

Let \( d \) be square-free. If every prime factor \( q \mid d \) satisfies \(\ell\mid q^2 - 1 \) or \( q =\ell\), define  
\[
C_d := \left\{ \sigma \in G_d \;\middle|\; \forall\, q \mid d,\ \exists\, 1 \leq i \leq\ell- 1 \text{ such that } e_{q,i}^2 - e_{q,i} \operatorname{tr}(\sigma) + \det(\sigma) \equiv 0 \pmod{q} \right\}.
\]
Otherwise, set \( C_d := \emptyset \).

To account for condition~(2) in Lemma~\ref{lemma_condition_1}, define
\[
\chi_d(p) =
\begin{cases}
	1 & \text{if every prime } q \mid d \text{ with }\ell\mid q + 1 \text{ is inert in } \mathcal{O}_{\mathbb{Q}(\sqrt{\Delta_p})}, \\
	0 & \text{otherwise}.
\end{cases}
\]
Then set
\[
i_d := \frac{
	\#\left\{ p \leq x \;\middle|\; p \nmid dN_E,\ \sigma_p \in C_d,\ \chi_d(p) = 1 \right\}
}{
	\#\left\{ p \leq x \;\middle|\; p \nmid dN_E,\ \sigma_p \in C_d \right\}
}.
\]
In particular, for a prime \( q \),
\[
i_q =
\begin{cases}
	1, & \text{if }\ell\mid q - 1 \text{ or } q = \ell, \\
	\in [0,1], & \text{if }\ell\mid q + 1, \\
	0, & \text{otherwise}.
\end{cases}
\]
Therefore, by Lemma~\ref{lemma_condition_1} and equation~\eqref{pi_C_q}, we have
\[
\#\left\{ p \leq x \;\middle|\; p \nmid qN_E,\ q \mid \frac{|E(\mathbb{F}_{p^\ell})|}{|E(\mathbb{F}_{p})|} \right\}
= i_q \cdot \pi_{C_q}(x, L_q/ K).
\]	
 
If we sieve the set  
\[
\left\{ \frac{|E(\mathbb{F}_{p^\ell})|}{|E(\mathbb{F}_p)|} : p \leq x,\; p \nmid N_E \right\},
\]
Theorem~\ref{function_g} suggests a density function \( h \), defined on square-free integers \( d \) by
\[
h(d) = \frac{i_d\, |C_d|}{|G_d|}.
\]

However, \( h(d) \) may fail to be multiplicative when \( (d, M_E) \neq 1 \), since the fields \( \mathbb{Q}(E[m]) \) and \( \mathbb{Q}(E[n]) \) are not necessarily linearly disjoint for \( (m, M_E) \neq 1 \) and \( (n, M_E) \neq 1 \); see \cite[Example 1.1]{zyw} and \cite[p.~103–105]{wu}.

To address this, we instead sieve the   set
\[
\left\{ \frac{|E(\mathbb{F}_{p^\ell})|}{|E(\mathbb{F}_p)|} : p \leq x,\; p \nmid N_E,\; \left( \frac{|E(\mathbb{F}_{p^\ell})|}{|E(\mathbb{F}_p)|}, M_E \right) = 1 \right\}.
\]
Accordingly, we introduce a correction factor to account for the proportion of primes \( p \leq x \), \( p \nmid N_E \), for which
\(\left( \frac{|E(\mathbb{F}_{p^\ell})|}{|E(\mathbb{F}_p)|}, M_E \right) = 1.\)
We may assume \( M_E \) is square-free, define
\begin{equation}\label{def_C_E}
	c_E := \sum_{d \mid M_E} \mu(d)\, \frac{i_d\, |C_d|}{|G_d|}.
\end{equation}

We now proceed with the proof of Theorem~\ref{intro_non-CM}, which we restate below.

\begin{theorem}[Theorem~\ref{intro_non-CM}]\label{non-CM_theorem}
	Let \( E/\mathbb{Q} \) be a non-CM elliptic curve, and suppose \( c_E \neq 0 \) (see \eqref{def_C_E}). For any odd prime \(\ell\geq 3 \), define
\[
m_\ell = \left\lfloor 5.1 (\ell - 1) + 1.2 \right\rfloor, \quad 
n_\ell = \left\lfloor 8.1 (\ell - 1) + 0.5 \right\rfloor.
\]
Assume GRH for all division fields \( \mathbb{Q}(E[n]) \). Then, for sufficiently large \( x \),
\[
\#\left\{ p \leq x : \omega\left( \frac{|E(\mathbb{F}_{p^\ell})|}{|E(\mathbb{F}_p)|} \right) \leq m_\ell \right\} \gg_E \frac{x}{(\log x)^2},
\]
and
\[
\#\left\{ p \leq x : \Omega\left( \frac{|E(\mathbb{F}_{p^\ell})|}{|E(\mathbb{F}_p)|} \right) \leq n_\ell \right\} \gg_E \frac{x}{(\log x)^2}.
\]
In particular, if \( E \) is a Serre curve, then \( c_E \neq 0 \), so the result holds for almost all non-CM curves.
\end{theorem}

\begin{rem}
	As noted in~\eqref{alpha_theta}, the level of distribution depends on \( \theta \). Our choice \( \theta = \tfrac{1}{2} \) in the \(\theta\)-quasi-GRH corresponds to GRH, but the argument also holds for \( \theta = \tfrac{1}{2} + \epsilon \) with any \( \epsilon > 0 \), given appropriate adjustments to the level of distribution and sieve parameters.
	In particular, if \( \epsilon \) is sufficiently small, the conclusion of Theorem~\ref{non-CM_theorem} remains valid by continuity of the sieve bounds.
\end{rem}

	\begin{proof} 
		We will apply the linear sieve to the set  
		\[
		\mathcal{A} = \left\{ \frac{|E(\mathbb{F}_{p^\ell})|}{|E(\mathbb{F}_p)|} : p \leq x,\; p \nmid N_E, \;   \left(\frac{|E(\mathbb{F}_{p^\ell})|}{|E(\mathbb{F}_p)|}, M_E \right) =1 \right\}.
		\]
		To proceed, we need to determine \( \frac{|C_q|}{|G_q|} \) for \( q \nmid M_E \), with $\ell\mid q^2-1$ or $\ell=q$.  
		
		Suppose \( q \nmid M_E \). Then we have \( G_q \cong \operatorname{GL}_2(\mathbb{F}_q) \), so  
		\[
		|G_q| = (q+1)q(q-1)^2 .
		\]

		Recall that the  conjugacy classes in \( \operatorname{GL}_2(\mathbb{F}_q) \) can be classified into the following categories:
		
		1. Diagonalizable matrices with distinct eigenvalues.
		Each set of distinct  eigenvalues \( \{a, b\} \) corresponds to a unique conjugacy class with a representative  
		\[
		g_{a,b} = \begin{bmatrix} a & 0 \\ 0 & b \end{bmatrix}.
		\]  
		Each conjugacy class has size \( (q+1)q \).
		
		2. Diagonalizable matrices with a single eigenvalue.  
		Each eigenvalue \( a \) corresponds to a representative  
		\[
		g_a = \begin{bmatrix} a & 0 \\ 0 & a \end{bmatrix},
		\]  
		which remains fixed under conjugation. Thus, each conjugacy class has size \( 1 \).
		
		3. Non-diagonalizable matrices with eigenvalues in \( \mathbb{F}_q \).  
		In this case, the eigenvalues must be equal, and we have a representative  
		\[
		\begin{bmatrix} a & 1 \\ 0 & a \end{bmatrix}.
		\]  
		Each such conjugacy class has size \( (q+1)(q-1) \).
		
		4. Non-diagonalizable matrices with eigenvalues \( \lambda_1, \lambda_2 \) in the quadratic extension \( \mathbb{F}_{q^2} \).  
		Here, \( \lambda_2 \) is uniquely determined by \( \lambda_1 \) and is given by \( F_q(\lambda_1) \), where \( F_q \) is the \( q \)-th power Frobenius automorphism in \( \operatorname{Gal}(\mathbb{F}_{q^2}/\mathbb{F}_q) \).  
		Each set of eigenvalues  \( \{\lambda_1,F_q(\lambda_1 ) \}\) defines a unique conjugacy class of size \( q(q-1) \).

		If \(\ell\mid q - 1 \), then all \( c_i \) are in \( \mathbb{F}_q \), and the conjugacy classes in \( C_q \) correspond to types (1), (2), and (3). We have  
		\begin{align}\label{C_q_l|q-1}
			|C_q|& = \left( (\ell-1)(q-l) + \binom{\ell-1}{2} \right) (q+1)q + (\ell-1) + (\ell-1) (q+1) (q-1)\nonumber \\
			&= (\ell-1)   \left( q^3 - \left(\frac{\ell}{2} -1\right)q^2 - \left(\frac{\ell}{2} +1\right)q \right).
		\end{align}

		Otherwise, if \(\ell\mid q + 1 \), then all \( c_i \) are in \( \mathbb{F}_{q^2} \), and the conjugacy classes in \( C_q \) come from type (4). Each choice of \( c_i \) as one of the eigenvalues uniquely determines the other eigenvalue, and since the order of the two eigenvalues does not matter, we obtain  
		\[
		|C_q| = \frac{\ell-1}{2} q(q-1) .
		\]

		Finally, if \( q=\ell \), then \( c_i = 1 \) for all \( 1 \leq i \leq \ell-1 \). The set \( C_q \) consists of the conjugacy classes of matrices in \( \operatorname{GL}_2(\mathbb{F}_q) \) that have at least one eigenvalue equal to 1.  
		Counting the relevant conjugacy classes from types (1), (2), and (3), we obtain  
		\[
		|C_q| = (q-2) q(q+1) + 1 + (q+1)(q-1) = q(q^2 -2).
		\]

	We now apply Theorems~\ref{theorem_ch1} and~\ref{theorem_ch2} to estimate \( \pi_{C_q}(x, L_q/\mathbb{Q}) \), following the approach of~\cite[pp.~108–109]{wu} and~\cite[p.~345]{steuding}.
	
	\begin{lemma}\label{function_g}
		Let \( E/\mathbb{Q} \) be an elliptic curve without complex multiplication, and let \( L_n = \mathbb{Q}(E[n]) \) denote the \( n \)-division field.  
		Assuming the \( \theta \)-quasi-GRH for \( L_q \), we have
		\[
		\pi_{C_q}(x, L_q/\mathbb{Q}) = \frac{|C_q|}{|G_q|} \operatorname{li}(x) + O_E\left(q^{3/2} x^\theta \log qx\right).
		\]
	\end{lemma}
	
		\begin{proof}
	Let \( B_q \subset G_q \) denote the Borel subgroup of upper triangular matrices, and let \( C_q(B) \) be the union of conjugacy classes generated by \( B_q \cap C_q \). Recall that \( G_q \) is a subgroup of \( \mathrm{GL}_2(\mathbb{F}_q) \), and by \cite{galois_serre}, we have
	\[
	\frac{q^4}{\log \log q} \ll [L_q : \mathbb{Q}] = |G_q| \ll q^4.
	\]
Moreover, by~\cite[Proposition~6]{serre}, the discriminant satisfies
\[
\log d_{L_q} \ll [L_q : \mathbb{Q}] \log\big( [L_q : \mathbb{Q}]\, q N_E \big).
\]
	Applying Theorem~\ref{theorem_ch1}, we obtain
	\begin{equation}\label{wu_1}
		\pi_{C_q}(x, L_q/\mathbb{Q}) = \frac{|B_q|}{|G_q|} \cdot \frac{|C_q|}{|C_q(B)|} \cdot \pi_{C_q(B)}(x, L_q/L_q^{B_q})
		+ O \left( q \log (q N_E) + q x^{1/2} \right).
	\end{equation}

Next, consider the subgroup of unipotent matrices \( U_q \subset B_q \). Since \( U_q \) is normal in \( B_q \) and \( B_q/U_q \) is the abelian group of diagonal matrices over \( \mathbb{F}_q \), the Artin Holomorphy Conjecture (AHC) holds for the extension \( L_q^{U_q}/L_q^{B_q} \). 
Applying Theorem~\ref{theorem_ch2} with \( G = B_q \), \( H = U_q \), and \( D = C_q(B) \), and assuming the \( \theta \)-quasi-GRH for \( L_q \), we obtain
\begin{align}\label{wu_2}
	\pi_{C_q(B)}(x, L_q/L_q^{B_q}) 
	&= \frac{|C_q(B)|}{|B_q|} \operatorname{li}(x) 
	+ O\left( q^{1/2} x^\theta [L_q^{B_q}:\mathbb{Q}] \log(M(L_q/L_q^{B_q}) x) \right) \notag \\
	&= \frac{|C_q(B)|}{|B_q|} \operatorname{li}(x) 
	+ O\left( q^{3/2} x^\theta \log(q N_E x) \right).
\end{align}

The claim then follows by combining equations~\eqref{wu_1} and~\eqref{wu_2}.
		\end{proof}
Therefore, by Lemma~\ref{function_g}, for square-free \( d \) with \( (d, M_E) = 1 \), we have
\[	\mathcal{A}_d 
= i_d\, \pi_{C_d}(x, L_d/\mathbb{Q}) 
+ \sum_{n \mid M_E} \mu(n)\, i_{dn}\, \pi_{C_{dn}}(x, L_{dn}/\mathbb{Q}) 
= c_E\, g(d)\, \operatorname{li}(x) + r(d),\]
and \( \mathcal{A}_d = \emptyset \) when \( (d, M_E) \neq 1 \).
Here, \( g \) is a multiplicative function supported on square-free integers, with \( g(q) = \frac{|C_q|}{|G_q|} \) for \( q \nmid M_E \), and \( g(q) = 0 \) for \( q \mid M_E \). Explicitly,
\[
g(q) =
\begin{cases}
	\displaystyle \frac{\varphi(\ell)\left(q^2 - \left( \frac{\ell}{2} - 1 \right) q - \left( \frac{\ell}{2} + 1 \right)\right)}{(q+1)(q-1)^2}, & \text{if } q \equiv 1 \pmod{\ell},\; q \nmid M_E, \\[10pt]
	
	\displaystyle \frac{i_q \varphi(\ell)}{2(q^2 - 1)}, & \text{if } q \equiv -1 \pmod{\ell},\; q \nmid M_E, \\[10pt]
	
	\displaystyle \frac{\ell^2 - 2}{(\ell+1)(\ell-1)^2}, & \text{if } q = \ell,\; q \nmid M_E, \\[10pt]
	
	0, & \text{if } q \mid M_E,
\end{cases}
\]
and the error term satisfies
\[
r(d) = O_E\left( d^{3/2} x^\theta \log(dx) \right).
\]		
By Mertens' estimate, we have
\[
\sum_{\substack{q < x \\ q \equiv 1 \!\!\!\!\pmod{\ell}\\q\nmid M_E}} \frac{q^2 - \left( \frac{\ell}{2} - 1 \right) q - \left( \frac{\ell}{2} + 1 \right)}{(q+1)(q-1)^2} 
= \frac{1}{\ell - 1} \log \log x + N_E(\ell) + O\left( \frac{1}{\log x} \right),
\]
for some constant \( N_E(\ell) \). Define
\begin{align}\label{def_W}
	W(x) :=& \prod_{q < x} \left(1 - g(q)\right) \nonumber \\
	=&\ (1 - g(\ell)) 
	\prod_{\substack{q < x \\ q \equiv 1 \!\!\!\!\pmod{\ell} \\ q \nmid M_E}} 
	\left( 1 - \frac{\varphi(\ell)\left( q^2 - \left( \frac{\ell}{2} - 1 \right) q - \left( \frac{\ell}{2} + 1 \right) \right)}{(q+1)(q-1)^2} \right) \nonumber \\
	&\quad \times 
	\prod_{\substack{q < x \\ q \equiv -1 \!\!\!\!\pmod{\ell} \\ q \nmid M_E}} 
	\left( 1 - \frac{i_q \varphi(\ell)}{2(q^2 - 1)} \right) \nonumber \\
	\sim&\ \frac{C(\ell)}{\log x},
\end{align}
where
\[	C(\ell) =  (1 - g(\ell))\, e^{-(\ell - 1) N_E(\ell)}
\prod_{\substack{q \equiv -1 \pmod{\ell} \\ q \nmid M_E}} 
\left( 1 - \frac{i_q \varphi(\ell)}{2(q^2 - 1)} \right),\]
and
\[
1 - g(\ell) =
\begin{cases}
	1, & \text{if } \ell \mid M_E, \\[8pt]
	1 - \dfrac{\ell^2 - 2}{(\ell + 1)(\ell - 1)^2}, & \text{if } \ell \nmid M_E.
\end{cases}
\]		
Therefore, the density function \( g  \) satisfies the conditions required for the linear sieve.
To control the error term, we choose the level of distribution \( x^{\alpha} \) so that
\[
\sum_{\substack{d < x^{\alpha} \\ d \mid P(x)}} r(d)
\ll \sum_{d < x^{\alpha}} d^{3/2} x^{\theta} \log(dx)
= o\left( \frac{x}{(\log x)^2} \right).
\]
The admissible range for \( \alpha \) depends on the exponent \( \theta \), and it suffices to require
\begin{equation} \label{alpha_theta}
	\alpha < \frac{2}{5}(1 - \theta).
\end{equation}
Assuming the \( \theta \)-quasi-GRH with \( \theta = \frac{1}{2} \), we may take
\begin{equation} \label{level_of_distribution}
	\alpha = \frac{1}{5} - \epsilon, \quad \text{for some small } \epsilon > 0.
\end{equation}

Let \( 0 < \rho < \delta < \tfrac{1}{2} \) and \( \lambda > 0 \), and define \( z = x^{\rho} \), \( y = x^{\delta} \). Similar to the proof of Proposition~\ref{higher_order}, we define the sieve function
\[
S(x) := |S(\mathcal{A}, z)| - \lambda \sum_{z \leq p < y} w_p\,|S(\mathcal{A}_p, z)|, \quad   w_p := 1 - \frac{\log p}{\log y}.
\]

By \eqref{bound_elements}, each element of \( \mathcal{A} \) satisfies
\(\frac{|E(\mathbb{F}_{p^\ell})|}{|E(\mathbb{F}_p)|} \leq c_\ell\, x^{\ell-1}\)
for some constant \( c_\ell > 0 \). Proceeding as in p.~\pageref{similar_S(x)}, we obtain
\begin{equation}\label{omega_S(x)}
	S(x) \leq \left(1 + \frac{\log c_\ell}{\delta \log x} + \frac{\ell - 1}{\delta} \right) 
	\#\left\{ a \in S(\mathcal{A}, z) : \omega(a) < \frac{1}{\lambda} + \frac{\ell - 1}{\delta} \right\}.
\end{equation}
To replace \( \omega(a) \) with \( \Omega(a) \), we must also bound the contribution from elements \( a \) divisible by \( p^2 \) for some \( z \leq p < y \).

Let \( q\neq\ell \) be a prime such that \(\ell\mid q^2 - 1 \), and recall that \( e_{q,1}, \dots, e_{q,\ell-1} \) denote the roots of the \(\ell\)-th cyclotomic polynomial \( \Phi_\ell \) in \( \mathbb{F}_{q^2} \). By Lemma~\ref{lemma_condition_2}, we have
\[
\#\left\{ p \leq x \,\middle|\, p \nmid q N_E,\; q^2 \mid \frac{|E(\mathbb{F}_{p^\ell})|}{|E(\mathbb{F}_p)|} \right\}
\leq \pi_{C_{q^2}}(x, L_{q^2}/\mathbb{Q}).
\]
Suppose \( q \nmid M_E \). Then \( G_{q^2} \cong \operatorname{GL}_2(\mathbb{Z}/q^2\mathbb{Z}) \). Consider the natural surjection
\[
\operatorname{GL}_2(\mathbb{Z}/q^2\mathbb{Z}) \twoheadrightarrow \operatorname{GL}_2(\mathbb{Z}/q\mathbb{Z}),
\]
whose kernel is isomorphic to \( M_2(\mathbb{Z}/q\mathbb{Z}) \cong \mathbb{F}_q^4 \). It follows that
\[
|G_{q^2}| = q^4 \cdot |\operatorname{GL}_2(\mathbb{Z}/q\mathbb{Z})| = (q + 1) q^5 (q - 1)^2.
\]

	Let \( C_1 \subset G_{q^2} \) be the union of conjugacy classes consisting of elements \( \sigma \in G_{q^2} \) whose characteristic polynomial
	\[
	X^2 - \operatorname{tr}(\sigma) X + \det(\sigma) \pmod q
	\]
	splits  as \( (X - e_{q,i})(X - e_{q,j}) \) for some \( 1 \leq i, j \leq\ell- 1 \), where \( e_{q,i} \in \mathbb{F}_{q^2} \) is a root of \( \Phi_\ell \).
	
If \( \ell \mid q - 1 \), also define \( C_2 \subset G_{q^2} \) to be the union of conjugacy classes consisting of elements \( \sigma \in G_{q^2} \) for which there exists \( 1 \leq k \leq \ell - 1 \) such that
\begin{equation}\label{additional}
	e_{q,k}'^2 - \operatorname{tr}(\sigma)\, e_{q,k}' + \det(\sigma) \equiv 0 \pmod{q^2},
\end{equation}
where \( e_{q,k}' \in \mathbb{Z}/q^2\mathbb{Z} \) is the unique lift of some \( e_{q,k} \in \mathbb{F}_q \).
	
Then \( C_1 \) consists of lifts of those matrices in \( \mathrm{GL}_2(\mathbb{Z}/q\mathbb{Z}) \) whose eigenvalues are precisely \( e_{q,i} \) and \( e_{q,j} \) for some \( 1 \leq i, j \leq \ell - 1 \). 
Let \( D_{q^2} \subset \mathrm{GL}_2(\mathbb{Z}/q\mathbb{Z}) \) denote the union of conjugacy classes consisting of matrices with eigenvalues \( e_{q,i} \) and \( e_{q,j} \) for some \( 1 \leq i, j \leq \ell - 1 \). Then, similar to the counting for \( C_q \), we have
\[
|D_{q^2}| = \binom{\ell - 2}{2} (q+1)q + (\ell - 1) + (\ell - 1)(q+1)(q-1) + \frac{\ell - 1}{2} q(q - 1).
\]
Therefore,
\[
|C_1| = q^4 \cdot |D_{q^2}| = O_\ell(q^6).
\]

The set \( C_2 \subset G_{q^2} \) consists of elements \( \sigma \in G_{q^2} \) that are lifts of elements in \( C_q \subset G_q \) and additionally satisfy equation~\eqref{additional}. Each element of \( C_q \) has \( q^4 \) lifts in \( G_{q^2} \), so the total number of such lifts is \( q^4 |C_q| = O(q^7) \). 
Heuristically, the condition~\eqref{additional} imposes an additional equation that cuts out a hypersurface, reducing the dimension by one when viewed as a quasi-affine variety in \( \mathbb{F}_q^8 \), which yields \(	|C_2| = O(q^6)\).  
We now give an explicit proof in terms of conjugacy classes in \( \operatorname{GL}_2(\mathbb{Z}/q^2\mathbb{Z}) \).

 By \cite[Lemma 2.1]{conjugacy_algebra}, every element \( \sigma \in \operatorname{GL}_2(\mathbb{Z}/q^2\mathbb{Z}) \) can be written in the form
\[
\sigma = dI + q^j \beta,
\]
where \( j \in \{0, 1, 2\} \) is maximal such that \(\sigma\) is congruent to a scalar matrix mod $q^j $   for a unique \( d \in \mathbb{Z}/q^j\mathbb{Z} \), and a unique cyclic matrix \( \beta \in M_2(\mathbb{Z}/q^{2-j}\mathbb{Z}) \). Furthermore, by \cite[Theorem~2.2]{conjugacy_algebra}, each conjugacy class can be represented by a matrix of the form
\[
dI + q^j \begin{pmatrix} 0 & 1 \\ -\det(\beta) & \operatorname{tr}(\beta) \end{pmatrix},
\]
and each such class is uniquely determined by the tuple \( (j, d, \operatorname{tr}(\beta), \det(\beta)) \), where \( d \in( \mathbb{Z}/q^j\mathbb{Z})^\times \) for \( j \geq 1 \), or if \( j = 0 \), then \( \det(\beta) \in (\mathbb{Z}/q^2\mathbb{Z})^\times \). 

 Then \eqref{additional} defines a union of conjugacy classes characterized by the equation
 \[
 e_{q,k }'^2 - (2d + q^j \operatorname{tr}(\beta)) e_{q,k }' + q^{2j} \det(\beta) + d^2 + d q^j \det(\beta) = 0.
 \]
The sizes of conjugacy classes in \( \operatorname{GL}_2(\mathbb{Z}/q^2\mathbb{Z}) \) are described in \cite[Example~4.3.3 and Theorem~4.3.14]{thesis}. We estimate \( |C_2| \) by considering the cases \( j = 0, 1, 2 \) separately, as the parameter \( j \) determines the   size of each conjugacy class:
 \begin{itemize}
	\item When \( j = 0 \), we must have \( d = 0 \). There are \( q(q-1) \) choices for \( \det(\beta) \), and each such choice determines \( \operatorname{tr}(\beta) \). Each conjugacy class with \( j = 0 \) has size \( O(q^4) \), so the total contribution is \( O(q^6) \).
	
	\item When \( j = 1 \), there are \( q - 1 \) choices for \( d \) and \( q^2 \) choices for the pair \( (\det(\beta), \operatorname{tr}(\beta)) \). Each such conjugacy class has size \( O(q^2) \), giving a total contribution of \( O(q^5) \).
	
	\item When \( j = 2 \), there are \( q(q-1)  \) choices for \( d \), and \( q^2 \) choices for \( (\det(\beta), \operatorname{tr}(\beta)) \). Each conjugacy class corresponds to a scalar matrix and has size one, so the total contribution is \( O(q^4) \).
\end{itemize}
Therefore, we conclude that
\[|C_2| = O(q^6) \quad \text{and hence} \quad |C_{q^2}| = |C_1| + |C_2| = O(q^6).\]	
Following the same argument as in Proposition~\ref{function_g}, we obtain
\begin{equation}\label{error_q^2}
	\pi_{C_{q^2}}(x, L_{q^2}/\mathbb{Q}) 
	= \frac{|C_{q^2}|}{|G_{q^2}|} \, \operatorname{li}(x) + O_E\left(q^3 x^{\theta} \log(qx)\right) 
	\ll \frac{1}{q^2} \, \operatorname{li}(x) + q^3 x^{\theta} \log(qx),
\end{equation}
and therefore the total contribution from elements \( a \in S(\mathcal{A}, z) \) divisible by \( q^2 \) for some \( z  < q \leq y \) is bounded by
\[
\#\left\{ p \leq x : p \nmid N_E,\ q^2 \mid  \frac{|E(\mathbb{F}_{p^\ell})|}{|E(\mathbb{F}_p)|},\ z < q \leq y \right\} 
\ll \frac{x}{\log x} \sum_{\substack{z < q \leq y\\ q^2 \equiv 1 \pmod{\ell}}} \frac{1}{q^2} 
+ x^\theta \log x \sum_{z < q \leq y} q^3.
\]
Assuming the \( \theta \)-quasi-GRH with \( \theta = \frac{1}{2} \), the contribution is negligible if we take
\begin{equation}\label{range_delta_Omega}
	y = x^\delta \quad \text{with} \quad \delta < \frac{1}{8},
\end{equation}
so that
\[
\frac{x}{\log x} \sum_{\substack{z < q \leq y \\ q^2 \equiv 1 \!\!\!\! \pmod{\ell}}} \frac{\varphi(\ell)}{q^2} 
+ x^{1/2} \log x \sum_{z < q \leq y} q^3 
\ll _l \frac{x^{1 - \rho}}{\log x} + x^{1/2 + 4\delta} \log x 
= o\left( \frac{x}{(\log x)^2} \right).
\]

We now apply the linear sieve to estimate \( S(x) \). Arguing as in Proposition~\ref{higher_order}, we obtain, for sufficiently large \( x \),
\[
S(x) \geq c_E\, W(z)\, \operatorname{li}(x) \cdot G' (\rho, \delta, \lambda),
\]
where \( W(z) \) is defined in~\eqref{def_W}, and
\[
G'(\rho, \delta, \lambda) = f\left(\frac{\alpha}{\rho}\right) 
- \lambda \int_{\rho}^{\delta} F\left(\frac{\alpha - t}{\rho}\right) 
\left(1 - \frac{t}{\delta}\right) \frac{dt}{t}, \quad \alpha = \frac{1}{5} - \epsilon ,
\]
where \( \alpha \) is the level of distribution, as given in~\eqref{level_of_distribution}, and \( f \), \( F \) denote the lower and upper bound functions of the linear sieve, with explicit formulas given in~\eqref{f} and~\eqref{F}.

A positive lower bound follows as long as \( c_E \neq 0 \) and \( G '(\rho, \delta, \lambda) > 0 \).
To minimize  
\[
\frac{1}{\lambda} + \frac{\ell - 1}{\delta},
\]  
we choose suitable values for \( \rho \), \( \delta \), and \( \lambda \), subject to the constraint \( \delta > \frac{1}{5} \), ensuring that the integrand variable \( \frac{\alpha}{\rho} - \frac{t}{\rho} \) in \( F \) remains positive. Moreover, the arguments of \( f \) and \( F \) must lie within the domains where explicit formulas are valid (see~\eqref{f} and~\eqref{F}):  
\[
0 < s \leq 3 \quad \text{for } F(s), \qquad 2 < s \leq 4 \quad \text{for } f(s).
\]

For the bound involving the number of distinct prime divisors \( \omega(a) \), we take \( \delta = \frac{1}{5.1} \) and compute  
\[
G' \left( \tfrac{1}{20}, \tfrac{1}{5.1}, \tfrac{1}{1.2} \right) = 0.0818,
\]  
which yields the estimate  
\[
\omega(a) \leq \lfloor 5.1(\ell - 1) + 1.2\rfloor
\]  
in~\eqref{omega_S(x)}.

To extend this to the total number of prime divisors \( \Omega(a) \), we require \( \delta < \frac{1}{8} \) (see~\eqref{range_delta_Omega}). Taking \( \delta = \frac{1}{8.1} \), we compute
\[
G' \left( \tfrac{1}{20}, \tfrac{1}{8.1}, \tfrac{1}{0.5} \right) = 0.1114,
\]
which gives the bound
\[
\Omega(a) \leq \lfloor8.1(\ell - 1) + 0.5\rfloor
\]
in~\eqref{omega_S(x)}, with \( \omega(a) \) replaced by \( \Omega(a) \).

Finally, suppose \( E \) is a Serre curve. Then for any \( n \), the index of \( G_n \) in \( \mathrm{GL}_2(\mathbb{Z}/n\mathbb{Z}) \) is at most \( 2 \). From earlier, when \( G_q \cong \mathrm{GL}_2(\mathbb{Z}/q\mathbb{Z}) \), we have
\(\frac{|C_q|}{|\mathrm{GL}_2(\mathbb{F}_q)|} = O\left(\frac{1}{q}\right),\)
which implies that for any square-free \( n \),  \( \frac{|C_n|}{|G_n|} \) is strictly less than \( 1 \). Hence, \( c_E \neq 0 \). By~\cite[Theorem~4]{serre_curves}, almost all elliptic curves are Serre curves, so the result holds for almost all non-CM elliptic curves.
	\end{proof}
\section{A Final Remark}
While this paper focuses on almost prime orders related to Koblitz's problems—namely, \( |E(\mathbb{F}_{p^\ell})| \) for prime \( \ell \)—we now outline the general approach for \( |E(\mathbb{F}_{p^d})| \) with arbitrary \( d \) in the CM case.

Let \( E/\mathbb{Q} \) be a CM elliptic curve with conductor \( N_E \) and CM by the ring of integer  \( \mathcal{O}_K   \) of \(K=  \mathbb{Q}(\sqrt{-D}) \). For primes \( p \nmid  N_E \), we have
\[
|E(\mathbb{F}_{p^d})| = N(\pi_p^d - 1),
\]
where \( \pi_p \in \mathcal{O} _K\) satisfies \(\pi_p^2 -a_p\pi_p +p=0   \). 
Define the element $\pi_E$  by
\[
\pi _E := \gcd\left\{\pi_p^d-1  : p \nmid  N_E,\, p \text{ splits in } K,\, \pi_p^2 -a_p\pi_p +p=0  \right\}.
\]
This integer reflects the factorization of \( \pi_p^d - 1 \) in \( \mathcal{O}_K \). In place of the equation \(\pi_p^2 -a_p\pi_p +p=0 \), the Frobenius element $\pi_p$ can be determined by a congruence condition depending on the level of the associated Grossencharacter. 
We need to first divide   $|E(\mathbb{F}_{p^d})|$ by $N(\pi_d)$ and then apply sieves.  
Indeed, \[
N(\pi_E) = d_E:=  \gcd\left\{ |E(\mathbb{F}_{p^d})|  : p \nmid  N_E,\, p \text{ splits in } K \right\}.
\]

 Applying a \( \tau(d) \)-dimensional sieve to the set
\[
\mathcal{A}(x) := \left\{ \frac{|E(\mathbb{F}_{p^d})|}{N(\pi_E)} =N\left(\frac{\pi_p^d-1}{\pi_E }\right) : p \nmid N_E,\; p \text{ splits in } K,\;  p \leq x \right\}
\]
yields a lower bound \( \gg \frac{x}{(\log x)^{\tau(d) + 1}} \) for the number of primes \( p \) such that all prime divisors of \( \frac{|E(\mathbb{F}_{p^d})|}{N(\pi_E)} \) exceed \( x^\alpha \) for some \( \alpha > 0 \).

%On the other hand, fix a number field \( K \) of  degree \( d \). Then for any \( n \) coprime to \( p \) and \(p\nmid N_E\), we have an embedding
%\[
%E(K)_{\mathrm{tors}}[n] \hookrightarrow E(\mathbb{F}_{p^m}) \quad \text{for some } m \mid d,
%\]
%so that
%\[
%|E(K)_{\mathrm{tors}}[n]| \mid \gcd\left\{ |E(\mathbb{F}_{p^d})| : p \nmid  N_E,\, p \nmid n \right\}.
%\]
%Fix a parameter \( y = o(x) \), and note that
%\[
%\left\lvert \bigcup_{p < y} E(K)_{\mathrm{tors}}[p^\infty] \right\rvert 
%\Biggm| 
%\gcd\left\{ |E(\mathbb{F}_{p^d})| : p > y,\; p \nmid  N_E \right\}.
%\]
%The sieve result gives a lower bound \( \gg \frac{x}{(\log x)^{\tau(d)+1}} \) for the number of primes \( p \leq x \) such that all prime divisors of \( |E(\mathbb{F}_{p^d})|/d_E \) exceed \( x^\alpha \). In particular, this implies
%\[
%\gcd\left\{ |E(\mathbb{F}_{p^d})| : p > y,\; p \nmid N_E \right\} = d_E,
%\]
%and hence,
%\[
%|E(K)_{\mathrm{tors}}| \leq d_E .
%\]

Any prime \( \ell \mid N(\pi_E) \) must satisfy:
\begin{itemize}
	\item if \( \ell \) splits or ramifies in \( K \), then \( \ell - 1 \mid d \);
	\item otherwise, \( \ell^2 - 1 \mid d \).
\end{itemize}
Suppose \( d \) is odd, only \( \ell = 2 \) is possible, so \(N(\pi_E)\) must be a power of \( 2 \).

For example, take the CM curve \( E: y^2 = x^3 - x \), with CM by \( \mathbb{Z}[i] \). Here, \( (1+i)^3 \mid \pi_p - 1 \), but \( 1+i \nmid \sum_{j=0}^{d-1} \pi_p^j \), so for odd \( d \), we   take $d_E=8$ and apply a \( \tau(d) \)-dimensional sieve to
\[
\left\{ \frac{|E(\mathbb{F}_{p^d})|}{8} :   p \leq x,\; p \nmid  N_E,\; p \text{ splits in } \mathbb{Z}[i] \right\}.
\]

For \( d = 2 \), since $(1+i)^3\mid \pi_p-1$ and $(1+i)^2\mid  \pi_p-1$, we take $d_E=32$ and apply a 2-dimensional sieve to
\[
\left\{ \frac{|E(\mathbb{F}_{p^2})|}{32} :   p \leq x,\; p \nmid  N_E,\; p \text{ splits in } \mathbb{Z}[i] \right\} .
\] 
 
%For even degrees \( d > 2 \), \( d_E = N(\pi_E) \) is less constrained and more difficult to determine.    For uniform bounds of \(|E(K)_{\mathrm{tors}}|\), see~\cite{torsion}.

\section*{Acknowledgment}
The author acknowledges support from the Dr. Lois M. Lackner Mathematics Fellowship at the University of Illinois at Urbana-Champaign during the course of this research.

	\bibliographystyle{alpha}

	\bibliographystyle{plain}

\end{document}